\documentclass{amsart}
\usepackage[margin=1in]{geometry}

\usepackage{bm,booktabs,comment,mathtools,multicol,tikz-cd,amssymb,stmaryrd}
\usetikzlibrary{decorations.pathmorphing}

\usepackage[inline]{enumitem}

\usepackage%
[%
bookmarksdepth=subsubsection,
pdfauthor={Peter Dillery and David Schwein},
pdftitle={A stacky generalized Springer correspondence and rigid enhancements of L-parameters},
unicode]%
{hyperref}

\usepackage[capitalize]{cleveref}
\usepackage[msc-links]{amsrefs} %

\numberwithin{equation}{subsection}
	\newtheorem{corollary}[equation]{Corollary}
	\newtheorem{conjecture}[equation]{Conjecture}
	\Crefname{conjecture}{Conjecture}{Conjectures}
	\newtheorem{fact}[equation]{Fact}
	
	\newtheorem{lemma}[equation]{Lemma}
	\newtheorem{proposition}[equation]{Proposition}
	
	\newtheorem{theorem}[equation]{Theorem}

\theoremstyle{definition}
	\newtheorem{construction}[equation]{Construction}
	\newtheorem{definition}[equation]{Definition}

\theoremstyle{remark}
	\newtheorem{example}[equation]{Example}
	\newtheorem{remark}[equation]{Remark}
	\newtheorem{warning}[equation]{Warning}

\DeclareMathOperator\Ad{Ad}
\DeclareMathOperator\Aut{Aut}

\DeclareMathOperator\Dbc{D^b_c}
\DeclareMathOperator\End{End}

\DeclareMathOperator\Four{Four}
\DeclareMathOperator\Gal{Gal}
\DeclareMathOperator\Hom{Hom}
\DeclareMathOperator\IC{IC}
\DeclareMathOperator\Ind{Ind}
\DeclareMathOperator\ind{ind}
\DeclareMathOperator\uInd{\underline{Ind}}
\DeclareMathOperator\uind{\underline{ind}}
\DeclareMathOperator\Infl{Infl}
\DeclareMathOperator\infl{infl}
\DeclareMathOperator\Irr{Irr}
\DeclareMathOperator\Isom{Isom}
\DeclareMathOperator\Lev{Lev}
\DeclareMathOperator\Lie{Lie}
\DeclareMathOperator\Loc{Loc}
\DeclareMathOperator\Out{Out}
\DeclareMathOperator\Par{Par}
\DeclareMathOperator\Perv{Perv}
\DeclareMathOperator\rank{rank}
\DeclareMathOperator\Res{Res}
\DeclareMathOperator\res{res}
\DeclareMathOperator\uRes{\underline{Res}}
\DeclareMathOperator\ures{\underline{res}}
\newcommand\cores{{}^\backprime\!\res}
\DeclareMathOperator\Ru{R_u}

\DeclareMathOperator\val{val}

\DeclareMathOperator\Cusp{Cusp}
\DeclareMathOperator\uCusp{Cusp^\prime}
\DeclareMathOperator\Sc{Sc}
\DeclareMathOperator\qcsupp{qCusp}
\DeclareMathOperator\connsupp{Cusp^\circ}

\DeclareMathOperator\GL{GL}
\DeclareMathOperator\SL{SL}
\DeclareMathOperator\Sp{Sp}
\DeclareMathOperator\Spin{Spin}
\DeclareMathOperator\SO{SO}

\newcommand\bbA{\mathbb{A}}
\newcommand\bbC{\mathbb{C}}
\newcommand\bbG{\mathbb{G}}

\newcommand\bbP{\mathbb{P}}
\newcommand\bbQ{\mathbb{Q}}
\newcommand\bbZ{\mathbb{Z}}
\newcommand{\ov}{\overline}

\newcommand\actson\curvearrowright

\newcommand\defeq{:=}
\newcommand\eqdef{=:}
\renewcommand\sslash{/\!\!/}

\newcommand\adq{/\!_\textnormal{ad}}

\newcommand\mathdef{\textsf}
\newcommand\tn{\textnormal}
\renewcommand\frak{\mathfrak}
\providecommand\cal{}
\renewcommand\cal{\mathcal}

\newcommand\rig{\textnormal{rig}}
\newcommand\rec{\textnormal{rec}}

\newcommand\smat[4]{\big[\begin{smallmatrix}#1&#2\\#3&#4\end{smallmatrix}\big]}

\newenvironment{gloss}%
{\renewcommand{\descriptionlabel}%
[1]{\hspace{\labelsep}##1:}
\begin{description}
\setlength\itemsep{1em}}
{\end{description}}

\title[A stacky Springer correspondence and rigid enhancements]%
{A stacky generalized Springer correspondence and\\ rigid enhancements of $L$-parameters}
\date{10 June 2024}

\author{Peter Dillery}
\thanks{The first author was supported by the Brin Postdoctoral Fellowship at the University of Maryland, College Park}
  \address{Department of Mathematics,
  University of Maryland, 4176 Campus Drive,
  College Park, MD 20742-4015, USA}

\author{David Schwein}
  \address{Mathematics Institute,
  University of Bonn, Endenicher Allee 60,
  53115 Bonn, Germany}

\begin{document}

\begin{abstract}
Motivated by applications to the Langlands program,
Aubert--Moussaoui--Solleveld extended Lusztig's
generalized Springer correspondence to disconnected reductive groups.
We use stacks to give a more geometric account of their theory,
in particular, formulating a truly geometric version
of the (relevant analogue of the)
Bernstein--Zelevinsky Geometrical Lemma
and explaining how to compare the correspondence
on the group and the Lie algebra using quasi-logarithms.
As an application, we study Kaletha's rigid enhancements of
$L$-parameters and draw the same conclusions
as Aubert--Moussaoui--Solleveld for this enhancement:
there exists a cuspidal support map and its fibers are parameterized
by irreducible representations of twisted group algebras.
\end{abstract}

\maketitle
\setcounter{tocdepth}{1}
\tableofcontents

\section{Introduction}

\subsection{The generalized Springer correspondence}
Partitions of a natural number~$n$
correspond to two very different objects of interest in representation theory.
On the one hand, partitions parameterize conjugacy classes of the symmetric group~$S_n$.
Using a noncanonical bijection of conjugacy classes
and irreducible representations of~$S_n$,
which exists by character theory of finite groups,
we obtain a parameterization of irreducible representations.
On the other hand, via Jordan block matrices, partitions parameterize
unipotent conjugacy classes of the general linear group $\GL_n(\bbC)$.
All in all, one has a noncanonical bijection between
the unipotent conjugacy classes of~$\GL_n(\bbC)$
and irreducible representations of its Weyl group.

One would hope that this bijection extends to an arbitrary reductive group~$G$,
with the Weyl group of $G$ playing the role of~$S_n$,
and that such a bijection could be made canonical by some natural construction.
A naive generalization is impossible, however:
for example, the group $E_8$ has $70$ unipotent orbits
but its Weyl group $W(E_8)$ has $112$ irreducible representations.
Nonetheless, in 1974, Springer \cite{springer74}
explained how to generalize the bijection using geometry:
he defined for every unipotent conjugacy class a certain variety now called a Springer fiber,
constructed an action of the Weyl group on the cohomology groups of the Springer fibers,
and showed that every irreducible representation of the Weyl group is realized in this way.

Eight years later, Lusztig \cite{lusztig84b}
both clarified and extended Springer's work using perverse $\ell$-adic sheaves.
Let $G$ be a connected (complex) reductive group
and let $\cal N_G$ be the \mathdef{nilpotent cone} of~$G$,
the variety of nilpotent elements of the Lie algebra~$\frak g$.
As we recall in \Cref{sec:param:springer},
the variety $\cal U_G$ of unipotent conjugacy classes
is $G$-equivariantly isomorphic to~$\cal N_G$,
so there is no harm in passing to the Lie algebra here.
For brevity, we summarize Lusztig's ideas
in the more recent formulation of Achar--Henderson--Juteau--Riche
\cite{achar_henderson_juteau_riche19};
see \Cref{sec:param:comparison}
for a comparison with Lusztig's earlier definitions.

The goal of Lusztig's generalized Springer correspondence
is to classify the simple objects in the category
\[
\Perv(\cal N_G/G)
\]
of $G$-equivariant \mathdef{perverse sheaves}
on $\cal N_G$ with coefficients in $\overline\bbQ_\ell$.
General facts about equivariant perverse sheaves
already give a coarse classification.
Specifically, given a unipotent conjugacy class $\cal O$
and an irreducible $G$-equivariant perverse sheaf
$\cal E$ on~$\cal O$, there is an associated
\mathdef{intersection cohomology} sheaf%
\footnote{In practice, all our intersection cohomology sheaves
will be equivariant and it would be more correct to write
$\IC(\cal O/G,\cal E)$ instead of $\IC(\cal O,\cal E)$.
To shorten notation we will use the latter notation
and leave the group action implicit.}
$\IC(\cal O,\cal E)\in\Irr\bigl(\Perv(\cal N_G/G)\bigr)$,
and every element of $\Irr\bigl(\Perv(\cal N_G/G)\bigr)$
is uniquely of this form.
In other words,
\[
\Irr\bigl(\Perv(\cal N_G/G)\bigr)
\simeq \bigsqcup_{\cal O} \Irr\bigl(\Perv(\cal O/G)\bigr).
\]
Moreover, because $\cal O$ is a homogeneous space for~$G$,
the simple objects of $\Perv(\cal O/G)$
are (shifts of) local systems, which are easy to describe:
fixing a basepoint $u\in\cal O$,
the stalk $j_u^*\cal E$ of $\cal E$ at $u$
carries a natural action of the group
$A_G(u) \defeq \pi_0\bigl(Z_G(u)\bigr)$,
and formation of stalks defines a $\overline\bbQ_\ell$-linear
equivalence of categories
\[
j_u^*\colon\Perv(\cal O/G) \simeq \tn{Mod}\bigl(A_G(u)\bigr).
\]
All in all, letting $[u]_G$ denote the $G$-orbit
of $u\in\cal N_G$, we find from general principles a bijection
\[
\Irr\bigl(\Perv(\cal N_G/G)\bigr)
\simeq \bigsqcup_{[u]_G\in\cal N_G/G} \Irr\bigl(A_G(u)\bigr).
\]

The generalized Springer correspondence
refines this description of $\Irr\bigl(\Perv(\cal N_G/G)\bigr)$
using parabolic induction and restriction functors,
like what one sees in the representation theory of reductive $p$-adic groups.
Let $L$ be a Levi subgroup of~$G$ and let $P$ be a parabolic subgroup of~$G$
with Levi quotient~$L$.
Associated to $(G,P,L)$ there is an adjoint pair of functors
\[
\begin{tikzcd}[column sep=large]
\res_{L\subseteq P}^G : \Perv(\cal N_G/G) \rar[shift left=0.5ex] &
\Perv(\cal N_L/L) : \ind_{L\subseteq P}^G \lar[shift left=0.5ex]
\end{tikzcd}
\]
called \mathdef{parabolic restriction} and
\mathdef{parabolic induction}.
Call a perverse sheaf $\cal F\in\Perv(\cal N_G/G)$
\mathdef{cuspidal} if $\cal F$ is irreducible
and $\res_{L\subseteq P}^G(\cal F) = 0$
for every proper parabolic subgroup~$P$ of~$G$.

Although the definition of cuspidality is quite sophisticated,
it is rare for a sheaf to be cuspidal and
one can combinatorially classify cuspidal sheaves,
a classification we summarize in \Cref{sec:tables}.
Parabolic induction then extends the classification to all sheaves.
Say a \mathdef{cuspidal datum} for~$G$ is a triple $(L,\cal O,\cal E)$
consisting of a Levi subgroup $L$ of~$G$,
a nilpotent orbit~$\cal O$ of~$L$,
and a cuspidal $L$-equivariant local system~$\cal E$ on~$\cal O$.
The \mathdef{induction series} 
for a cuspidal datum $(L,\cal O,\cal E)$ is the set
\[
\Irr_{(L,\cal O,\cal E)}\bigl(\Perv(\cal N_G/G)\bigr)
\subseteq \Perv(\cal N_G/G)
\]
of $\cal F\in\Irr\bigl(\Perv(\cal N_G/G)\bigr)$
such that $\cal F$ is a subobject
(equivalently, summand) of $\ind_{L\subseteq P}^G(\IC(\cal O,\cal E))$
for some parabolic subgroup of~$G$ with Levi factor~$L$.
In other words, we say the conjugacy class $[L,\cal O,\cal E]_G$
lies in the \mathdef{cuspidal support} of~$\cal F$.

With the basic definitions explained,
we can outline the correspondence.
There are two steps.
\begin{enumerate}
\item
(Construction of induction series.)
The induction series partition $\Irr\bigl(\Perv(\cal N_G/G)\bigr)$;
in other words, the cuspidal support exists and is unique.

\item
(Parameterization of induction series.)
The induction series for $[L,\cal O,\cal E]$
is parameterized by the irreducible representations
of the relative Weyl group $N_G(L)/L$.
\end{enumerate}

Before elaborating on these two steps,
we should mention the final ingredient
in the generalized Springer correspondence
under discussion in this paper:
allowing the reductive group~$G$ to be disconnected.
Motivated by applications to the Langlands correspondence,
to which we will return,
Aubert--Moussaoui--Solleveld \cite{AMS18}
generalized Lusztig's correspondence to disconnected~$G$
and their so-called quasi-Levi subgroups.
The argument proceeds by reduction to the connected case
and concludes that the induction series for $[L,\cal O,q\cal E]$
is parameterized by irreducible representations not of the group
$W=N_G(L)/L$, but rather, of the twisted group algebra
\[
\bbC[W_{(\cal O,q\cal E)},\kappa_{(\cal O,q\cal E)}]
\]
where $W_{(\cal O,q\cal E)}$ is the stabilizer of $(\cal O,q\cal E)$ in~$W$
and $\kappa_{(\cal O,q\cal E)}\in Z^2(W_{(\cal O,q\cal E)},\bbC^\times)$
is a certain subtle cocycle.

The first main goal of this paper is to redevelop
the disconnected generalized Springer correspondence
of Aubert--Moussaoui--Solleveld
by directly working with disconnected reductive groups,
rather than reducing to the known connected case using cocycles.
This redevelopment consists of implementing the two steps above,
the construction and parameterization of induction series,
and both steps require new ideas.
Let us briefly outline each.

The first step, the construction of induction series,
is carried out in \Cref{sec:springer}.
Existence of the cuspidal support is a formality but uniqueness,
in other words, disjointness of induction series, is not at all automatic.
Proving it requires a generalization of Mackey's theorem known,
at least in the setting of reductive $p$-adic groups,
as the Geometrical Lemma of Bernstein--Zelevinsky,
which shows that the parabolic restriction of a parabolic induction
is filtered by parabolic inductions of parabolic restrictions;
we refer the reader to \Cref{thm17} for a more precise statement.
When $G$ is connected, the Geometrical Lemma for $\Perv(\cal N_G/G)$
is a theorem of Achar--Henderson--Juteau--Riche
\cite[Lemma~2.1]{achar_henderson_juteau_riche17b},
building on earlier work of Lusztig and Mars.

In trying to generalize the proof of the Geometrical Lemma
to disconnected reductive groups,
one encounters the following well-known technical problem:
when $G$ is connected, a $G$-equivariant structure on a perverse sheaf is unique if it exists,
but when $G$ is disconnected, a $G$-equivariant structure is truly an additional datum.
Naively generalizing the proof to the disconnected case
would require carefully keeping track of equivariant structures,
a difficult bookkeeping problem.

Fortunately, there is a different way of organizing equivariant structures:
algebraic stacks, specifically, quotient stacks.
Rewriting the proof of the geometrical in stack language
is extremely natural because it illuminates many
change-of-coordinate maneuvers in the earlier proofs,
as well as the definitions of parabolic induction and restriction themselves,
which become pull-push functors along a certain correspondence
(in the sense of algebraic geometry, in other words,
a span in the sense of category theory).
The full statement requires a bit of set-up,
but to give the reader a flavor of the result we preview it below.
Let $P$ and~$Q$ be parabolic subgroups of~$G$ with Levi factors~$M$ and~$L$.
There is an algebraic stack $\frak X$
which encodes the composition $\res_{M\subseteq Q}^G\circ\ind_{L\subseteq P}^G$
and sits in the composition of the correspondences defining these functors.
This stack is further filtered by double cosets $w\in W_M\backslash W/W_L$,
and for each of these one has an algebraic stack $\frak Y_w$
encoding a certain parabolic restriction followed by a certain parabolic induction.

\begin{theorem}[\Cref{thm16}]
Let $\cal N_w \defeq \cal N_{Q\cap{}^wL}
\underset{\cal N_{M\cap{}^wL}}{\times}
\cal N_{M\cap{}^wP}$ and
$G_w \defeq (Q\cap{}^wL)\underset{M\cap{}^wL}{\times}M\cap{}^wP$.
\begin{enumerate}
\item
There are isomorphisms
$\frak X_w \simeq \cal N_{Q\cap{}^wP}/(Q\cap{}^wP)$
and $\frak Y_w \simeq \cal N_w/G_w$.

\item
The composition
$f_w\colon\frak X_w \simeq \cal N_{Q\cap{}^wP}/(Q\cap{}^wP)
\longrightarrow\cal N_w/G_w\simeq\frak Y_w$
fits into a commutative diagram
\[
\begin{tikzcd}
\cal N_L/L &
\frak X_w \lar \dar{f_w} \drar \\
\cal N_{{}^wL}/{}^wL \uar{\tn{Ad}(w^{-1})} &
\frak Y_w \lar \rar &
\cal N_M/M.
\end{tikzcd}
\]
\end{enumerate}
\end{theorem}

The Geometrical Lemma for perverse sheaves
then follows by a standard open--closed triangles argument,
using the fact that $f_w$ has ``contractible fibers''.
Our formulation is completely geometric
in the sense that it has completely removed all mention of sheaf theory,
and can thus be recycled for other six-functor formalisms.
For example, turning to coherent sheaves,
we expect \Cref{thm16} to  assist in extending Bezrukavnikov's perverse-coherent
Springer correspondence \cite{bezrukavnikov03} to disconnected reductive groups.

To conclude our discussion of the construction of induction series,
we note that \Cref{thm16},
and more generally our treatment of the cuspidal support map,
has no restrictions on the parabolic subgroups,
unlike the work of Aubert--Moussaoui--Solleveld:
for us, a parabolic subgroup of~$G$
is simply a subgroup whose neutral component is parabolic.
In general one must record the family $\cal P$
of parabolic subgroups with which one carries out parabolic induction and restriction,
leading to a formalism of \mathdef{$\cal P$-cuspidal support}
that we explain in \Cref{sec:springer:cuspidal}.

In \Cref{sec:param} we turn to the second step,
parameterization of induction series.
At the beginning of the analysis
we can more or less follow the classical strategy from the connected case.
When $G$ is connected,
the definitions of parabolic induction
and restriction extend to the entire Lie algebra,
and if one further restricts to a certain regular locus,
depending on the induction series to parameterize,
one finds that the correspondence defining parabolic induction
becomes a finite étale covering.
The endomorphism algebra of this regular parabolic induction
is relatively easy to compute,
and can naturally be interpreted as the endomorphisms
of a certain induced representation of a finite group.
The Fourier--Laumon transform then exchanges parabolic induction on the nilpotent cone
and on this regular locus, roughly speaking,
which shows that the two parabolic inductions have the same endomorphisms.

For the most part, this argument for parameterization
of induction series extends to disconnected groups.
The main difficulty is to compute the endomorphism algebra
of the parabolic induction on the regular locus.
Here one must be careful with the class $\cal P$ of parabolic subgroups:
without additional hypotheses on~$\cal P$,
we do not expect that the $\cal P$-cuspidal induction series will be parameterized
by irreducible representations of twisted groups algebras,
as we argue in \Cref{thm78}.
Consequently, and for later applications,
we mostly focus on the quasi-Levi subgroups of Aubert--Moussaoui--Solleveld.
At this point, it would seem that we are done because
the endomorphism algebras of the relevant parabolic inductions
are computed in their earlier work.

However, there is one subtlety:
on the one hand, the argument above must take place on the Lie algebra,
because the Fourier--Laumon transform operates on sheaves on vector spaces,
but on the other hand, the work of Aubert--Moussaoui--Solleveld takes place on the group.
It seems possible that their work extends to the group setting
but we did not closely investigate this possibility.
Instead, we develop a comparison of parabolic induction on the group and the Lie algebra
using the quasi-logarithms of Bardsley--Richardson \cite{bardsley_richardson85}.
The comparison is not immediate:
although the nilpotent cone and unipotent variety are equivariantly isomorphic,
the same cannot be said of the relevant regular loci on the group and the Lie algebra.
The following result resolves this discrepancy.
Its proof passes through several lemmas of independent interest
on the étale fundamental group of an algebraic stack and its connection to local systems.

\begin{theorem}[\Cref{thm84}]
Let $\lambda\colon G\to\frak g$ be a quasi-logarithm.
Let $\cal O$ be a unipotent orbit of~$L$, let $u\in\cal O$,
and let $\cal E\in\Loc(\lambda(\cal O)/L)$ be irreducible. 
Suppose that $\pi_0(L)$ is a normal subgroup of $\pi_0(N_G(L^\circ,\cal O))$.
Then the map
\[
\lambda^*\colon
\End\bigl(\uind_{L\subseteq P}^G(\infl_{L\tn{-reg}}(\cal E))\bigr) \to
\End\bigl(\uInd_{L\subseteq P}^G(\Infl_{L\tn{-reg}}(\lambda^*\cal E))\bigr)
\]
is an isomorphism.
\end{theorem}

All in all, we cannot say for sure that our quasi-cuspidal support map
agrees with that of Aubert--Moussaoui--Solleveld,
though our connected cuspidal support maps agree;
see \Cref{thm96} and \Cref{sec:param:comparison} for more discussion.
Nonetheless, these results are strong enough to show that with our definitions
the induction series for quasi-Levi subgroups are parameterized by 
irreducible representations of twisted group algebras.

\subsection{Supercuspidal support for $L$-parameters}
The second half of our paper applies the results of the first
to study the relationship between supercuspidal support
and the refined local Langlands conjectures.

Let $F$ be a nonarchimedean local field and
$G$ a connected quasi-split reductive $F$-group.
As was the case for $\Perv(\cal N_G/G)$, the smooth dual $\Pi(G)$ of~$G(F)$
can be organized using (normalized) parabolic induction and restriction.
Call a representation~$\pi\in\Pi(G)$
\mathdef{supercuspidal} if $\res_{L\subseteq P}^G(\pi) = 0$
for every proper parabolic subgroup $P\subset G$.
It follows by adjunction that every $\pi\in\Pi(G)$ 
arises as a subrepresentation of a parabolic induction
$\ind_{L\subseteq P}^G(\sigma)$ with $\sigma$ supercuspidal,
and by the Bernstein--Zelevinsky Geometrical Lemma,
the $G(F)$-conjugacy class of the pair $(L,\sigma)$ is unique.
We call such a pair $(L,\sigma)$ a \mathdef{cuspidal pair} for~$G$
and its conjugacy class $[L,\sigma]_G$
the \mathdef{supercuspidal support} $\Sc(\pi)$ of~$\pi$.
Writing $\Omega(G)$ for the set of $G(F)$-conjugacy classes of cuspidal pairs,
the so-called \mathdef{Bernstein variety},
supercuspidal support is a map
\[
\Sc \colon \Pi(G) \longrightarrow \Omega(G).
\]
Organizing $\Pi(G)$ by the fibers over the connected components of~$\Omega(G)$
yields the \mathdef{Bernstein decomposition} of~$\Pi(G)$.

As for the local Langlands conjectures and their refinements,
in general, a refined local Langlands correspondence
enhances the classical notion of $L$-parameter
and collects together representations of groups related to~$G$
into compound $L$-packets.
We know of at least four types of enhancement,
and in this paper, two of them appear:
the Arthur--Vogan enhancement using pure inner forms \cite{vogan93,arthur06},
which appeared in the earlier work of Aubert--Moussaoui--Solleveld,
and the rigid enhancements of Kaletha and Dillery \cite{Kal16,Dillery}.
Let us review the second, our main focus.

Kaletha's enhancement of $L$-parameters rests
on the notion of a \mathdef{rigid inner twist}.
Roughly speaking, a rigid inner twist is a way of equipping an inner twist
$(G', \xi)$ of $G$ with an extra piece of ``rigidifying data''
(namely, a certain torsor $\mathcal{T}$) in order to shrink the morphisms of inner twists
to only allow morphisms which preserve representation-theoretic information.
Fixing a sufficiently large finite central subgroup $Z \subseteq G$
(for example, $Z(G_\tn{der})$, the center of the derived subgroup of $G$)
simplifies the discussion of rigid inner twists, and we do so in this article. 
On the Galois side, the subgroup~$Z$ intervenes through the following construction:
given a subgroup $K$ of~$\widehat G$, let
\[
K^+ \defeq \widehat{G/Z} \times_{\widehat G} K
\]
be the preimage of~$K$ under the isogeny $\widehat{G/Z}\to\widehat G$.
In particular, $\widehat G^+ = \widehat{G/Z}$.

For an $L$-parameter $\varphi\colon W_F\times\SL_2\to{}^LG$,
the local Langlands conjectures predict the existence of a finite subset
$\Pi_{\varphi}$ of irreducible representations of all rigid inner twists of $G$.
They further predict the existence,
after fixing a Whittaker datum $\mathfrak{w}$ for $G$,
of a canonical bijection 
\[
\rec_{\frak w}\colon \Pi_{\varphi} \longrightarrow
\Irr\bigl(\pi_{0}\bigl(Z_{\widehat G^+}(\varphi)\bigr)\bigr)
\]
between $\Pi_\varphi$ and the set of isomorphism classes of irreducible
complex representations of the finite group $\pi_0\bigl(Z_{\widehat G^+}(\varphi)\bigr)$.
This bijection is expected to satisfy many compatibility properties, cf.~\cite[\S 7.3]{Dillery}.
In particular, the restriction map
\begin{equation}\label{centralchar}
\Irr\bigl(\pi_{0}\bigl(Z_{\widehat G^+}(\varphi)\bigr)\bigr)
\to \bigl(Z(\widehat{G})^{\Gamma,+}\bigr)^{*}
\end{equation}
whose target is identified, via an analogue of Tate-Nakayama duality,
with the set of isomorphism classes of rigid inner twists of~$G$,
is expected to identify the twist~$G'$
whose representation corresponds to
a given irreducible representation of $\pi_{0}\bigl(Z_{\widehat G^+}(\varphi)\bigr)$.
With the parameterization $\varphi\mapsto\Pi_\varphi$ in hand,
it is natural to define an \mathdef{enhanced $L$-parameter} $(\varphi, \rho)$
to consist of an $L$-parameter $\varphi$ and representation
$\rho \in \Irr\bigl(\pi_{0}\bigl(Z_{\widehat G^+}(\varphi)\bigr)\bigr)$.
One can extract from $\varphi$ the pair 
\[
(H,u_\varphi) \defeq \Bigl(Z_{\widehat G^+}(\varphi |_{W_F}),\;
\varphi\bigl(1,[\begin{smallmatrix}1&1\\0&1\end{smallmatrix}]\bigr)\Bigr)
\]
where $H \defeq Z_{\widehat G^+}(\varphi |_{W_F})$
is a (possibly disconnected, complex) reductive group and
$u_{\varphi} \defeq \varphi\bigl(1,[\begin{smallmatrix}1&1\\0&1\end{smallmatrix}]\bigr) \in H$
is unipotent.

What is the relationship between the local Langlands conjectures
and our earlier notion of supercuspidal support?
In \cite{AMS18}, Aubert--Moussaoui--Solleveld use their generalized Springer correspondence
for disconnected reductive groups to define a notion of cuspidality
for $L$-parameters with Arthur--Vogan enhancement
as well as a \mathdef{cuspidal support} map $\Cusp_\tn{AMS}$,
which produces from $(H, u_\varphi, \rho)$
a cuspidal enhanced $L$-parameter for some Levi subgroup of a pure inner twist of $G$.
They conjecture that under the refined local Langlands correspondence,
cuspidal enhanced $L$-parameters correspond to supercuspidal representations
and $\Cusp_\tn{AMS}$ corresponds to the supercuspidal support map for these enhanced $L$-parameters.

Our goal is to extend this story to rigid enhancements of $L$-parameters.
This extension enables, among other applications,
a conjectural description of the Bernstein decomposition
for representations of rigid inner twists.
In the second half of the article
we redevelop the theory of the cuspidal support map in the rigid setting,
culminating in the following conjecture.

\begin{conjecture}\label{introconj1} (\Cref{thm56} and \Cref{thm60}) 
The local Langlands correspondence for $G$
and rigid inner twists of its Levi subgroups
makes the following square commute:
\[
\begin{tikzcd}[column sep=large, row sep=large]
\Pi^{\rig,Z}(G)  \arrow["\Sc+\tn{ABPS}"]{d} \arrow["\rec_{ \mathfrak{w}}"]{r} &
\Phi_\tn{e}^Z({}^LG) \arrow["\Cusp"]{d} \\
\bigsqcup\limits_{\substack{[\mathcal{T}] \in H^{1}(\mathcal{E}, Z \to G)\\M' \in \Lev(G^{\mathcal{T}})}}
\bigl(\Pi_\tn{sc}(M') \sslash W(G,M')\bigr)_\natural \arrow["\rec_{\mathfrak{w}_{M}}"]{r} &
\bigsqcup\limits_{\substack{[\mathcal{T}] \in H^{1}(\mathcal{E}, Z \to G)\\M' \in \Lev(G^{\mathcal{T}})}}
\bigl(\Phi_{\tn{cusp}}(M',\mathcal{T}_{M}) \sslash W({}^LG, {}^LM)\bigr)_{\kappa}.
\end{tikzcd}
\]
\end{conjecture}

In the above statement,
\begin{itemize}
\item
$\Phi_\tn{e}^Z({}^LG)$ is the set of $\widehat G$-conjugacy classes of
$L$-parameters of $G$ with rigid enhancement,

\item
$\Pi^{\rig,Z}(G)$ is the set of isomorphism classes of irreducible representations of rigid inner twists of $G$,

\item
``$\Sc+\tn{ABPS}$'' is the composition of the supercuspidal support map and Solleveld's solution to the ABPS conjecture \cite[Theorem~E]{solleveld22},

\item
``$\Cusp$'' is our rigid analogue of the cuspidal support map $\Cusp_\tn{AMS}$ of \cite{AMS18},

\item
the cohomology set $H^{1}(\mathcal{E}, Z \to G)$ parameterizes rigid inner twists $(G^{\mathcal{T}}, \mathcal{T})$ of $G$ (see \Cref{sec:corr:gerbe}),

\item
$\Lev(G^{\mathcal{T}})$ is a fixed set of representatives of Levi subgroups of $G^{\mathcal{T}}$ up to $G^{\mathcal{T}}(F)$-conjugacy,

\item
``$(-\sslash-)_{-}$" is notion for a \mathdef{twisted extended quotient} (see \Cref{sec:bernstein:quotient}),

\item
$\Pi_\tn{sc}(M')$ is the set of isomorphism classes of supercuspidal representations of $M'(F)$,

\item
$\Phi_{\tn{cusp}}(M',\mathcal{T}_{M})$ is the set of cuspidal enhanced $L$-parameters mapping, via \eqref{centralchar} and Tate--Nakayama duality, to the rigid inner twist $(M', \mathcal{T}_{M})$, and

\item
the superscript ``$Z$'' records the fixed choice of~$Z$.
\end{itemize}

As \cite[7.2, 7.3]{Solleveld} outlines,
when $Z=Z(G_\tn{der})$ it is possible to formulate \Cref{introconj1}
by reducing to the earlier work of Aubert--Moussaoui--Solleveld
rather than redeveloping the theory of the cuspidal support map.
Nonetheless, our direct approach,
which in particular applies the disconnected generalized Springer correspondence
directly to the groups $Z_{\widehat G^+}(\varphi |_{W_F})$,
has several advantages.

First, our methods permit $Z$ to be an arbitrary subgroup of~$Z(G)$
instead of a subgroup of $Z(G_\tn{der})$.
This level of generality is needed in \cite{kaletha18a}
to compare the rigid enhancement of $L$-parameters with a third enhancement that uses isocrystals.
The isocrystal enhancement is of great interest because,
as elaborated in \cite{bertoloni_meli_oi23},
it is closely related to the geometrization program of Fargues and Scholze.
One would hope for a theory of cuspidal support for isocrystal enhancements of $L$-parameters,
and once this theory is developed, our higher level of generality
should assist in a comparison between this theory and the rigid version.

Second, our methods illuminate a new compatibility condition
between cuspidal support and twisting by characters.
More specifically, let $G'$ be an inner form of~$G$.
On the one hand, after we fix a rigidification~$\cal T$ of~$G'$,
the fiber over $[\cal T]$ of the commutative diagram in \Cref{introconj1}
predicts that the cuspidal support map for irreducible representations of $G'(F)$
can be described purely in terms of $L$-parameters and the Langlands correspondence
for $G'$ and its Levi subgroups, as ${\rec_{\frak w}^{-1}}\circ{\Cusp}\circ{\rec_{\frak w}}$.
On the other hand, formation of supercuspidal support does not depend on
either of these auxiliary choices, the rigidification or the Whittaker datum.
On the Galois side, the expected effect of changing either of these choices is well understood
(as explained in \cite[\S 6]{kaletha18a} and \cite[\S 4]{kaletha13})
and amounts to twisting the enhancement by a certain character.
All in all, we find that the independence of the supercuspidal support from these two choices
is expressed in a natural compatibility of the cuspidal support map
and twisting the enhancement by a certain character (\Cref{thm66,thm95}).

Proving this compatibility reduces to proving an analogous compatibility
in the disconnected generalized Springer correspondence (\Cref{thm18}).
In the formulation of Aubert--Moussaoui--Solleveld
verifying this compatibility condition would seem to require
a careful selection of various representatives
that arise in the construction,
but in our formulation, the compatibility is almost immediate.

\subsection{Notation}
We fix a field isomorphism $\bbC\simeq\overline\bbQ_\ell$,
which we use to freely pass between the two fields.
More precisely, in \Cref{sec:springer,sec:param} we use
$\bbC$ as the base field for varieties or stacks
and $\overline\bbQ_\ell$ as the coefficient field
for constructible or perverse sheaves,
whereas in \Cref{sec:bernstein,sec:corr}
we always use complex coefficients
for $\widehat G$ and for representations
of finite groups or of~$G(F)$.
Our field isomorphism $\bbC\simeq\overline\bbQ_\ell$
transfers results between these two settings.
We make this standard but unsatisfying
maneuver because one needs $\overline\bbQ_\ell$
to talk about perverse sheaves%
\footnote{In fact, one could probably rewrite
\Cref{sec:springer} using perverse sheaves
with complex coefficients and the analytic topology,
the primary setting of \cite{achar21}.
We did not choose this option because
we are unaware of any systematic development
of a six-functor formalism for perverse sheaves
in the analytic setting.}
but $\bbC$ to talk about temperedness
(cf.\ \Cref{thm69}).

Suppose $G$ is a linear algebraic group acting on a variety~$X$.
In \Cref{sec:springer} we write $X/G$ for the stack quotient,
but in the rest of the paper, $X/G$ refers to the ordinary quotient.
Given $Y\subseteq X$, we write $N_G(Y)$ for
the stabilizer of $Y$ and $Z_G(Y)$ for
the pointwise stabilizer of~$Y$.
Given $x\in X$, we write $[x]_G$ for the $G$-orbit of~$x$.

Our conventions for $L$-parameters are slightly different
than the usual ones and for us an $L$-parameter
is often a pair $(\phi,u)$ (or $(\psi,v)$)
instead of a homomorphism~$\varphi$;
see \Cref{sec:lparam} for details.

Usually $G$ is a reductive group.
In \Cref{sec:bernstein,sec:corr} we assume $G$ is connected
but in \Cref{sec:springer,sec:param} we make no such assumption.

The cheat sheet below records notation used in this paper,
together with an optional parenthetical cross-reference for more information on the notation.

\begin{multicols}{2}
\begin{gloss}
\item[$A_G(u)$]
$\pi_0(Z_G(u))$

\item[$\frak B(G)$]
inertial~eq.~cl.~of cuspidal data
(\S\ref{sec:bernstein:variety})

\item[$\frak B({}^LG)$]
inertial~eq.~cl.~of cuspidal data
(\S\ref{sec:bernstein:variety})

\item[$\Cusp$]
cuspidal support of enh.~parameter
(\S\ref{sec:bernstein:support})

\item[$\uCusp$]
unnormalized cuspidal support 
(\S\ref{sec:bernstein:support})

\item[$\connsupp$]
connected cuspidal support
(\S\ref{sec:springer:cuspidal})

\item[$\qcsupp$]
quasi-cuspidal support
(\S\ref{sec:springer:cuspidal})

\item[$\Dbc(-)$]
bdd der.\ cat.\ of constr.\
$\ell$-adic shv.\ (\S\ref{sec:springer:sheaves})

\item[$\cal E$ (or $q\cal E$)]
(quasi-cuspidal) local system (in \S\S\ref{sec:springer}--\ref{sec:bernstein})

\item[$\cal E$]
Kaletha gerbe (in \S\ref{sec:corr}; see \S\ref{sec:corr:gerbe})

\item[$\cal F$]
a perverse (or just constructible) sheaf

\item[$F$]
nonarchimedean local field

\item[$\overline F$]
a fixed algebraic closure of~$F$

\item[$F^\tn{sep}$]
separable closure of~$F$ in~$\overline F$

\item[$\Four_{\frak g}$]
Fourier--Laumon transform (for $\frak g$) (\S\ref{sec:param:fibers})

\item[$G$]
reductive $\bbC$- (\S2--\S3)
or (conn.)\ $F$-group (\S4--5)

\item[$G^\circ$]
connected component of the identity

\item[$G'$]
(rigid) inner form of~$G$ (\S\ref{sec:corr:gerbe})

\item[$G_\tn{der}$]
derived subgroup of~$G$

\item[$G_\tn{sc}$]
simply-connected cover of~$G_\tn{der}$

\item[$(G',\xi)$]
inner twist of~$G$

\item[$(G',\xi,\bar h,\cal T)$]
a rigid inner twist of~$G$
(\S\ref{sec:corr:gerbe})

\item[$G^{\cal T}$]
the inner twist of~$G$ corr.~to~$\cal T$
(\S\ref{sec:corr:gerbe})

\item[$\widehat G$]
Langlands dual group of~$G$ over~$\bbC$

\item[$G\actson X$]
a variety with group action
(\S\ref{sec:springer:stacks})

\item[$H\times^G X$]
induction space (\S\ref{sec:springer:stacks})

\item[$H$]
possibly disconnected reductive $\bbC$-group

\item[$H^\circ$]
identity component of~$H$

\item[$\IC(\cal O,\cal E)$]
intersection cohomology complex

\item[$\infl(\cal E)$, $\infl_{G\tn{-reg}}(\cal E)$]
\eqref{thm97}

\item[$\Irr(H)$]
isom.~classes of algebraic irreps of~$H$

\item[$\Irr\bigl(\Perv(\frak X)\bigr)$]
irr.~perverse sheaves on~$\frak X$
(\S\ref{sec:springer:sheaves})

\item[$\ind_{L\subseteq P}^G$]
parabolic ind.\ of perv.\ sheaves
(\S\ref{sec:springer:ind})

\item[$K^+$]
quasi-Levi subgroup of~$H^+$
(cf.~\eqref{thm52})

\item[$\Lev(G)$]
representatives of
conjugacy classes of Levi subgroups of~$G$

\item[$\Loc(\frak X)$]
$\overline\bbQ_\ell$ local systems on~$\frak X$

\item[$M$]
quasi-Levi subgroup of~$H$,
often $Z_{\widehat G^+}(\phi)$

\item[$\cal M$]
Levi $L$-subgroup of~${}^LG$
(\S\ref{sec:bernstein:levi})

\item[$\cal N_G$]
nilpotent cone of~$G$

\item[$N_G(x)$]
stabilizer of~$x$ in~$G$

\item[$\cal O$]
nilpotent orbit in $\Lie(G)$

\item[$\cal P$]
admissible class of parabolics
(\S\ref{sec:springer:cuspidal})

\item[$\Perv(\frak X)$]
perverse sheaves on~$\frak X$
(\S\ref{sec:springer:sheaves})

\item[$\rec$]
local Langlands correspondence
(\S\ref{sec:corr:conj})

\item[$\res_{L\subseteq P}^G$]
parabolic restr.\ of perv.\ sheaves
(\S\ref{sec:springer:ind})

\item[$s$ (or $\hat s$)]
cuspidal datum for~$\widehat G$ (or~$G$)
(\S\ref{sec:bernstein:variety})

\item[$\frak s_{\widehat G}$
(or $\hat{\frak s}_{\widehat G}$)]
$\llbracket s \rrbracket_{\widehat G}$
(or $\llbracket \hat s \rrbracket_{\widehat G}$)
(\S\ref{sec:bernstein:variety})

\item[$\cal T$]
fpqc $G$-torsor on~$\cal E$
(\S\ref{sec:corr:gerbe})

\item[$t^\circ$, $qt$]
certain (quasi-)cuspidal supports (\S\ref{sec:param:fibers})

\item[$\frak t^\circ$, $q\frak t$]
conjugacy classes of $t^\circ$ and $qt$
(\S\ref{sec:param:fibers})

\item[$\cal U_G$]
unipotent variety of~$G$
(\S\ref{sec:param:springer})

\item[$u$, $v$]
unipotent elements

\item[$u$]
$\varprojlim_{E/F,n} \text{Res}_{E/F}(\mu_{n})/\mu_{n}$
(only in \S\ref{sec:corr:gerbe})

\item[$W_F$]
Weil group of~$F$

\item[$W(G,M)$]
Weyl group $N_G(M)/M$

\item[$W({}^LG, \cal M)$]
``Weyl group'' $N_{\widehat G}(\cal M)/\cal M^\circ$
(\S\ref{sec:bernstein:variety})

\item[$W_{qt}$] 
$N_G(qt)/M$ (\S\ref{sec:param:fibers})

\item[$W_{q\frak t}$]
$\varprojlim_{[qt]_G = q\frak t} W_{qt}$
(\S\ref{sec:param:fibers})

\item[$W_{t^\circ}$]
$N_{G^\circ}(M^\circ)/M^\circ$
(\S\ref{sec:param:fibers})

\item[$W_{\frak t^\circ}$]
$\varprojlim_{[t^\circ]_{G^\circ} = \frak t^\circ} W_{t^\circ}$
(\S\ref{sec:param:fibers})

\item[$\frak w$]
Whittaker datum for~$G$
(\S\ref{sec:corr:whittaker})

\item[$\frak X$, $\frak Y$]
algebraic stack over~$\bbC$, usually $X/G$
(\S\ref{sec:springer:stacks})

\item[$(X\sslash\Gamma)_\kappa$]
twisted extended quotient (\S\ref{sec:bernstein:quotient})

\item[$X^*_\tn{nr}(G)$]
unramified characters of $G(F)$
(\S\ref{sec:bernstein:variety})

\item[$X^*_\tn{nr}({}^LG)$]
Galois analogue of $X^*_\tn{nr}(G)$
(\S\ref{sec:bernstein:variety})

\item[$Y_{(L,\cal O)}$, $Y_{(\frak l,\cal O)}$]
Lusztig's strata (\S\ref{sec:param:regular})

\item[$Z$]
finite central $F$-subgroup of~$Z(G)$

\item[$Z_G(x)$]
pointwise stabilizer of~$x$ in~$G$

\item[$Z(G)$]
center of~$G$

\item[$z^{\val_F}\cdot\phi$]
unramified twist of~$\phi$
(\S\ref{sec:bernstein:variety})

\item[$\Gamma$]
$\Gal(F^\tn{sep}/F)$
(except in \S\ref{sec:bernstein:quotient})

\item[$\kappa_{(\cal O,q\cal E)}$]
AMS cocycle (\ref{thm87})

\item[$\lambda\colon G\to\frak g$]
quasi-logarithm (\S\ref{sec:param:quasi-log})

\item[$\Pi(G)$]
(smooth, $\bbC$)-irreps of~$G(F)$
(up to isom.)

\item[$q\Sigma_{q\frak t}$]
$\qcsupp_G^{-1}(q\mathfrak{t})
\simeq \Irr(\bbC[W_{q\frak t}, \kappa_{q\frak t}])$
(\S\ref{sec:param:fibers})

\item[${}^L\Sigma_{q\frak t,\psi,v}$]
$\uCusp^{-1}(s)
\simeq \Irr(\bbC[W_{q\frak t}, \kappa_{q\frak t}])$
(\S\ref{sec:bernstein:fibers})

\item[$\pi_0(G)$] components group $G/G^\circ$

\item[$\rho$, $\varrho$]
elements of $\Irr\bigl(\pi_0(Z_{\widehat G}(\varphi)^+)\bigr)$
(\S\ref{sec:corr:enh})

\item[$\Phi(G)$]
(equiv.~cl.~of) $L$-parameters of~$G$

\item[$\Phi_\tn{e}^Z(G)$]
(equiv.~cl.~of) $Z$-enhanced $L$-params

\item[$\Phi_\tn{e,cusp}^Z(G)$]
cuspidal elements of $\Phi_{\tn e}^Z$

\item[$\Phi_\tn{e}^Z(G)^{\hat{\frak s}_{\widehat G}}$]
block of $\Phi_\tn{e}^Z(G)$
for~$\hat{\frak s}_{\widehat G}$
(\S\ref{sec:bernstein:quotient})

\item[$\varphi$ or $(\phi,u)$]
an $L$-parameter for~$G$ (\S\ref{sec:lparam})

\item[$\Omega(G)$]
Bernstein variety of~$G$
(\S\ref{sec:bernstein:variety})

\item[$\Omega({}^LG)$]
Bernstein variety of~${}^LG$
(\S\ref{sec:bernstein:variety})

\item[$(-)^+$]
$\widehat{G/Z}\times_{\widehat G}(-)$
(\S\ref{sec:corr:enh})

\item[$(-)^*$]
the Pontryagin dual (esp.\ in \S\ref{sec:corr})

\item[$\llbracket - \rrbracket_G$]
the $G$-inertial eq.~cl.~of $(-)$
(\S\ref{sec:bernstein:variety})

\end{gloss}
\end{multicols}

\subsection*{Acknowledgments}
We would like to thank Tom Haines for proposing this project,
Alexander Bertoloni Meli for directing us to \cite[\S 2.2.2]{ABMUnitary}
and pointing out that $\Spin_{1225}$ has two cuspidal sheaves,
Johannes Anschütz and Gurbir Dhillon for providing
some tips on computing fiber products of quotient stacks,
and Alex Youcis for patiently explaining to us some aspects of
local systems in the lisse-étale topology.
This project benefited greatly from the excellent textbook of Achar \cite{achar21}.

\section{Constructing induction series}
\label{sec:springer}
Let $G$ be a complex reductive group, possibly disconnected.
In this section we explain how to organize the category
$\Perv(\cal N_G/G)$ into induction series,
the first step in the generalized Springer correspondence
for disconnected reductive groups.
The conclusion is similar to the earlier work
of Aubert--Moussaoui--Solleveld \cite[\S2--5]{AMS18}
but our treatment is different from theirs
and we hope the reader finds it illuminating
to compare the two perspectives.

For the most part, we will use properties of perverse sheaves as a black box.
These properties are nicely summarized in the reference appendix to \cite{achar21},
as well as the Springer correspondence itself in Chapter~8.
Although Achar's book works in the analytic category,
rather than the étale category, many of these properties,
especially those using the six-functor formalism,
are known to hold or can be proved in a similar way in both settings.

\subsection{Background on algebraic stacks}
\label{sec:springer:stacks}
Let $X$ be a (possibly singular) variety over $C$
and let $G$ an affine algebraic group (also over $C$) acting on~$X$.
We call $X$ a \mathdef{$G$-variety},
and the pair $G\actson X$ a \mathdef{variety with group action}.
Write $\tn{act}\colon G\times X\to X$ for the action map.

Generally speaking,%
\footnote{We are deliberately vague here
about the specific type of sheaf,
for instance constructible or quasi-coherent,
because this formalism works for
many different types of sheaves.}
a \mathdef{$G$-equivariant sheaf}
on $X$ is a sheaf~$\cal F$ on~$X$
together with an isomorphism
$\tn{pr}_2^*(\cal F)\simeq\tn{act}^*(\cal F)$
of sheaves on~$G\times X$.
There is a remarkable class of geometric objects
known as \mathdef{algebraic stacks},%
\footnote{For us, it suffices to work with
algebraic stacks of finite type over a fixed algebraically closed field of characteristic zero.}
enlarging the usual category of schemes,
which contains spaces whose sheaves are equivariant sheaves.
Namely, there is an algebraic stack $X/G$,
called a \mathdef{quotient stack},
and an abelian%
\footnote{The situation is much more complex for derived categories:
the derived category of equivariant sheaves
is not some category of complexes of sheaves
with an equivariant structure
\cite{bernstein_lunts94}.}
category of sheaves on~$X/G$
that is equivalent to the usual abelian
category of $G$-equivariant sheaves on~$X$.

Although there are many potential motivations
for working with algebraic stacks,
for us, the primary motivation
is that quotient stacks give
a convenient language
for discussing equivariant sheaves.
In this section we give a few comments
on quotient stacks,
the goal being to orient the reader
to whom these objects are unfamiliar
without losing ourselves in a thicket
of technical detail.

\begin{remark}
Strictly speaking, there is no technical need
to phrase our formulation of
the generalized Springer correspondence
in the language of stacks:
it is perfectly possible to avoid this language.
We chose this language nonetheless
-- and at risk of losing the reader --
for the following reasons.

First, many of the definitions
in the stack setting are conceptually simpler.
For instance, as we will see,
the definitions of parabolic induction
and parabolic restriction are as simple
as one could possibly hope,
as pull-push operations along a correspondence of stacks.

Second, the main property that we will prove
in this section on equivariant perverse sheaves,
the Geometrical Lemma, is fundamentally a statement
about the geometry of a certain stack.
Although the proof could be phrased without stacks,
isolating the key aspects of the geometry
could allow the Geometrical Lemma
to be easily transferred to other classes of sheaves.

Third, and relatedly, stacks play an important role
in the geometrization of the local Langlands correspondence
\cite{fargues_scholze21v2}.
It is natural to hope that the cuspidal support
map and related constructions for $L$-parameters
can be phrased on the level of stacks
in a way that integrates with
Fargues and Scholze's stack of $L$-parameters.
\end{remark}

\begin{remark}
In the notation $X/G$, the action of $X$ is left implicit.
The notation is usually clear from context.
We warn the reader, however, that the notation is ambiguous when $X=G$:
there are two natural actions of~$G$ by itself,
translation and conjugation.
The quotient under the translation action is a point
while the quotient under the conjugation action
is much more complex.
\end{remark}

\begin{definition}
The \mathdef{quotient stack} $X/G$ is defined as
the fibered category (cf., e.g. \cite[Definition 2.3]{Dillery}) $$X/G \to \text{Sch}/\mathrm{Spec}(C)$$
whose fiber $(X/G)(S)$ over a $C$-scheme $S$ is the category
of pairs $(T, h)$, where $T \to S$ is an \'{e}tale $G_{S}$-torsor
and $h \colon T \to X_{S}$ is a $G_{S}$-equivariant morphism
of schemes over $S$. The morphisms in this category are
morphisms of torsors making the obvious diagrams
commute.
\end{definition}

Since by assumption $G$ is smooth, the object $X/G$ is always an algebraic stack.%
\footnote{Depending on one's preferred definition of algebraic stack,
when $X$ is singular an appeal to Artin's theorem on flat groupoids
\cite[06DC]{stacks}
is needed to show that $X/G$ is an algebraic stack.}
If the action of $G$ on~$X$ is sufficiently nice
then the stack quotient is again a variety.
For instance, if $X$ is a $G$-torsor over the variety~$Y$,
then the stack quotient $X/G$ is just~$Y$.

The quotient stack construction is functorial
in the following sense.
Suppose we have two varieties with group action
$G\actson X$ and $H\actson Y$,
together with a homomorphism $G\to H$ and
a map $X\to Y$ equivariant for this homomorphism.
We call this pair of maps a morphism of varieties with group action
and denote it by $(G\actson X)\to(H\actson Y)$ in brief.
Then there is an induced map $X/G\to Y/H$ of quotient stacks.

In particular, if $G=H$ and $X$ is a subvariety of~$Y$
that is open, closed, or locally closed
then $X/G$ is an open, closed, or locally closed
substack of~$Y/G$, respectively.
So $G$-equivariant stratifications of~$X$
give rise to stratifications of the quotient~$X/G$.

In our proof of the Geometrical Lemma,
there are several points where it becomes necessary
to compute a fiber product of quotient stacks.
The general problem would be the following:
given a diagram
\[
\begin{tikzcd}
&
(G'\actson X') \dar{f'} \\
(G \actson X) \rar{f} &
(Y \actson H)
\end{tikzcd}
\]
of varieties with group action,
compute the fiber product $X/G\times_{Y/H}X'/G'$
of the induced diagram of quotient stacks.
Without commenting on the general calculation
of such a fiber product,
we describe the three cases of relevance to us.

First, suppose $G=H=G'$.
Then formation of fiber product commutes with quotients:
\[
X/G\times_{Y/G}X'/G \simeq (X\times_Y X')/G.
\]

Second, suppose the maps of groups are inclusions.
In this case, define the induction space $H\times^G X$
as the geometric quotient of $H\times X$
by the diagonal $G$-action $g\cdot(h,x) = (hg^{-1},gx)$.
This space is an~$H$-space: $H$ acts on~$H\times^G X$
by left multiplication on the first factor.
The map $X\to H\times^G X$ defined by $x\mapsto(1,x)$ is equivariant
for the inclusion $G\hookrightarrow H$ and induces an isomorphism of stacks
\[
\frac{X}{G} \simeq \frac{H\times^G X}{H}.
\]
Moreover, the map $X\to Y$ extends to an $H$-equivariant map
$H\times^G X\to Y$ by acting on~$Y$, and similarly for $G'\actson X'$.
Now that all the stacks have a common denominator,
we can compute that
\[
X/G\times_{Y/H}X'/G' \simeq \bigl((H\times^G X)\times_Y(H\times^{G'} X')\bigr)/H.
\]

Third, suppose one of the maps of groups, say~$f$,
is surjective and admits a (homomorphic) section~$s$.
This case is more subtle because the maps on quotient stacks
need not be representable.

\begin{lemma} \label{thm20}
Let $(G\actson X)\xrightarrow{f}(H\actson Y)\xleftarrow{f'}(G'\actson X')$
be maps of varieties with group action.
Suppose $f$ is surjective and $f|_H$ admits a homomorphic section~$s$.
Then the diagram at left induces
a Cartesian diagram of quotient stacks as at right.
\[
\begin{tikzcd}
(G\times_HG'\actson X\times_Y X') \rar\dar & (G'\actson X') \arrow{d}[name=A]{}
& (X\times_Y X')/(G\times_H G') \rar\arrow{d}[name=B]{} & X'/G' \dar
\\
(G\actson X) \rar & (H\actson Y)
& X/G \rar & Y/H.
\arrow[shorten >=15pt, shorten <=10pt, mapsto,to path={(A) -- (B)}]
\end{tikzcd}
\]
\end{lemma}

\begin{proof}
The diagram at left induces a homomorphism
\[
\begin{tikzcd}[column sep=huge]
\phi\colon \frak G \defeq
(G\times_HG')\ltimes(X\times_Y X') \rar &
(G\ltimes X)\times_{(H\ltimes Y)}(G'\ltimes X')
\eqdef \frak G_\tn{st}
\end{tikzcd}
\]
from the action groupoid of the fiber products
to the fiber products of the action groupoids.
(The subscript ``st'' on $\frak G_\tn{st}$ abbreviates ``strict'',
as this groupoid has fewer objects than~$\frak G$.)
We will show that this map is an equivalence of groupoids
and thus induces an isomorphism of quotient stacks,
by exhibiting a quasi-inverse.
Namely, define $\psi\colon\frak G_\tn{st}\to\frak G$
as follows, where we implicitly use~$s$ to interpret $G$
as a subgroup of~$H$.
On $0$-simplices (where $x \in X(S)$, $x' \in X'(S)$, and $h \in H(S)$ for some $C$-scheme $S$),
\[
\psi \colon (x,x',h)\mapsto(h\cdot x,x').
\]
On $1$-simplices (with $x_{i}, x_{i}'$, and $h_{i}$ as above and $g \in G(S)$, $g' \in G'(S)$),
\[
\psi \colon \bigl((g,g')\colon
(x_1,x_1',h_1) \to (x_2,x_2',h_2) \bigr)
\mapsto \bigl((h_2gh_1^{-1},g')\colon
(h_1x_1,x_1',1) \to (h_2x_2,x_2',1) \bigr).
\]
Then $\phi\circ\psi = \tn{id}$,
and we leave it to reader
to construct a $2$-isomorphism
$\psi\circ\phi\simeq\tn{id}$.
\end{proof}

\subsection{Background on equivariant perverse sheaves}
\label{sec:springer:sheaves}
At this point, there is a well-developed theory
of perverse $\ell$-adic sheaves on an algebraic stack $\frak X$
\cite{behrend03,laszlo_olsson08a,laszlo_olsson08b,laszlo_olsson09,
liu_zheng12av4,liu_zheng12bv3,liu_zheng14v2}
Let $\Dbc(\frak X) \defeq \Dbc(\frak X,\overline\bbQ_\ell)$
denote the derived category of constructible
of $\ell$-adic sheaves with coefficients in $\overline\bbQ_\ell$
as constructed by any of these authors.
Inside $\Dbc(\frak X)$ one has the category $\Loc(\frak X)$
(sitting in some fixed degree)
of $\overline\bbQ_\ell$-local systems on~$\frak X$
in the lisse-étale site.

In full generality there are many subtleties in this theory,
most notably, the failure of the pullback
and pushforward functors to preserve boundedness.
But in our case, when the stacks are
quotient stacks for varieties with group action,
the six operations can be described more simply
in terms of operations on equivariant perverse sheaves,
as outlined for instance in \cite[Table 6.8.1]{achar21}.
Moreover, in our setting pushforward maps
will usually preserve boundedness:
in the problematic case when
the original map is not representable,
as in the third part of Achar's table,
the pushforward will arise from
a surjective homomorphism with unipotent kernel,
which induces an equivalence of bounded
derived categories \cite[6.6.16]{achar21}
and in particular preserves boundedness.
In this paper we treat the six functors as automatically derived,
writing, for instance, $f_*$ instead of $Rf_*$.

There is one small difference between
$G$-equivariant perverse sheaves on~$X$
and perverse sheaves on~$X/G$, however:
a shift in the perverse $t$-structure
to account for the fact that
(at least when $X$ is equidimensional)
$\dim(X/G) = \dim(X) - \dim(G)$.
Writing $\tn{D}^\tn{b}_G(X)$ for the classical Bernstein--Lunts
\cite{bernstein_lunts94} category of $G$-equivariant perverse sheaves
and writing $\iota \colon \Dbc(X/G) \simeq \tn{D}^\tn{b}_G(X)$
for the natural (if one likes, tautological) identification,
the functor
\[
\iota[\dim G] \colon \Dbc(X/G) \to \tn{D}^\tn{b}_G(X)
\]
is a $t$-exact equivalence for
the perverse $t$-structures on source and target.
This shift is a potential source of confusion
during comparison with earlier works, such as \cite{lusztig84b}.

In the classical theory of constructible $\ell$-adic sheaves
on a connected Noetherian scheme~$X$,
there is a well-known equivalence of categories between
finite-dimensional $\ell$-adic local systems on~$X$
and finite-dimensional $\overline\bbQ_\ell$-representations
of the étale fundamental group $\pi_1^\tn{ét}(X)$.
Unsurprisingly, this bijection generalizes
to connected Noetherian algebraic stacks.

\begin{lemma} \label{thm81}
Let $\frak X$ be a connected Noetherian algebraic stack
and let $\pi_1^\tn{ét}(\frak X)$ be Noohi's étale fundamental group.
The natural Galois category structure on $\frak X_\tn{fét}$
induces an equivalence between the categories of
\begin{enumerate}
\item
finite-dimensional lisse-étale $\overline\bbQ_\ell$-local systems on~$\frak X$ and
\item
finite-dimensional $\overline\bbQ_\ell$-representations of $\pi_1^\tn{ét}(\frak X)$.
\end{enumerate}
\end{lemma}

In the proof we will explain the new terminology
used in the statement of the \namecref{thm81}.

\begin{proof}[Proof sketch]
Noohi's thesis \cite{noohi04}
developed the theory of $\pi_1^\tn{ét}(\frak X)$,
showing that it behaves like the classical case.
In more detail, one defines the finite étale site, $\frak X_\tn{fét}$,
of~$\frak X$ in the usual way,
where the covers are representable finite étale maps,%
\footnote{Contrary to \cite[1.4]{noohi04}, however,
one should not require the source in a finite étale cover to be connected.}
and one then checks that $\frak X_\tn{fét}$ forms a Galois category.
The resulting identification of $\frak X_\tn{fét}$
with the category of finite $\pi_1^\tn{ét}(\frak X)$-sets
gives rise to an equivalence between
finite-dimensional étale local systems of $\overline\bbQ_\ell$-vector spaces
and finite-dimensional $\overline\bbQ_\ell$-representations
of~$\pi_1^\tn{ét}(\frak X)$,
as follows from unspooling the (lengthy!)
construction of the former category,
as given in (for instance) \cite[Chapter~1.2]{behrend03}.

This identification is not the end of the story, however,
because the the category $\Dbc(\frak X)$,
and thus the local systems in that category,
is constructed using the lisse-étale site $\frak X_\tn{lis-ét}$ of~$\frak X$.
As Noohi mentions at the beginning of \cite[Page~76]{noohi04},
this change of topology is immaterial;
let us briefly explain why.

To complete the construction of the equivalence,
consider the Yoneda functor
\[
\frak X_\tn{fét} \longrightarrow
\tn{Shv}(\frak X_\tn{lis-ét})
\colon
\frak Y\longmapsto\Hom_{\frak X}(-,\frak Y)
\]
to the lisse-étale topos.
Since objects of $\frak X_\tn{fét}$ and $\frak X_\tn{lis-ét}$
are representable, each of these $\Hom$
functors is valued in set(oid)s,
and since $X$ is Noetherian, $\Hom_{\frak X}(-,\frak Y)$
is a finite locally constant sheaf.
The Yoneda functor above is fully faithful by Yoneda's lemma,
so it remains to show that it is essentially surjective.
For this, given a finite locally constant sheaf $F$
on $\frak X_\tn{lis-ét}$, define the prestack $\frak Y_F$
by setting $\frak Y_F(S)$, for $S$ a scheme,
to be the set of pairs $(x,y)$ with $x\colon S\to\frak X$
and $y\colon\Hom_{\frak X}(-,S)\to F$.
One then checks that $\frak Y_F$ is an algebraic stack
and that the structure map $\frak Y_F\to\frak X$
(namely, $(x,y)\mapsto x$) is finite étale.
\end{proof}

\begin{corollary} \label{thm80}
Let $\frak X$ be an irreducible Noetherian algebraic stack,
let $j\colon\frak U\to\frak X$ be a nonempty open substack,
and let $\cal E\in\Loc(\frak X)$.
Then the induced map
$j^*\colon\End(\cal E)\to\End(j^*\cal E)$
is an isomorphism.
\end{corollary}

\begin{proof}
Since $\frak X$ is irreducible,
any finite étale cover of~$\frak X$ is irreducible.
So the pullback map $\frak X_\tn{fét}\to\frak U_\tn{fét}$
preserves connected objects and thus induces a surjection
$\pi_1^\tn{ét}(\frak X)\to\pi_1^\tn{ét}(\frak U)$.
The result now follows from the bijection of \Cref{thm81}
and the fact that inflation of representations is fully faithful.
\end{proof}

To finish, we prove that when a map of quotient stacks has
suitably contractible fibers,
the resulting $*$-pull $!$-push functor is an equivalence.
Later we will use this result in our proof
of the Geometrical Lemma.

\begin{lemma} \label{thm24}
Let $f = (f_1,f_2) \colon (G\actson X)\to(H\actson Y)$ be
a map of varieties with group action such that
\begin{enumerate}
\item $f_1\colon G\to H$ is surjective with unipotent kernel and

\item $f_2\colon X\to Y$ is a vector bundle of rank~$\dim\ker f_1$.
\end{enumerate}
Let $\bar f\colon X/G\to Y/H$ be the corresponding
map of quotient stacks.
Then there is a natural isomorphism
\[
\bar f_!\circ \bar f^* \simeq \tn{id} \colon \Dbc(Y/H) \to \Dbc(Y/H).
\]
\end{lemma}

\begin{proof}
Factor $\bar f$ as
$\begin{tikzcd}
X/G \rar{\alpha} &
Y/G \rar{\beta} &
Y/H
\end{tikzcd}$
where $\alpha$ is induced by~$f_2$
and $\beta$ is induced by~$f_1$.
Since $\alpha$ is a vector bundle, $\alpha_*\circ\alpha^*\simeq\tn{id}$,
or equivalently, by Verdier duality, $\alpha_!\circ\alpha^!\simeq\tn{id}$.
And since $\alpha$ is smooth of relative dimension~$\dim\ker f_1$
we have $\alpha^! \simeq \alpha^*[2\dim\ker f_1]$.
Putting these together we find that
\[
\alpha_!\circ\alpha^* \simeq [-2\dim\ker f_1](-).
\]
At the same time, since $\ker f_1$ is unipotent,
$\beta^*$ is an equivalence of categories
(cf.~\cite[6.6.16]{achar21}).
Since $\beta_![-2\dim\ker f_1]$ is right adjoint to~$\beta^*$
it is a quasi-inverse, meaning
\[
\beta_!\circ\beta^* \simeq [2\dim\ker f_1](-).
\]
The claim follows by combining the two displayed equations.
\end{proof}

In particular, \Cref{thm24}
implies that pullback of the group action forming a quotient stack
yields a pullback diagram of stacks.

\begin{corollary} \label{thm92}
Let $(G\actson X)\to(G\actson Y)$ be a morphism of varieties with group action
and let $G'\to G$ be a homomorphism of algebraic groups.
Then the resulting commutative square below is a pullback square
\[
\begin{tikzcd}
X/G' \rar\dar & Y/G' \dar \\
X/G \rar & Y/G.
\end{tikzcd}
\]
\end{corollary}

\subsection{Disconnected parabolic subgroups}
Let $G$ be a (possibly disconnected) reductive group.
For us, a subgroup $P$ of~$G$ is \mathdef{parabolic}
if $G/P$ is proper, or equivalently,
$P^\circ$ is a parabolic subgroup of~$G^\circ$.
In this section we review the classification
and basic structure of such subgroups.

The parabolic subgroups of~$G$
are classified combinatorially as follows,
using the theory of generalized Tits systems
\cite[Exercise~IV.2.8]{lie4-6}
Fix a Borel subgroup~$B^\circ$ of~$G^\circ$;
we will classify the parabolic subgroups of~$G$ containing~$B^\circ$.
Fix a maximal torus~$T^\circ$ of~$B^\circ$ and
let $W = N_G(T^\circ)/T^\circ$, an extended Weyl group.
There is a generalized Bruhat decomposition $G = B^\circ WB^\circ$,
and since $P$ is a union of~$B^{\circ}$ double cosets,
we need only describe the corresponding subset of~$W$
representing these double cosets.
For this, let $W^\circ = N_{G^\circ}(T^\circ)/T^\circ$,
an ordinary Weyl group.
Since all maximal tori are conjugate,
there is a short exact sequence
\[
\begin{tikzcd}[column sep=small]
1 \rar &
W^\circ \rar &
W \rar &
\pi_0(G) \rar &
1.
\end{tikzcd}
\]
Since all Killing pairs are conjugate,
$\pi_0(G) = \pi_0\bigl(N_G(B^\circ,T^\circ)\bigr)$
and this identification yields a splitting
of this short exact sequence:
\[
W \simeq W^\circ \rtimes \pi_0(G).
\]
Given a subset~$X$ of the set of simple reflections,
let $W^\circ_X$ denote the subgroup of~$W^\circ$ generated by~$X$
and let $\Omega$ be a subgroup of $\pi_0(G)$
that stabilizes~$X$.
Then the set $P_{(X,\Omega)} \defeq B^\circ W_X\Omega B^\circ$
is a parabolic subgroup of~$G$ containing~$B^\circ$,
and the assignment $(X,\Omega)\mapsto P_{(X,\Omega)}$
is a bijection between the set of such pairs
and the set of such parabolic subgroups.
Moreover, this bijection yields a bijection
between $\pi_0(G)$-conjugacy classes of pairs $(X,\Omega)$
and $G$-conjugacy classes of parabolic subgroups of~$G$.

Let $P$ be a parabolic subgroup of~$G$.
A \mathdef{Levi subgroup}~$L$ of~$P$
is a complement in~$P$ to its unipotent radical~$U$.
In particular, $\pi_0(L) = \pi_0(P)$
and $L$ provides a section
of the projection from $P$ onto its maximal reductive
(in other words, Levi) quotient.
Levi factors always exist:
when $G = G^\circ$ this is well known,
and in general, the assignment
$L^\circ\mapsto N_P(L^\circ)$ defines a bijection
between Levi factors of~$P^\circ$ and those of~$P$.

To finish this section,
we prove that the set of double cosets
for a pair of parabolic subgroups
admits representatives of a particularly nice form.
These representatives play a key role
in our statement of the Geometrical Lemma.

\begin{lemma} \label{thm15}
Let $P$ and~$Q$ be parabolic subgroups of~$G$
and let $L$ and $M$ be Levi factors
of $P$ and~$Q$, respectively.
Let $\Sigma(M,L)$ be the set of $x\in G$
such that $M\cap{}^xL$ contains a maximal torus.
\begin{enumerate}
\item
Inclusion induces a bijection
$M\backslash \Sigma(M,L)/L
\simeq Q\backslash G/P$.

\item
There is a finite sequence $g_1,\dots,g_s$
of double coset representatives of
$M\backslash \Sigma(M,L)/L$ such that
if $Qg_i P\subseteq \overline{Qg_jP}$ then $i\leq j$.
\end{enumerate}
\end{lemma}

It is not too difficult to give
a description of $Q\backslash G/P$
in terms of the extended Weyl group
and the generalized Tits system,
but there is no advantage in this for us
so we omit it.

\begin{proof}
The first part is the same as in the connected case
\cite[5.2.2]{digne_michel20},
which rests ultimately on a reduction
to double cosets in the Weyl group.
This reduction is permitted in the disconnected case
by the combinatorial classification of
disconnected parabolic subgroups discussed above.
For the second part, it suffices to
order the coset representatives so that
if $\dim(Qg_iP) \leq \dim(Qg_jP)$ then $i\leq j$,
and such an order can certainly be arranged.
\end{proof}

\subsection{Parabolic induction and restriction}
\label{sec:springer:ind}
Let $P$ be a parabolic subgroup of~$G$
and $L$ a Levi subgroup of~$P$.
Consider the correspondence of algebraic stacks
\begin{equation}\label{thm71}
\begin{tikzcd}[row sep=tiny]
& \frak p/P \dlar[swap]{\pi} \drar{i} & \\
\frak l/L & & \frak g/G
\end{tikzcd}
\end{equation}
induced by the usual maps from $P$ to $L$ and~$G$.

\begin{definition}
The functor $\uind_{L\subseteq P}^G\colon
\Dbc(\frak l/L) \to\Dbc(\frak g/G)$ is defined by 
\[
\ind_{L\subseteq P}^G \defeq i_!\circ\pi^*.
\]
\end{definition}

The functor $\uind_{L\subseteq P}^G$ is a version
of parabolic induction in the setting of
equivariant constructible sheaves on the Lie algebra.
It satisfies all the usual properties one
expects from the setting of reductive $p$-adic groups,
as we now explain.

It follows from standard adjunctions
that $\uind_{L\subseteq P}^G$ admits a right adjoint,
the functor
\[
\cores_{L\subseteq P}^G \defeq \pi_*\circ i^!.
\]
Since $\pi$ is proper, $\pi_*=\pi_!$.
Moreover, decomposing $i$ as the projection
$\frak p/P \to \frak m/P \to \frak m/M$,
we see by \cite[6.6.16]{achar21} that $i^* = i^!$.
Hence the functor $\uind_{L\subseteq P}^G$
admits a left adjoint as well, the functor
\[
\ures_{L\subseteq P}^G \defeq \pi_!\circ i^*.
\]
It is this left adjoint that is of the most interest to us.

\begin{proposition}
Let $P\subseteq Q\subseteq G$ be parabolic subgroups of~$G$
with respective Levi factors $L\subseteq M$.
Then
\[
\uind_{M\subseteq Q}^G\circ \uind_{L\subseteq M\cap P}^M
\simeq \uind_{L\subseteq P}^G
\qquad\qquad
\ures_{L\subseteq M\cap P}^M\circ\ures_{M\subseteq Q}^G
\simeq \ures_{L\subseteq P}^G.
\]
\end{proposition}

\begin{proof}
The second part follows from the first part by adjunction.
For the first part, by proper base change it suffices to show that
the following diagram is Cartesian:
\[
\begin{tikzcd}
\frak p/P \rar\dar   & \frak q/Q \dar \\
(\frak m\cap \frak p)/(M\cap P) \rar & \frak m/M.
\end{tikzcd}
\]
\Cref{thm20} reduces us to showing that the maps
\[
\frak p \to (\frak m\cap\frak p)\times_{\frak m}\frak q
\qquad\tn{and}\qquad
P \to (M\cap P)\times_M Q
\]
are isomorphisms, which is easy to check.%
\end{proof}

One advantage of treating parabolic subgroups of~$G$
in this generality is that $G^\circ$ is a parabolic subgroup.
Specializing to this case yields the following corollary.

\begin{corollary} \label{thm21}
Let $P\subseteq G$ be a parabolic subgroup of~$G$
with Levi factor~$L$. Then
\[
\ures_{L^\circ\subseteq P^\circ}^{G^\circ}\circ \ures_{G^\circ}^G
\simeq \ures_{L^\circ}^L\circ\ures_{L\subseteq P}^G
\qquad\tn{and}\qquad
\uind^G_{G^0}\circ \uind_{L^\circ\subseteq P^\circ}^{G^\circ}
\simeq \uind_{L\subseteq P}^G\circ\uind^L_{L^\circ}.
\]
\end{corollary}

For the most part, we will apply parabolic induction and restriction
to sheaves supported on the nilpotent cone, yielding functors
\begin{equation} \label{thm67}
\begin{tikzcd}[column sep=large]
\res_{L\subseteq P}^G : \Dbc(\cal N_G/G) \rar[shift left=0.5ex] &
\Dbc(\cal N_L/L) : \ind_{L\subseteq P}^G \lar[shift left=0.5ex]
\end{tikzcd}
\end{equation}
defined by the restriction of the correspondence \eqref{thm67}
to quotients of nilpotent cones.
In this situation, these operations are especially well-behaved.
Recall that a \mathdef{semisimple complex}
is a finite direct sum of shifts of simple perverse sheaves.

\begin{lemma} \label{thm70}
The functors $\ind_{L\subseteq P}^G$ and $\res_{L\subseteq P}^G$
take semisimple complexes to semisimple complexes
and are $t$-exact for the perverse $t$-structure.
\end{lemma}

In fact, the two functors $\pi^*$ and $i_!$
whose composition comprises $\ind_{L\subseteq P}^G$
are $t$-exact for the perverse $t$-structure;
this is one advantage of our indexing conventions for
the perverse $t$-structure on $\Dbc$ of a quotient stack.

\begin{proof}
The first part follows by the decomposition theorem.
The second part is \cite[Theorem~8.4.11]{achar21} when $G=G^\circ$,
and the general case reduces to this connected case
because $t$-exactness can be checked after applying
the forgetful functor from $G$-equivariant complexes
to $G^\circ$-equivariant complexes.
\end{proof}

Next, we show that parabolic induction
is compatible with formation of central characters.
This claim rests on the following natural observation.

\begin{lemma} \label{thm93}
Let $f\colon G'\to G$ be a morphisms of reductive groups
inducing an isomorphism on adjoint quotients
and let $L'\defeq f^{-1}(L)$ and $P' \defeq f^{-1}(P)$.
Then there is an isomorphism of functors
\[
f^*\circ \ind_{L\subseteq P}^G \simeq
\ind_{L'\subseteq P'}^{G'}\circ f^*.
\]
\end{lemma}

\begin{proof}
After factoring $G'\to G$ as a composition
of a surjection followed by an injection,
the claim follows from \Cref{thm92} and base change.
\end{proof}

Given a variety with group action $G\actson X$
and an equivariant perverse sheaf $\cal F\in\Perv(X/G)$,
suppose that $\cal F=\IC(\cal O,\cal E)$
where $\cal E\in\Perv(Y/G)$ for some locally closed $G$-orbit $Y\subseteq X$.
Then, after choosing a basepoint $y\in Y$,
we may identify $\cal E$ with a $\overline\bbQ_\ell$-representation
$\rho_{\cal F} = j_y^*(\cal E)$ of the finite group $\pi_0(Z_G(y))$.
Letting $A\subseteq Z(G)$ be a central subgroup of~$G$,
consider the restriction of $\rho_{\cal F}$ to the finite abelian group
\[
\pi_0\bigl(Z_A(y)\bigr).
\]
along its natural map to $\pi_0(Z_G(y))$.
We say that $\cal F$ \mathdef{admits an $A$-character}
if this restriction is isotypic for a character of $\pi_0\bigl(Z_A(y)\bigr)$,
which is then called the \mathdef{$A$-character}.
In particular, every irreducible $\cal F$ admits an $A$-character.

It is often the case that operations on equivariant perverse sheaves
are compatible with formation of the $A$-character.
To prove such a compatibility result,
it helps to reformulate the $A$-character as follows.
Let $f\colon A\times G\to G$ be the multiplication map,
a group homomorphism because $A$ is central.
In the context of the previous paragraph, we see that 
\[
\rho_{f^*\cal F}\in\Irr\bigl(\pi_0(Z_{A\times G}(y))\bigr)
\simeq\Irr\bigl(\pi_0(Z_A(y))\bigr)\times\Irr\bigl(\pi_0(Z_G(y))\bigr)
\]
is the external tensor product of $\rho_{\cal F}$
and its central character.

\begin{corollary} \label{thm94}
If $\cal F\in\Perv(\cal N_L/L)$ admits a $Z(G)$-character
then $\ind_{L\subseteq P}^G$ admits a $Z(G)$-character,
which is the same as that of~$\cal F$.
\end{corollary}

\begin{proof}
Apply \Cref{thm93} to the multiplication map $Z(G)\times G\to G$.
\end{proof}

Later, in our comparison with the Aubert--Moussaoui--Solleveld
formulation of the Springer correspondence,
we will need a variant of parabolic induction and restriction on the group.
Let $\cal U_G$ be the variety of unipotent elements of~$G$,
the \mathdef{unipotent variety}.
In exactly the same way, the correspondence
\begin{equation}\label{thm72}
\begin{tikzcd}[row sep=tiny]
& P\adq P \dlar[swap]{\pi} \drar{i} & \\
L\adq L & & G\adq G
\end{tikzcd}
\end{equation}
gives rise to two adjoint pairs of functors
\[
\begin{tikzcd}[row sep=tiny]
\uRes_{L\subseteq P}^G : \Dbc(G\adq G) \rar[shift left=0.5ex] &
\Dbc(L\adq L) : \uInd_{L\subseteq P}^G, \lar[shift left=0.5ex] \\
\Res_{L\subseteq P}^G : \Dbc(\cal U_G/G) \rar[shift left=0.5ex] &
\Dbc(\cal U_L/L) : \Ind_{L\subseteq P}^G, \lar[shift left=0.5ex]
\end{tikzcd}
\]
where the second pair is the restriction of the first
to unipotent varieties.
Everything we claimed above for parabolic induction
and restriction on the Lie algebra,
except (possibly) \Cref{thm70},
still holds in this setting with essentially the same proofs.

\begin{remark}
If $j\colon X\to\frak g$ is a locally closed subvariety
then the functor $j_!\colon\Dbc(X/G)\to\Dbc(\frak g/G)$
is an exact fully-faithful embedding.
We will sometimes identify $\cal F$
with its image under this embedding,
writing, for instance,
$\res_{L\subseteq P}^G(\cal F)$ for $\res_{L\subseteq P}^G(j_!\cal F)$.
\end{remark}

\subsection{The Geometrical Lemma}
\label{sec:geometrical}
In this section we will prove the uniqueness part of
$\cal P$-cuspidal support, the hard part of \Cref{thm14}.
The key tool is the following version
of the Geometrical Lemma of Bernstein and Zelevinsky
for perverse sheaves.
Let $P$ and $Q$ be parabolic subgroups of~$G$
with Levi factors $L$ and~$M$ and
let the sequence $g_1,\dots, g_s$
of coset representatives of $W_M\backslash W/W_L$
be as in \Cref{thm15}.

\begin{theorem} \label{thm17}
There is a natural isomorphism of functors
\[
\res_{M\subseteq Q}^G\circ\ind_{L\subseteq P}^G
\simeq \bigoplus_{i=1}^s
\ind_{M\cap{}^{g_i}L\subseteq M\cap{}^{g_i}P}^M
\circ\res^{{}^{g_i}L}_{M\cap{}^{g_i}L\subseteq Q\cap{}^{g_i}L}
\circ\Ad(g_i^{-1})^*.
\]
\end{theorem}

To prove \Cref{thm17} we further reformulate it into a truly
geometric statement about certain quotient stacks.
The starting point is to observe that the lefthand side of \Cref{thm17} is
a composition of two functors obtained by $*$-pulling and $!$-pushing along a correspondence,
and, by proper base change,
such a composition is given by $*$-pulling and $!$-pushing
along the composite correspondence
\[
\begin{tikzcd}[column sep=small, row sep=small]
& &
\frak X \dlar\drar & & \\
& \cal N_P/P \dlar\drar &
& \cal N_Q/Q \dlar\drar & \\
\cal N_L/L &
& \cal N_G/G &
& \cal N_M/M,
\end{tikzcd}
\]
in which the top square is Cartesian.
Similarly, if we omit the initial $\Ad(g_i^{-1})^*$
then the summands appearing in the Geometrical Lemma
are obtained by $*$-pulling and $!$-pushing
along the following correspondence $\frak Y_{g_i} = \frak Y_i$:
\[
\begin{tikzcd}[column sep=small, row sep=small]
& &
\frak Y_i \dlar\drar & & \\
& \cal N_{Q\cap{}^{g_i}L}/(Q\cap{}^{g_i}L) \dlar\drar &
& \cal N_{M\cap{}^{g_i}P}/(M\cap{}^{g_i}P) \dlar\drar & \\
\cal N_{{}^{g_i}L}/{}^{g_i}L &
& \cal N_{M\cap{}^{g_i}L}/(M\cap{}^{g_i}L) &
& \cal N_M/M.
\end{tikzcd}
\]
To prove the Geometrical Lemma,
we introduce a stratification $(\frak X_i)_{i=1}^s$
on~$\frak X$ and construct a map $\frak X_i \to \frak Y_i$.
Although this map is not an isomorphism,
its fibers are contractible and thus
have no effect on the resulting constructions
for perverse sheaves.

To stratify $\frak X$, we first observe that
there is a natural map from $\frak X$
to the fiber product
\[
\frac{\tn{pt}}{P}\times_{\tn{pt}/G}\frac{\tn{pt}}{Q}
\simeq Q\backslash G/P.
\]
So it suffices to stratify the base
and pull back the stratification to~$\frak X$.
But we have already seen a stratification of $Q\backslash G/P$:
in the language of stacks, \Cref{thm15} says that
the locally closed substacks $Q\backslash Qg_iP/P$
form a stratification.
Let $\frak X_{g_i} = \frak X_i$ be
the resulting pullback stratification.

\begin{remark}
The strata $\frak X_i$
all have the same dimension, $-\rank(G)$.
This follows from the fact that
\[
\frak X_i \simeq 
\frac{\cal N_{Q\cap{}^{g_i}P}}{Q\cap{{}^{g_i}P}},
\]
as will be shown in the proof of \Cref{thm16}.
We initially found this behavior surprising,
but in retrospect it can already be seen in 
the following simple toy model,
a special case of the general construction.

For $G=\SL_2$, the line $\cal N_B$
is stratified into two $\bbG_\tn{m}$-orbits,
one of dimension zero and one of dimension one.
The preimage of $\cal N_B$
under the Springer resolution
consists of $\bbA^1$ and $\bbP^1$ glued
together at the origin,
and the preimages of the strata
are $\bbA^1$ and $\bbP^1$,
both of dimension one.

It would be interesting to see
if the equidimensionality
of the strata $\frak X_w$
has any implications for
the study of equivariant perverse sheaves.
\end{remark}

Changing notation slightly,
focus attention on a single $w\in G$
such that $M\cap{}^w L$ contains a maximal torus.

\begin{theorem} \label{thm16}
Let $\cal N_w \defeq \cal N_{Q\cap{}^wL}
\underset{\cal N_{M\cap{}^wL}}{\times}
\cal N_{M\cap{}^wP}$ and
$G_w \defeq (Q\cap{}^wL)\underset{M\cap{}^wL}{\times}M\cap{}^wP$.
\begin{enumerate}
\item
There are isomorphisms
$\frak X_w \simeq \cal N_{Q\cap{}^wP}/(Q\cap{}^wP)$
and $\frak Y_w \simeq \cal N_w/G_w$.

\item
The composition
$f_w\colon\frak X_w \simeq \cal N_{Q\cap{}^wP}/(Q\cap{}^wP)
\longrightarrow\cal N_w/G_w\simeq\frak Y_w$
fits into a commutative diagram
\[
\begin{tikzcd}
\cal N_L/L &
\frak X_w \lar \dar{f_w} \drar \\
\cal N_{{}^wL}/{}^wL \uar{\tn{Ad}(w^{-1})} &
\frak Y_w \lar \rar &
\cal N_M/M.
\end{tikzcd}
\]
\end{enumerate}
\end{theorem}

\begin{proof}
Our proof is a stacky rehash of the proof of
\cite[Lemma 2.9]{achar_henderson_juteau_riche17b},
itself inspired by \cite[\S 10.1]{mars_springer89}.

For $\frak X_w$, by the computations in \Cref{sec:springer:stacks},
\[
\frak X \simeq
\frac{\bigl(G\times^P\cal N_P\bigr)\times_{\cal N_G}
\bigl(G\times^Q\cal N_Q\bigr)}{G}.
\]
We can further pull out the $P$ and $Q$ actions
in both factors to write this as
\[
\frak X \simeq
\frac{\bigl(G\times\cal N_P\bigr)\times_{\cal N_G}
\bigl(G\times\cal N_Q\bigr)}{P\times G\times Q}.
\]
Dividing out by the action of~$G$
eliminates the rightmost $G$ in the numerator
and yields $\frak X \simeq Z/(P\times Q)$
where $Z = \{(g,z)\in G\times\cal N_P \colon g\cdot z\in\cal N_Q\}$
with $P\times Q$-action $(p,q)\cdot(g,z) = (qgp^{-1},pz)$.
Let $Z_w$ be the preimage in~$Z$
of $QwP$ under the projection to~$G$.
Tracing through the isomorphisms, we see that
$\frak X_w \simeq Z_w/(P\times Q)$.
Define the map $P\times Q\times \cal N_{Q\cap{}^wP}\to Z_w$
by $(p,q,y)\mapsto (qwp^{-1},pw^{-1}\cdot y)$
and let $Q\cap{}^wP$ act on the source by
\[
r\cdot (p,q,y) = (p w^{-1}r^{-1}w, qr^{-1}, r\cdot y).
\]
Under this action, $P\times Q\times\cal N_{Q\cap{}^wP}$
is a $Q\cap{}^wP$-torsor over $Z_w$.
Hence
\[
\frak X_w \simeq \frac{Z_{w}}{P\times Q} 
\simeq \frac{(P\times Q\times \cal N_{Q\cap{}^wP})/(Q\cap{}^wP)}{P\times Q}
\simeq \frac{\cal N_{Q\cap{}^wP}}{Q\cap{}^wP}.
\]
As for $\frak Y_w$, we see by \Cref{thm20} that
$\frak Y_w \simeq \cal N_w/G_w$.

The commutativity of the diagram
of the second part is proved by following
the projection maps through the construction,
a check we leave to the reader.
\end{proof}

Although the map~$f_w$ of \Cref{thm16} is not an isomorphism,
its corresponding $*$-pull $!$-push functor is
an isomorphism on~$\Dbc$ by \Cref{thm24}.
We can now explain how \Cref{thm16} implies \Cref{thm17}.

\begin{proof}[Proof of \Cref{thm17}]
For $i=1,\dots,s$ let
\[
(Q\backslash P/G)_i = Q\backslash Qg_iP/P,
\qquad (Q\backslash P/G)_{\leq i} = Q\backslash \bigsqcup_{j\leq i}Qg_jP/P,
\qquad (Q\backslash P/G)_{<i} = Q\backslash \bigsqcup_{j<i}Qg_jP/P.
\]
We chose the ordering on the~$g_i$ so that
$(Q\backslash P/G)_i$ is an open substack of $(Q\backslash P/G)_{\leq i}$
with closed complement $(Q\backslash G/P)_{<i}$.
By pullback along the natural map $\frak X\to Q\backslash G/P$,
these subobjects give rise to substacks of~$\frak X$:
\[
\frak X_i \defeq \frak X|_{(Q\backslash G/P)_i},
\qquad
\frak X_{\leq i} \defeq \frak X|_{(Q\backslash G/P)_{\leq i}},
\qquad
\frak X_{<i} \defeq \frak X|_{(Q\backslash G/P)_{<i}}.
\]
Let
$\begin{tikzcd}[column sep=small]
\cal N_L/L &
\frak X \lar[swap]{\pi} \rar{\rho} &
\cal N_M/M
\end{tikzcd}$
denote the maps in the correspondence defining
$\res_{L\subseteq P}^G\circ\ind_{L\subseteq P}^G$
and use subscripts in the obvious way
to denote the restrictions
of these maps to the substacks
$\frak X_i$, $\frak X_{\leq i}$,
and $\frak X_{<i}$ defined above.
Let $\cal F\in\Dbc(\cal N_L/L)$.
The distinguished triangle
for the open substack $\frak X_{<i}$
with closed complement $\frak X_i$
gives rise to a distinguished triangle
\[
\rho_{<i!}\pi_{<i}^*(\cal F) \longrightarrow
\rho_{\leq i!}\pi_{\leq i}^*(\cal F) \longrightarrow
\rho_{i!}\pi_{i}^*(\cal F) \longrightarrow
\]
Assuming now that $\cal F$ is perverse,
since parabolic induction and restriction
preserve semisimple objects
these distinguished triangles give rise
to a direct sum decomposition
\[
\res_{L\subseteq P}^G\ind_{L\subseteq P}^G(\cal F)
\simeq \bigoplus_{i=1}^s \rho_{i!}\pi_{i}^*(\cal F).
\]
The proof now follows by unspooling the definitions and using \Cref{thm24}.
\end{proof}

\subsection{\texorpdfstring{$\cal P$}{P}-cuspidal support}
\label{sec:springer:cuspidal}
The definition of cuspidality in the disconnected case
proceeds by restriction to the identity component.

\begin{definition} \label{thm63}
$\cal F\in\Irr\bigl(\Perv(\cal N_G/G)\bigr)$
is \mathdef{cuspidal} if some
(equivalently, every) irreducible summand
of $\res_{G^\circ}^G(\cal F)$ is cuspidal.
\end{definition}

For $G$ disconnected the definition
of cuspidal support is more subtle,
as it depends on a choice
of parabolic subgroups against which
to test cuspidality.
There is no canonical such choice,
so we will include it as an input to the construction.
Suppose we are given for every reductive group~$G$
a set $\cal P_G\subseteq\Par(G)$ of parabolic subgroups of~$G$.
Say the collection $\cal P = (\cal P_G)_G$
is \mathdef{admissible} if for every~$G$,
\begin{enumerate}
\item
the set $\cal P_G$ is stable under $G$-conjugacy and
\item
for every $P,Q\in\cal P_G$ and Levi factor~$L$ of~$P$,
if $Q$ and~$L$ contain a common maximal torus
then $Q\cap L\in\cal P_L$.
\end{enumerate}
The first condition is not controversial and the second condition is needed
for the statement of the Geometrical Lemma to make sense.
Fix such a collection~$\cal P$.

Our goal in this section is to define cuspidal support
in the disconnected case, generalizing the definitions
of the introduction to \Cref{sec:springer}
while taking into account the choice of~$\cal P$.

\begin{lemma} \label{thm22}
Let $P$ be a parabolic subgroup of~$G$
and let $\cal F\in\Irr\bigl(\Perv(\cal N_G/G)\bigr)$.
The following conditions are equivalent.
\begin{enumerate}
\item
$\res_{L\subseteq P}^G(\cal F) = 0$.

\item
$\Hom\bigl(\cal F,\ind_{L\subseteq P}^G(\cal F')\bigr) = 0$
for all $\cal F'\in\Irr\bigl(\Perv(\cal N_G/G)\bigr)$.
\end{enumerate}
\end{lemma}

Here $L$ is the Levi quotient of~$P$.
We say $\cal F$ is \mathdef{$\cal P$-cuspidal}
if it satisfies the equivalent conclusions
of \Cref{thm22} for every $P\in\cal P$
with $P\neq G$.

\begin{proof}
The key facts are the adjunction between
$\res_{L\subseteq P}^G$ and $\ind_{L\subseteq P}^G$
and the semisimplicity of parabolically induced perverse sheaves.
If $\res_{L\subseteq P}^G(\cal F)\neq 0$
then for any irreducible summand
$\cal F'$ of $\res_{L\subseteq P}^G(\cal F)$
we have $\Hom\bigl(\cal F,\ind_{L\subseteq P}^G(\cal F')\bigr) \neq 0$.
Conversely, if $\Hom\bigl(\cal F,\ind_{L\subseteq P}^G(\cal F')\bigr) \neq 0$
then $\res_{L\subseteq P}^G(\cal F)\neq0$.
\end{proof}

Similarly, one has notions of $\cal P$-cuspidal datum
and $\cal P$-cuspidal support.
The notion of induction series does not depend on~$\cal P$,
so is unchanged.

\begin{theorem} \label{thm14}
Let $\cal F\in\Irr\bigl(\Perv(\cal N_G/G)\bigr)$.
There is a unique $G$-conjugacy class
of $\cal P$-cuspidal data $[L,\cal O,\cal E]$ such that
for some parabolic $P\in\cal P$ with Levi quotient~$L$,
\[
\Hom\bigl(\ind_{L\subseteq P}^G\IC(\cal O,\cal E), \cal F\bigr)\neq 0.
\]
\end{theorem}

We call the $G$-conjugacy class $[L,\cal O,\cal E]$
of \Cref{thm14} the \mathdef{$\cal P$-cuspidal support} of~$\cal F$.
Our goal is to prove that it exists and is unique.
Existence is easy, but uniqueness is nontrivial
and relies on the Geometrical Lemma.

\begin{proof}
To prove existence,
if $\cal F = \IC(\cal O,\cal E)$ is already $\cal P$-cuspidal
then its cuspidal support is $(G,\cal O,\cal E)$.
Otherwise, there is $P\in\cal P$
with Levi quotient~$L$ such that
$\res_{L\subseteq P}^G(\cal F)\neq 0$.
By transitivity of parabolic restriction,
we may assume that all further proper
parabolic restrictions of
$\res_{L\subseteq P}^G(\cal F)$ vanish.
Let $\IC(\cal O,\cal E)$ be a simple quotient of this sheaf.
By adjunction, $\cal F$ lies in the induction series
for $(L,\cal O,\cal E)$.

To prove uniqueness, for $i=1,2$
let $P_i$ be a parabolic subgroup of~$G$
with Levi quotient~$L_i$ and let $[L_i,\cal O_i,\cal E_i]$
be $\cal P$-cuspidal data.
Suppose that
\[
\Hom\bigl(\ind_{L_i\subseteq P_i}^G
\IC(\cal O_i,\cal E_i),\cal F\bigr) \neq 0,
\]
so that both pairs are $\cal P$-cuspidal supports of~$\cal F$.
Then
\begin{align*}
0 &\neq
\Hom\bigl(\ind_{L_1\subseteq P_1}^G \IC(\cal O_1,\cal E_1),
\ind_{L_2\subseteq P_2}^G \IC(\cal O_2,\cal E_2)\bigr) \\
&= \Hom\bigl(\res_{L_2\subseteq P_2}^G
\ind_{L_1\subseteq P_1}^G \IC(\cal O_1,\cal E_1),
\IC(\cal O_2,\cal E_2)\bigr).
\end{align*}
By \Cref{thm15,thm17},
there is $g\in G$ such that
$\IC(\cal O_2,\cal E_2)$ admits a nonzero homomorphism from
\[
\ind_{L_2\cap{}^gL_1\subseteq L_2\cap{}^gP_1}^{L_2}
\res_{L_2\cap{}^gL_1\subseteq P_2\cap{}^gL_1}^{{}^gL_1}
\Ad(g^{-1})^*\IC(\cal O_2,\cal E_2).
\]
But since $\IC(\cal O_2,\cal E_2)$ is $\cal P$-cuspidal,
the only way the object above can be nonzero
is if the parabolic restriction is trivial,
that is, if $L_2={}^g L_1$,
so that the outer parabolic induction is also trivial.
Then
\[
\Ad(g)[L_1,\cal O_1,\cal E_1] = [L_2,\cal O_2,\cal E_2],
\]
meaning that the $\cal P$-cuspidal support is well-defined
up to $G$-conjugacy as claimed.
\end{proof}

In the remainder of this section we comment
on the choice of~$\cal P$.
On a first attempt, it seems most natural
to define $\cal P_G$ to be the set $\Par(G)$
of all parabolic subgroups of~$G$,
for which we write $\cal P=\Par$.
However, this choice does not produce
an interesting theory.

\begin{example} \label{thm23}
Suppose $G$ is a finite group and $\cal P=\Par$.
Then an irreducible representation of~$G$
is $\Par$-cuspidal if and only if $G=1$,
and the only $\Par$-cuspidal datum is
$(1,1,\tn{triv})$.
\end{example}

Something similar to \Cref{thm23} happens in general.
Let $\cal P=\Par^\circ$ be the collection $\cal P_G=\Par(G^\circ)$.

\begin{lemma} \label{thm37}
Let $\cal F\in\Irr\bigl(\Perv(\cal N_G/G)\bigr)$
\begin{enumerate}
\item
$\cal F$ is $\Par$-cuspidal if and only if
$\cal F$ is cuspidal and $G=G^\circ$.

\item
Let $\cal P=\Par$ or $\Par^\circ$.
The $\cal P$-cuspidal support of $\cal F$
is the cuspidal support of
any irreducible summand of $\res_{G^\circ}^G(\cal F)$.
\end{enumerate}
\end{lemma}

The second part claims implicitly
that this support is independent of
the choice of irreducible summand.

\begin{proof}
For the non-obvious implication in the first part,
suppose $\cal F$ is $\Par$-cuspidal and $G\neq G^\circ$.
Then $\res_{G^\circ}^G(\cal F) = 0$.
But this forces $\cal F=0$
because $\res_{G^\circ}^G$ is faithful,
a contradiction.

For the second part, first assume $\cal P=\Par$.
Let $\cal F'$
be an irreducible summand of $\res_{G^\circ}^G(\cal F)$.
Then $\Hom\bigl(\cal F,\ind_{G^\circ}^G(\cal F')\bigr)\neq0$.
If $\Hom\bigl(\cal F',\ind_{L^\circ\subseteq P^\circ}^{G^\circ}(\cal F'')\bigr)\neq0$
then $\Hom\bigl(\cal F,\ind_{L^\circ\subseteq P^\circ}^G(\cal F'')\bigr)$
by the transitivity of parabolic induction
and the irreducibility of~$\cal F''$.
All choices of summand~$\cal F'$ yield the same
$G$-conjugacy class of $\Par(G)$-cuspidal data
because all choices of~$\cal F'$ are $G$-conjugate.
The proof for $\cal P=\Par^\circ$ is similar.
\end{proof}

Hence cuspidal representations need not be $\Par$-cuspidal.
However, as soon as $\cal P$ lacks variation
in the components group,
$\cal P$-cuspidality is the same as cuspidality.

\begin{lemma} \label{thm46}
Let $\cal F\in\Irr\bigl(\Perv(\cal N_G/G)\bigr)$,
let $\cal P^\circ \defeq \{P^\circ \mid P\in\cal P\}$,
and let $\cal P\to\cal P^\circ$
be the map $P\mapsto P^\circ$.
\begin{enumerate}
\item
Suppose $\cal P\to\cal P^\circ$ is surjective.
If $\cal F$ is $\cal P$-cuspidal then $\cal F$ is cuspidal.

\item
Suppose $\cal P\to\cal P^\circ$ is injective.
If $\cal F$ is cuspidal then $\cal F$ is $\cal P$-cuspidal.
\end{enumerate}
Consequently, if $\cal P\mapsto\cal P^\circ$
is a bijection then $\cal P$-cuspidality
is the same as cuspidality.
\end{lemma}

\begin{proof}
This follows from the following two observations.
First, $\cal F$ is cuspidal
if and only if $\res_{L^\circ\subseteq P^\circ}^G(\cal F) = 0$
for every parabolic subgroup $P^\circ\subsetneq G^\circ$.
Second,
$\res_{L^\circ\subseteq P^\circ}^G(\cal F) = 0$
if and only if $\res_{L\subseteq P}^G(\cal F) = 0$.
\end{proof}

There are several natural candidates for the collection~$\cal P$
besides the two already discussed, $\cal P$ and $\cal P^\circ$.
The choice of parabolics that carries the most information
is the (evidently admissible) collection $\cal P = \tn{mPar}$
(``m'' for ``maximal'') for which $P\in\tn{mPar}(G)$ if and only if $P = N_G(P^\circ)$.
However, for our application to the local Langlands conjectures,
we need an intermediate collection.

\begin{definition} \label{thm51}
\begin{enumerate}
\item
A Levi subgroup~$L$ of~$G$ is a \mathdef{quasi-Levi subgroup}
if $L = Z_G\bigl(Z^\circ(L^\circ)\bigr)$.
\item
A parabolic subgroup~$P$ of~$G$ is a \mathdef{quasi-parabolic subgroup}
if one (equivalently, any) of its Levi factors is a quasi-Levi subgroup of~$G$.
\end{enumerate}
\noindent
Write $\tn{qPar}$ for the collection of quasi-parabolic subgroups.
\end{definition}

The terminology ``quasi--Levi'' is due to Aubert--Moussaoui--Solleveld
and we have retained it to ease comparison of our work and theirs.
We warn the reader, however, that with our definition of Levi subgroup,
every quasi--Levi subgroup is a Levi subgroup but not conversely.

The assignments $L\mapsto L^\circ$
and $L^\circ\mapsto Z_G\bigl(Z(L^\circ)^\circ\bigr)$
form a $G$-equivariant bijection between
Levi subgroups of~$G^\circ$ and quasi-Levi subgroups of~$G$.
Hence $\tn{qPar}$-cuspidality is the same as cuspidality.
It is also not so hard to see that $\tn{qPar}$ is admissible:
if $L$ and $M$ are quasi-Levis containing a common maximal torus
then
\[
Z_G\bigl(Z^\circ(L)\bigr)\cap Z_G\bigl(Z^\circ(M)\bigr)
= Z_G\bigl(Z^\circ(L)\cdot Z^\circ(M)\bigr)
= Z_G\bigl(Z^\circ(L\cap M)\bigr).
\]
For legibility in the future we will say
``connected-cuspidal'' to mean ``$\Par(G^\circ)$-cuspidal''
and ``quasi-cuspidal'' to mean ``$\tn{qPar}(G)$-cuspidal''.

\begin{remark} \label{thm32}
Using the correspondence between equivariant local systems
and representations of finite groups,
\Cref{thm63} gives a notion of $\cal P$-cuspidality
for pairs $(u,\rho)$ and
\Cref{thm14} defines a $\cal P$-cuspidal support map $\cal P\tn{-Cusp}$
which takes as input a $G$-conjugacy class of pairs $(u,\rho)$
and outputs a $G$-conjugacy class of triples $(L,v,\varrho)$,
where $L$ is some Levi subgroup of~$G$.
To avoid the overhead of perverse sheaves,
this more classical formulation of the cuspidal support map
is the one that we will use in \Cref{sec:bernstein,sec:corr}.

We write $\qcsupp$ for $\tn{qPar-Cusp}$ and
$\connsupp$ for $\Par^\circ\!\tn{-Cusp}$.
By the discussion in \Cref{sec:param:springer},
it makes no difference if~$u$ is a unipotent element of~$G$
or a nilpotent element of $\frak g$.
\end{remark}

\subsection{Twisting}
In \Cref{sec:corr}, we will need to show
the cuspidal support map for $L$-parameters
is independent of various choices,
all of which differ by a twist by
a character of the components group.
In this section, we show that
the cuspidal support map for perverse sheaves
is compatible with such twists,
proving several other compatibilities along the way.

Let $\chi$ be a character
(that is, one-dimensional representation)
of the finite group~$\pi_0(G)$
and let $X$ be a $G$-scheme.
We can interpret $\chi$
as an irreducible perverse sheaf~$\cal E$ on $\tn{pt}/G$,
and then, by pullback along the map $X/G\to\tn{pt}/G$,
as an irreducible perverse sheaf on $X/G$.

\begin{lemma} \label{thm18}
Let $\cal F\in\Dbc(\cal N_G/G)$,
let $\cal E\in\Dbc(\tn{pt}/G)$,
and let $\res^G_L\cal E|_L$ be the pullback
of~$\cal E$ along the map $\tn{pt}/L\to\tn{pt}/G$.
Then
\[
\res_{L\subseteq P}^G\bigl(\cal E\otimes\cal F)
= \res^G_L(\cal E)\otimes\res_{L\subseteq P}^G(\cal F).
\]
\end{lemma}

\begin{proof}
Consider the following commutative diagram
of correspondences.
\[
\begin{tikzcd}
\cal N_G/G \dar{f_G} &
\cal N_P/P \lar[swap]{i}\rar{\pi}\dar{f_P} &
\cal N_L/L  \dar{f_L} \\
\tn{pt}/G &
\tn{pt}/P \lar[swap]{j}\rar{\rho} &
\tn{pt}/L 
\end{tikzcd}
\]
Then
\[
\res_{L\subseteq P}^G(\cal E\otimes\cal F)
\simeq \pi_!i^*\bigl(f_G^*(\cal E)\otimes\cal F\bigr)
\simeq \pi_!\bigl(f_P^*j^*(\cal E)\otimes\cal F\bigr).
\]
Since the projection $P\to L$
induces an isomorphism $\pi_0(P)\to\pi_0(L)$,
it induces an equivalence of categories
$\Dbc(\tn{pt}/P)\simeq\Dbc(\tn{pt}/L)$.
Hence there is $\cal E'$ in $\Dbc(\tn{pt}/L)$
such that $\rho^*\cal E' = j^*\cal E$.
Using the projection formula we get
\[
\pi_!\bigl(f_P^*j^*(\cal E)\otimes\cal F\bigr)
\simeq \pi_!\bigl(\pi^*f_L^*(\cal E')\otimes\cal F\bigr)
\simeq \res^G_L(\cal E)\otimes\res_{L\subseteq P}^G(\cal F).
\qedhere
\]
\end{proof}

\begin{proposition} \label{thm64}
Let $\chi$ be a character of~$\pi_0(G)$,
let $X$ be a $G$-variety,
and let $\cal F\in\Irr\bigl(\Perv(\cal N_G/G)\bigr)$.
If $[L,\cal O,\cal E]$ is the cuspidal support of~$\cal F$
then $[L,\cal O,\chi|_L\otimes\cal E]$
is the cuspidal support of $\chi\otimes\cal F$.
\end{proposition}

\begin{proof}
We compare the Hom groups
(taken in the category of perverse sheaves)
that witness the cuspidal support
before and after twisting by~$\chi$.
After twisting, the group is
\[
\Hom\bigl(\chi\otimes\cal F,
\ind_{L\subseteq P}^G \IC(\cal O,\chi|_L\otimes\cal E)\bigr)
\simeq \Hom\bigl(\res_{L\subseteq P}^G(\chi\otimes\cal F),
\IC(\cal O,\chi|_L\otimes\cal E)\bigr)
\]
By \Cref{thm18} and a basic compatibility for IC sheaves,
\begin{align*}
\Hom\bigl(\res_{L\subseteq P}^G(\chi\otimes\cal F),
\IC(\cal O,\chi|_L\otimes\cal E)\bigr)
& \simeq\Hom\bigl(\res^G_L(\chi)\otimes\res_{L\subseteq P}^G(\cal F),
\res_L^G(\chi)\otimes\IC(\cal O,\cal E)\bigr) \\
& \simeq\Hom\bigl(\res_{L\subseteq P}^G(\cal F), \IC(\cal O,\cal E)\bigr) \\
& \simeq\Hom\bigl(\cal F,\ind_{L\subseteq P}^G(\IC(\cal O,\cal E))\bigr).
\end{align*}
As these Hom groups are isomorphic,
the cuspidal support of the twist is the twist
of the cuspidal support.
\end{proof}

\section{Parameterizing induction series}
\label{sec:param}
In \Cref{sec:springer}, we organized the elements of
$\Perv(\cal N_G/G)$ into induction series.
The next step is to parameterize these series.
In the connected generalized Springer correspondence,
the induction series are parameterized by
representations of Weyl groups,
and in the generalization of \cite{AMS18} to disconnected groups,
the series are parameterized by representations of twisted Weyl groups.
But in the full generality of $\cal P$-cuspidal support,
it is unclear how to parameterize induction series in such a way.

\begin{example} \label{thm78}
In the extreme case when $G^\circ = 1$,
so that the Levi subgroup $L$
is simply a finite subgroup of the finite group~$G$,
describing the induction series for $L$
amounts to describing the isotypic components of
representations induced from $L$ to~$G$,
in other words, the irreducible representations
of the algebra $\End(\Ind_L^G \rho)$
where $\rho$ is an irreducible representation of $L$
(ignoring the question of whether or not $\rho$ is cuspidal).
When $L$ is normal in $G$ this algebra is a twisted group ring,
but we do not know of such a description in general.

There are several objections one could raise about this example,
for instance, that it does not fit into a class $\cal P$
satisfying the hypotheses of \Cref{thm46}.
Nonetheless, we expect to see similar behavior
whenever, for example, $\pi_0(L)$ is not normal in $\pi_0(G)$.
\end{example}

Due to the phenomenon of \Cref{thm78},
after first giving a rough parameterization of the induction series,
we will restrict our attention to the quasi-Levi subgroups of \cite{AMS18}
and construct a parameterization in this case by comparing with that earlier work.

The parameterization arises by first performing
parabolic induction on a regular stratum
where parabolic induction becomes the pushforward
of a local system along a finite étale covering,
meaning that it is relatively easy to compute
the corresponding regular induction series.
One then compares this regular parabolic induction
with the nilpotent parabolic induction from before.
There are at least two ways to effect the comparison,
but the one that we find the most natural,
and which requires us to work on the Lie algebra,
uses the Fourier--Laumon transform.
The procedure is expressed in the following schematic,
where for simplicity we assume the Levi is a torus.
Here $\frak g_\tn{rs}$ is the regular semisimple locus
and $\frak t_{G\tn{-reg}}$ is defined in \Cref{sec:param:regular}.
\begin{equation} \label{thm86}
\begin{tikzcd}[row sep=large]
\Perv(\tn{pt}/T) \arrow[r,leftrightsquigarrow,"\Four_{\frak t}"]
 \arrow[d,squiggly,"\ind_{T\subseteq B}^G"'] &
\Perv(\frak t/T) &
\Loc(\frak t_{G\tn{-reg}}/T) \arrow[l,"\IC"']
 \arrow[d,squiggly,"\;\tn{\parbox{2cm}{(finite-étale pushforward)}}","\uind_{T\subseteq B}^G"'] \\
\Perv(\cal N_G/G) \arrow[r,leftrightsquigarrow,"\Four_{\frak g}"] &
\Perv(\frak g/G) &
\Loc(\frak g_\tn{rs}/G). \arrow[l,"\IC"']
\end{tikzcd}
\end{equation}

After the formalism is set up,
the main technical challenge is to understand the structure of the endomorphism ring
of a parabolically induced local system on the regular locus;
after all, it is the irreducible representations of this algebra
that parameterize the induction series.
Even in the connected case,
there is something nontrivial to check here when the Levi is not a torus
\cite[9.2(d)]{lusztig84b}.
Due to \Cref{thm78},
we cannot expect this task to be entirely straightforward in the disconnected case,
and indeed, the construction of the parameterization in \cite{AMS18} is rather subtle.
Rather than redeveloping the parameterization
by a close study of the relevant finite étale covering,
we have chosen to bootstrap from \cite[5.4]{AMS18},
using their result to exhibit the desired parameterization of induction series.

There are several ways one might carry out this bootstrap,
for example, by rewriting several portions of \cite{lusztig84b}
in the disconnected setting.
We found it the most illuminating to proceed via a careful comparison
of the generalized Springer correspondence on the Lie algebra and on the group,
a comparison that we hope is of independent interest.
On the unipotent and nilpotent loci one has a direct comparison:
these loci are equivariantly isomorphic via a Springer isomorphism,
as we recall in \Cref{sec:param:springer}.
On the regular loci the comparison is more difficult:
in general, the regular semisimple loci of a group and its Lie algebra
are not isomorphic, equivariantly or otherwise.
Nonetheless, in \Cref{sec:param:regular}
we carry out the comparison using quasi-logarithms,
themselves reviewed in \Cref{sec:param:quasi-log}.
The bootstrap is completed in \Cref{sec:param:fibers}.
To conclude, we explain in \Cref{sec:param:comparison}
why our notion of connected cuspidal support
agrees with that of Lusztig and Aubert--Moussaoui--Solleveld.

\subsection{The Springer isomorphism}
\label{sec:param:springer}
Our formulation of the Springer correspondence
uses the nilpotent cone~$\cal N_G$,
but in the earlier work of \cite{AMS18}
and in the application to the local Langlands correspondence
what appears is the unipotent variety $\cal U_G$,
the closed subvariety of unipotent elements of~$G$.
In this section we indicate why one
may pass freely between the two settings.
In brief, this passage is possible because
the two varieties are equivariantly isomorphic
and the resulting comparison of sheaf theory
is independent of the choice of isomorphism.

\begin{proposition} \label{thm41}
There is a nonempty connected variety~$Y$ and a morphism
\[
f\colon Y\times\cal U_G \to \cal N_G
\]
satisfying the following property:
the assignment
\[
Y \to \Isom_G(\cal U_G,\cal N_G) \colon
y \mapsto f|_{y\times\cal U_G}
\]
is an isomorphism between $Y$
and the space of $G$-equivariant
isomorphisms $\cal U_G\simeq\cal N_G$.
\end{proposition}

Such an isomorphism is called a \mathdef{Springer isomorphism}.
As \Cref{thm39} shows, a Springer isomorphism
for $G^\circ$ need not be $G$-equivariant.

\begin{proof}
First, assume $G=G^\circ$.
Using the normality and affineness of $\cal U_G$ and~$\cal N_G$
together with the algebraic Hartogs's Lemma
\cite[\href{https://stacks.math.columbia.edu/tag/031T}{Tag 031T}]{stacks},
we can replace $\cal U_G$ and $\cal N_G$ by their regular loci,
which are homogeneous spaces for~$G^\circ$.
Hence any Springer isomorphism is uniquely determined
by the image~$X$ of a fixed regular unipotent element~$u$,
which must furthermore lie in $\Lie(Z_G(u))$.
We may then take $Y = N_G\bigl(Z_G(u)\bigr)/Z_G(u)$
with the map $f\colon Y\times\cal U_G\to\cal N_G$
defined by $f(y,gu)=gy^{-1}X$.
Note that $f$ depends on the pair $(u,X)$.
For more details of this case,
see Serre's appendix to \cite{mcninch05}.

For the general case of disconnected~$G$
we may assume $G^\circ \subseteq G \subseteq
\Aut(G_\tn{ad}^\circ)$.
Choose a pinning~$\cal P$ of~$G^\circ$
and let $A$ be the subgroup
of $\Aut(G^\circ,\cal P)$
corresponding to~$\pi_0(G)$.
From the pinning~$\cal P$ we can form
the regular unipotent element
\[
u_{\cal P } \defeq \prod_{\alpha\in\Delta} \exp_\alpha(X_\alpha)
\]
and the regular nilpotent element
\[
X_{\cal P } \defeq \sum_{\alpha\in\Delta} X_\alpha.
\]
The ($G^\circ$-equivariant)
Springer isomorphism that sends $u_{\cal P}$ to~$X_{\cal P}$
is evidently $A$-equivariant, hence $G$-equivariant.
Using the pair $(u_{\cal P}, X_{\cal P})$
to define the map $f\colon Y'\times\cal U_G\to\cal N_G$, where 
$Y' := N_{G^{\circ}}(Z_{G^{\circ}}(u))/Z_{G^{\circ}}(u)$, we see first
that $A$ acts on $Y'$ by group automorphisms, simply because 
$A$ fixes~$u$, and second that the $G^\circ$-equivariant map 
$f|_{y\times\cal U_G}$ is $A$-equivariant, hence $G$-equivariant,
if and only if $y\in (Y')^A$. So $(Y')^A$ is evidently nonempty,
and, via setting $Y = (Y')^{A}$, this proves the result except for the claim 
that $(Y')^{A}$ is connected.

The group~$Y'$ is solvable
and its maximal reductive quotient
$Y'_\tn{red}$ is a torus of rank one
\cite[Theorem~B]{mcninch_testerman09}.
At the same time, the cocharacter
$2\rho^\vee\colon\bbG_{\tn m}\to T$
is fixed by~$A$ and
preserves the root line containing~$u$
(which is well-defined because
we work in characteristic zero)
and thus lies in $N_G\bigl(Z_G(u)\bigr)$.
So $(Y')_{\tn{red}}^A = Y'_\tn{red}$
and it suffices to show
that the $A$-fixed points
on the unipotent radical $\Ru(Y')$ are connected.
This is a general fact about
finite group actions on unipotent groups
in characteristic zero
which one proves by induction
on the nilpotency class of~$G$.
\end{proof}

\begin{example} \label{thm39}
In general, a $G^\circ$-equivariant Springer isomorphism
need not be $G$-equivariant.
For one example, take $G=H^n\rtimes S_n$ where
$H$ is a (nontrivial) simple group,
then use a different Springer isomorphism
on each simple factor of~$G^\circ$.
But even when $G_\tn{ad}$ is simple,
$G$-equivariance can fail.
For $G=\GL_n$, every Springer isomorphism
is of the form
\[
\sigma_p\colon 1 + e \mapsto e\cdot p(e)
\]
where $e$ is nilpotent and
$p$ is a polynomial of degree
$\leq n-2$ with $p(0)\neq0$.
The map~$\sigma_p$ is equivariant
for the outer automorphism $g \mapsto (g^{-1})^t$
if and only if $\sigma_p(u^{-1}) = -\sigma_p(u)$,
and one can check by induction on~$n$
that such a $\sigma_p$ is uniquely
determined by the leading term~$p(1)$.
So there are $n-1$ dimensions worth of Springer isomorphisms
but only $1$ dimension of them are $\Aut(G)$-equivariant.
\end{example}

\begin{corollary} \label{thm38}
Let $\sigma\colon\cal U_G\to\cal N_G$ be a $G$-equivariant isomorphism.
The resulting identifications are independent of the choice of~$\sigma$.
\begin{enumerate}
\item
The bijection between unipotent and nilpotent conjugacy classes.

\item
For each nilpotent orbit~$\cal O$, the equivalence of categories
$\Perv\bigl(\sigma^{-1}(\cal O)/G\bigr)
\simeq\Perv(\cal O/G)$.
\end{enumerate}
\end{corollary}

Note, however, that if $G^\circ$ is not a torus
then different choices of~$\sigma$
can induce inequivalent isomorphisms
$\cal U_G/G\simeq\cal N_G/G$ of quotient stacks,
simply because $N_G\bigl(Z_G(u)\bigr)\supsetneq Z_G(u)$
in this case.

To finish this section, we'll show that
parabolic induction (and therefore restriction)
is compatible with the identification of \Cref{thm38}.

\begin{proposition} \label{thm40}
Let $P$ be a parabolic subgroup of~$G$
with Levi quotient~$L$.
There is a commutative diagram
\[
\begin{tikzcd}
L \dar[phantom]{\actson} &
P \lar\rar \dar[phantom]{\actson} &
G \dar[phantom]{\actson} \\
\cal U_L \dar{\sigma_L} &
\cal U_P \lar\rar\dar{\sigma_P} &
\cal U_G \dar{\sigma_G} \\
\cal N_L &
\cal N_P \lar\rar &
\cal N_G
\end{tikzcd}
\]
of varieties with group action,
in which each horizontal arrow
is equivariant for the homomorphism above it
and each vertical arrow is an equivariant
isomorphism for the group above it.
\end{proposition}

The stack $\cal N_P/P$
is of a different nature than $\cal N_G/G$:
it often has infinitely many points \cite{rohrle96}.

\begin{proof}
Using \Cref{thm41},
choose $G$-equivariant isomorphism
$\sigma_G\colon\cal U_G\to\cal N_G$.
Let $\sigma_P \defeq \sigma_G|_P$,
a $P$-equivariant isomorphism.
Then $\sigma_P$ induces
a $P$-equivariant isomorphism $\cal U_L\to\cal N_L$.
The resulting diagram clearly commutes.
\end{proof}

It follows from \Cref{thm38,thm40} that
the functors of parabolic restriction and parabolic induction,
and hence the generalized Springer correspondence,
agree with each other in the nilpotent and unipotent settings.
So for many purposes we may safely ignore
the difference between $\cal N_G$ and $\cal U_G$.

\subsection{Quasi-logarithms}
\label{sec:param:quasi-log}

In this section we recall a tool to compare
the Springer correspondence on the group and the Lie algebra.
The following definition is due to Bardsley and Richardson
\cite[9.3]{bardsley_richardson85}
and was further developed by Kazhdan and Varshavsky
\cite[1.8]{kazhdan_varshavsky06}.

\begin{definition}
A \mathdef{quasi-logarithm} is a $G$-equivariant morphism of varieties
$\lambda\colon G\to\frak g$
such that $\lambda(1)=0$ and $\Lie(\lambda) = \tn{id}_{\frak g}$.
\end{definition}

In characteristic zero,
every (possibly disconnected) reductive group admits a quasi-logarithm,
as one proves by giving an explicit formula arising from a faithful representation of~$G$
\cite[1.8.2, 1.8.10]{kazhdan_varshavsky06}.

\begin{example}
Let $G=\GL_n$.
Every function of the form $g \mapsto 1 + g + g^2\cdot f(g)$
with $f$ a polynomial is a quasi-logarithm.
Consequently, unlike for Springer isomorphisms,
there are usually infinitely many dimensions of quasi-logarithms.
\end{example}

One pleasant property of quasi-logarithms is that they easily transfer:
often a quasi-logarithm on a group induces a quasi-logarithm,
or other equivariant map, on certain subgroups or conjugacy classes.
We record two instances of this phenomenon for later use.

\begin{lemma} \label{thm83}
Let $\lambda\colon G\to\frak g$ be a quasi-logarithm.
\begin{enumerate}
\item
The map $\lambda$ restricts to a $G$-equivariant
quasi-logarithm $Z(G^\circ)\to\frak z(G^\circ)$.

\item
For every Levi subgroup $L\subseteq G^\circ$ the restriction $\lambda|_L$
is an $N_G(L^\circ)$-equivariant quasi-logarithm $L^\circ\to\frak l$.
\end{enumerate}
\end{lemma}

Like the Lie algebra,
a quasi-logarithm contains no information
about the non-identity components of~$G$.

\begin{proof}
For the first part, the $G$-equivariance of $\lambda$
implies that $\lambda(Z(G^\circ))\subseteq\frak z(G^\circ)$,
since $Z^\circ_G(x) \subseteq Z^\circ_G(\lambda(x))$ for all $x\in G$.
For the second part, we need only show that $\lambda(L)\subseteq\frak l$.
There is a semisimple element $s\in G^\circ$ such that $Z_{G^\circ}(s) = L$
and $\frak z_G(s) = \frak l$.
But because $\lambda$ is $G$-equivariant,
$\lambda(L)\subseteq \frak z_G(s) = \frak l$.
\end{proof}

\begin{lemma} \label{thm82}
Let $\lambda\colon G\to\frak g$ be a quasi-logarithm.
\begin{enumerate}
\item
$\lambda$ restricts to give a commutative square
\[
\begin{tikzcd}
\cal U_G\times Z(G^\circ) \rar{\lambda}\dar{\tn{pr}_2} &
\cal N_G\times\frak z(G^\circ) \dar{\tn{pr}_2} \\
Z(G^\circ) \rar{\lambda} & \frak z(G^\circ).
\end{tikzcd}
\]
\item
After possibly restricting to an open neighborhood of $0$ in~$\frak z(G^\circ)$,
this square is Cartesian and the horizontal arrows are finite étale.
\end{enumerate}
\end{lemma}

In the second part, ``restricting'' is shorthand for pulling back
the entire commutative square along the open embedding.

\begin{proof}
For the first part, it suffices to show that 
$\lambda(z\cal U_G) \subseteq \lambda(z) + \cal N_G$ for all $z\in Z(G^\circ)$,
or in other words, fixing~$z$, that
$\lambda(zu) - \lambda(z) \in \cal N_G$ for every $u\in \cal U_G$.
In fact, after choosing a Borel subgroup~$B$ of~$G$ containing~$u$,
we see from \cite[1.8.3(b)]{kazhdan_varshavsky06} applied to~$B$
that an even stronger statement holds:
$\lambda(zu) - \lambda(z)$ lies in the Lie algebra
of the unipotent radical of~$B$.

For the second part, since $\lambda|_{Z(G)}$ is a quasi-logarithm,
the bottom arrow becomes finite étale
after a possible restriction on the target.
Moreover, for every $z\in Z(G^\circ)$ the map
$\lambda(z\cal U_G)\to\lambda(z)+\cal N_G$
is an isomorphism because it is $G$-equivariant,
as one sees using the construction of the Springer isomorphism in \Cref{thm41}.
It follows that the restricted square is Cartesian.
\end{proof}

\subsection{Parabolic induction on regular loci}
\label{sec:param:regular}
In this section we define certain regular loci of $\frak g$
(and $G$) and observe that parabolic induction on these loci
is pushforward along a finite étale covering.
Since the effect of such a pushforward on a local system
can be understood purely in terms of representations of finite groups,
it is not surprising that induction series in this setting
are parameterized by endomorphism algebras of induced representations.

Following Lusztig, call an element $x\in L$ (or $X\in\frak l$) 
\mathdef{$G$-regular}
if $Z_G^\circ(x_\tn{s}) \subseteq L$
(or $Z_G(X_\tn{s}) \subseteq\frak l$),
where $(-)_\tn{s}$ denotes the semisimple part
in the Jordan decomposition.
Let $L_{G\tn{-reg}}$ (or $\frak l_{G\tn{-reg}}$)
denote the subset of $G$-regular elements.

\begin{lemma} \label{thm73}
Let $X$ be a locally closed subvariety
of $L_{G\tn{-reg}}$ or $\frak l_{G\tn{-reg}}$
and let $\pi\colon P\to L$ be the standard projection.
Then the induced map
\[
\pi^{-1}(X)/P \to X/L
\]
is an isomorphism.
\end{lemma}

\begin{proof}
We focus on the group case;
the proof for the Lie algebra is similar.
To fix ideas, identify $L$ with a Levi subgroup of~$P$
and identify $\pi^{-1}(X)$ with 
$U\cdot X = \{u\cdot x : u\in U, x\in X\}$.
It suffices to prove the claim after pullback to $X$,
meaning, by \Cref{thm20}, that the map
\[
(U\cdot X)/U \to X
\]
is an isomorphism.
For this claim we may prove that the map $U\cdot X\to X$
is a $U$-torsor for the conjugation action on the source,
and for that, it suffices to show that for every $x\in X$,
the coset $Ux$ is a $U$-torsor under conjugation.
By the equation
\[
uxu^{-1} = [u,x] x,
\]
this amounts to the assertion that the map
$u\mapsto [u,x]$, an endomorphism of~$U$,
is bijective.
Since $x$ is $G$-regular,
this endomorphism is injective,
and thus, by the Ax--Grothendieck theorem
\cite[10.4.11]{ega4-3}, bijective.
\end{proof}

It follows from \Cref{thm73} that the correspondence
\eqref{thm71} (and similarly, \eqref{thm72})
simplifies greatly over the $G$-regular locus,
where the map $\frak p/P\to\frak l/L$
becomes an isomorphism.
The other map in the correspondence simplifies
as well over this locus,
after a possible refinement that we now explain.
Following Lusztig, given a unipotent conjugacy class
(or nilpotent orbit) $\cal O$ of $L$,
let $Y_{(L,\cal O)}$ (or $Y_{(\frak l,\cal O)}$)
be the $G$-orbit of $\cal O\times Z(G^\circ)_{G\tn{-reg}}$
(or $\cal O\times\frak z(G^\circ)_{G\tn{-reg}}$).

\begin{lemma} \label{thm74}
Let $\cal O$ be a nilpotent orbit of~$L$.
Then the following diagram commutes,
where all maps are induced by
the obvious maps of varieties with group action:
\[
\begin{tikzcd}[column sep=tiny]
& (\cal O\times \frak z(L^\circ)_{G\tn{-reg}})/L
\dlar\drar
& \\
(\cal O\times\frak z(L^\circ)_{G\tn{-reg}})/N_G(L^\circ,\cal O)
\arrow[rr,"\sim"] & & Y_{(\frak l,\cal O)}/G.
\end{tikzcd}
\]
If $\pi_0(L)$ is normal in $\pi_0(N_G(L^\circ,\cal O))$
then the vertical maps are Galois coverings with group
$N_G(L,\cal O)/L$.
\end{lemma}

The same holds on the group instead of the Lie algebra;
we omit the standard proofs of both.
Note that if $\pi_0(L)$ is normal in $\pi_0(N_G(L^\circ,\cal O))$
then $N_G(L,\cal O) = N_G(L^\circ,\cal O)$.

We now turn to the comparison between
parabolic induction on the group and its Lie algebra
over the regular loci, using the quasi-logarithms.
Since quasi-logarithms do not induce an isomorphism in this context,
so that we lose control over their precise description,
we resort instead to the following general lemma
about local systems on algebraic stacks.

\begin{lemma} \label{thm79}
Let $W$ be a finite group and let
\[
\begin{tikzcd}
\frak X' \rar{\lambda} \dar{p'} & \frak X \dar{p} \\
\frak Y' \rar{\mu} & \frak Y
\end{tikzcd}
\]
be a commutative square of irreducible Noetherian algebraic stacks
in which the vertical arrows are $W$-torsors and
the horizontal arrows are $W$-equivariant and generically finite étale.
Let $\cal E\in\Loc(\frak X)$.
Then
\begin{enumerate}
\item
the square is Cartesian and

\item
The map $\mu^* \colon
\End(p_*\cal E) \to \End(p'_*\lambda^*\cal E)$
is an isomorphism if and only if
$N_W(\cal E) = N_W(\lambda^*\cal E)$.
\end{enumerate}
\end{lemma}

The map in the second part uses the first part
and the base change isomorphism
$\mu^*p_*\cal E\simeq p'_*\lambda^*\cal E$.

\begin{proof}
For the first part, we can check if the square is Cartesian
after pullback to a chart of $\frak Y$ in schemes,
reducing to the case where the four stacks
are representable in schemes and $p$ and $p'$ are trivial $W$-torsors.
The resulting square is Cartesian by the $W$-equivariance of~$\lambda$.

For the second part,
since finite étale morphisms of algebraic stacks are open,
after pulling back the square to a nonempty open substack of~$\frak Y$
we may ensure that the horizontal morphisms are finite étale.
By \Cref{thm80} this reduction has no effect
on $\End(p_*\cal E)$ and $\End(p'_*\lambda^*\cal E)$.
Using the description from \Cref{thm81}
of local systems in terms of representations of étale fundamental groups,
we are reduced to the case where the given square is of the form
\[
\begin{tikzcd}
\tn{pt}/G' \rar\dar & \tn{pt}/G \dar \\
\tn{pt}/H' \rar & \tn{pt}/H,
\end{tikzcd}
\]
with all groups finite and $G$ is,
say, the quotient of $\pi_1^\tn{ét}(\frak X)$
by a certain open subgroup (depending on~$\cal E$)
with similar definitions for $G'$, $H$, and $H'$.
Here, moreover, the maps of quotient stacks
are induced by injective group homomorphisms
fitting into the following map of short exact sequences:
\[
\begin{tikzcd}
1 \rar & G' \rar[hookrightarrow]\dar[hookrightarrow] &
G \rar\dar[hookrightarrow] & W \rar\dar[equals] & 1 \\
1 \rar & H' \rar[hookrightarrow] & H \rar & W \rar & 1.
\end{tikzcd}
\]

Let $\rho$ be the representation of $G$ corresponding to~$\cal E$.
The problem is now reduced to showing that the restriction map
\[
\End\bigl(\Ind_H^G(\rho)\bigr)
\to \End\bigl(\Res^H_{H'}\Ind_H^G(\rho)\bigr)
\]
is an isomorphism if and only if
$N_W(\rho) = N_W(\Res^H_{H'}(\rho))$.
This claim follows from Frobenius reciprocity.
\end{proof}

Everything is now in place for the comparison.
Given a nilpotent orbit $\cal O$ of~$L$
and $\cal E\in\Loc(\cal O/L)$, let
\begin{equation}\label{thm97}
\begin{aligned}
\infl(\cal E) &\defeq
\cal E\boxtimes \underline{\overline\bbQ_\ell}_{\frak z(L^\circ)/L}
\in\Loc\bigl((\cal O\times\frak z(L^\circ))/L\bigr) \\
\infl_{G\tn{-reg}}(\cal E) &\defeq
\cal E\boxtimes \underline{\overline\bbQ_\ell}_{\frak z(L^\circ)_{G\tn{-reg}}/L}
\in\Loc\bigl((\cal O\times\frak z(L^\circ)_{G\tn{-reg}})/L\bigr),
\end{aligned}
\end{equation}
where $\underline{\overline\bbQ_\ell}_{\frak X}$
is the constant sheaf on~$\frak X$.
Define $\Infl_{G\tn{-reg}}(\cal E)$ and $\Infl(\cal E)$
similarly on the group.

\begin{theorem} \label{thm84}
Let $\lambda\colon G\to\frak g$ be a quasi-logarithm.
Let $\cal O$ be a unipotent orbit of~$L$,
let $u\in \cal O$,
and let $\cal E\in\Loc(\lambda(\cal O)/L)$ be irreducible. 
Suppose that $\pi_0(L)$ is a normal subgroup of $\pi_0(N_G(L^\circ,\cal O))$.
Then the map
\[
\lambda^* \colon
\End\bigl(\uind_{L\subseteq P}^G(\infl_{L\tn{-reg}}(\cal E))\bigr) \to
\End\bigl(\uInd_{L\subseteq P}^G(\Infl_{L\tn{-reg}}(\lambda^*\cal E))\bigr)
\]
is an isomorphism.
\end{theorem}

\begin{proof}
Consider the following commutative diagram, using \Cref{thm82,thm83}:
\[
\begin{tikzcd}
(\cal O\times Z(L^\circ)_{G\tn{-reg}})/L \rar{\lambda}\dar &
(\lambda(\cal O)\times \frak z(L^\circ)_{G\tn{-reg}})/L \dar \\
Y_{L,\cal O}/G \rar{\lambda} &
Y_{\frak l,\lambda(\cal O)}/G.
\end{tikzcd}
\]
By \Cref{thm73}, pushforward along the vertical arrows
computes $\uInd_{L\subseteq P}^G$ and $\uind_{L\subseteq P}^G$.
By \Cref{thm74}, these vertical arrows are torsors
for the finite group $\pi_0(N_G(L^\circ,\cal O))/\pi_0(L)$,
and the horizontal arrows are generically finite étale
because $\lambda$ is étale at the identity.
Moreover, $\cal E$ and $\lambda^*\cal E$
have the same normalizer in this finite group.
Hence we may apply \Cref{thm79}
to the sheaf $\infl_{G\tn{-reg}}(\cal E)$,
concluding the proof.
\end{proof}

\begin{theorem} \label{thm87}
Let $\cal O$ be a nilpotent orbit of~$L$,
let $q\cal E\in\Loc(\cal O/L)$ be cuspidal,
and let $W(G,L)_{(\cal O,q\cal E)}$
be the stabilizer of $(\cal O,q\cal E)$ in~$W(G,L)$.
If $L$ is a quasi-Levi subgroup of~$G$ then there exists a cocycle
$\kappa_{(\cal O,q\cal E)}\in Z^2\bigl(W(G,L)_{(\cal O,q\cal E)},\overline\bbQ_\ell^\times\bigr)$
such that
\begin{equation}\label{thm90}
\End\bigl(\uind_{L\subseteq P}^G(\infl_{G\tn{-reg}}(q\cal E))\bigr)
\simeq \overline\bbQ_\ell\bigl[W(G,L)_{(\cal O,q\cal E)},\kappa_{(\cal O,q\cal E)}\bigr].
\end{equation}
\end{theorem}

The $q$ in $q\cal E$ reflects our assumption that $L$
is a quasi-Levi and is written for compatibility
with notation from \cite{AMS18}.

\begin{proof}
The analogue of this result on the group is \cite[5.4]{AMS18}.
Our version on the Lie algebra follows from this analogue by \Cref{thm84}. 
\end{proof}

\begin{remark}
We will call $\kappa_{(\cal O,q\cal E)}$ an \mathdef{AMS cocycle}
and \eqref{thm90} an \mathdef{AMS isomorphism}
because both stem ultimately from \cite{AMS18}.
By \cite[5.3(b)]{AMS18}, the AMS cocycle $\kappa_{(\cal O,q\cal E)}$
is inflated from $W(G,L)_{(\cal O,q\cal E)}/W(G^\circ,L^\circ)$.
Moreover, the AMS isomorphism is uniquely determined up to
\begin{itemize}
\item
a choice of quasi-logarithm of~$G$ and
\item
automorphisms of the twisted group algebra
arising from characters of $W(G,L)_{(\cal O,q\cal E)}/W(G^\circ,L^\circ)$.
\end{itemize}
The second dependence is stated in \cite[5.4]{AMS18}.
The first dependence stems from the technique of our comparison.
We would be surprised if two different quasi-logarithms
resulted in different AMS isomorphisms.
\end{remark}

As \eqref{thm86} indicates,
intersection cohomology is the tool to pass our sheaves
from the regular locus to (a closed subvariety of) the full Lie algebra.
Fortunately, this passage commutes with parabolic induction.

\begin{lemma} \label{thm85}
Let $\cal O$ be a nilpotent orbit of~$L$ and let $\cal E\in\Loc(\cal O/L)$.
There exists a canonical isomorphism
\[
\IC\bigl(Y_{(\frak l,\cal O)},\ind_{L\subseteq P}^G
\infl_{G\tn{-reg}}(\cal E)\bigr)
\simeq \ind_{L\subseteq P}^G
\bigl(\IC(\cal O\times\frak z(L^\circ),\infl(\cal E))\bigr).
\]
\end{lemma}

\begin{proof}
This claim is the disconnected analogue of
\cite[Proposition~2.17]{achar_henderson_juteau_riche16}.
The argument given there reduces the problem to checking
that the righthand sheaf satisfies the conditions
defining intersection cohomology sheaves.
These conditions depend only on the underlying $G^\circ$-equivariant sheaf,
reducing the claim to the known connected case.
\end{proof}

\subsection{The parameterization}
\label{sec:param:fibers}
We now turn to the parameterization of induction series,
relating parabolic induction on the regular loci of \Cref{sec:param:regular}
and our original parabolic induction on nilpotent cones.
The key operation relating the two types of parabolic induction
is the \mathdef{Fourier--Laumon transform} of \cite{laumon03},
a self-equivalence of categories
\[
\Four_{\frak g}
\colon \Dbc\bigl(\frak g/(\bbG_\tn{m}\times G)\bigr)
\simeq \Dbc\bigl(\frak g/(\bbG_\tn{m}\times G)\bigr)
\]
discussed at length in \cite[6.9]{achar21}
(and depending on an implicit choice of
$G$-equivariant linear isomorphism $\frak g\simeq\frak g^*$
which we suppress).
As the discussion around \cite[8.2.7]{achar21} explains,
the fully-faithful forgetful functor
\[
\Perv\bigl(\cal N_G/(\bbG_\tn{m}\times G)\bigr)
\longrightarrow \Perv(\cal N_G/G)
\]
induces a bijection of isomorphism classes of simple objects,
so in practice, we may mostly ignore the $\bbG_\tn{m}$-action.
We warn the reader that the sources we cite
will sometimes use analogues of the Fourier--Laumon transform
in other settings---for instance, in
\cite{achar_henderson_juteau_riche16},
the Fourier--Sato transform---but the formal structure
is similar enough that the properties we need
carry over to our setting.

The starting point is the fact that
the Fourier transform of a cuspidal perverse sheaf is again cuspidal, up to an inflation.
The case where $G=T$ is a torus shows why inflation is necessary:
the Fourier--Laumon transform of the skyscraper sheaf at the origin of~$\frak t$,
a cuspidal sheaf, is the constant sheaf on~$\frak t$.

\begin{theorem} \label{thm75}
Let $\cal O$ be a nilpotent orbit of~$G$
and let $\cal E\in\Loc(\cal O/G)$ be cuspidal.
\begin{enumerate}
\item
There is a (unique) cuspidal $\cal E'\in\Loc(\cal O/G)$ such that
$\IC(\cal O,\cal E) \simeq \Four_{\frak g}(\IC(\cal O\times\frak z(G^\circ),\infl(\cal E')))$.

\item
$\Res^G_{G^\circ}(\cal E)\simeq \Res^G_{G^\circ}(\cal E')$.
\end{enumerate}
\end{theorem}

\begin{proof}
Both parts follow from the special case where $G=G^\circ$:
the first claim is \cite[Corollary~2.12]{achar_henderson_juteau_riche16}
and the second is \cite[5(b)]{lusztig_utrecht}.
\end{proof}

We expect that one can even take $\cal E=\cal E'$ in \Cref{thm75}:

\begin{conjecture} \label{thm91}
If $\cal E\in\Perv(\cal N_G/G)$ is cuspidal then
$\IC(\cal O,\cal E) \simeq \Four_{\frak g}(\IC(\cal O\times\frak z(G^\circ),\infl(\cal E)))$.
\end{conjecture}

Everything is now in place to abstractly parameterize induction series.

\begin{lemma} \label{thm88}
Let $\cal O$ be a nilpotent orbit of~$L$,
let $\cal E\in\Loc(\cal O/L)$ be cuspidal,
and (by \Cref{thm75}) let $\cal E'\in\Loc(\cal O/L)$
be the cuspidal local system such that
\[
\IC(\cal O,\cal E) \simeq \Four_{\frak l}(\IC(\cal O\times\frak z(L^\circ),\infl(\cal E'))).
\]
Then the Fourier--Laumon transform induces
an isomorphism of endomorphism algebras
\[
\End\bigl(\uind_{L\subseteq P}^G(\infl_{G\tn{-reg}}(\cal E'))\bigr)
\simeq \End\bigl(\ind_{L\subseteq P}^G(\IC(\cal O,\cal E))\bigr).
\]
\end{lemma}

\begin{proof}
Since parabolic induction and restriction commute with
the Fourier--Laumon transform
(see for instance \cite[Lemma~2.9]{achar_henderson_juteau_riche16}
in a slightly different setting),
\[
\Four_{\frak g}\bigl(\ind_{L\subseteq P}^G(\IC(\cal O,\cal E))\bigr)
\simeq \ind_{L\subseteq P}^G\bigl(\Four_{\frak l}(\IC(\cal O,\cal E))\bigr);
\]
moreover, the functorial equivalence $\Four_{\frak g}$
induces an isomorphism of endomorphism rings.
By assumption, this last sheaf is isomorphic to
\[
\ind_{L\subseteq P}^G\bigl(\infl(\IC(\cal O,\cal E'))\bigr)
\simeq \ind_{L\subseteq P}^G\bigl(\IC(\cal O\times\frak z(L^\circ),\infl(\cal E'))\bigr),
\]
which, by \Cref{thm85}, is isomorphic in turn to
\[
\IC\bigl(Y_{(\frak l,\cal O)},\ind_{L\subseteq P}^G
\infl_{G\tn{-reg}}(\cal E')\bigr).
\]
Finally, since formation of intersection cohomology complexes is fully faithful,
the endomorphism ring of this perverse sheaf
is canonically isomorphic to the endomorphism ring of
\[
\ind_{L\subseteq P}^G\bigl(\infl_{G\tn{-reg}}(\cal E')\bigr). \qedhere
\]
\end{proof}

\begin{proposition} \label{thm89}
Let $\cal O$ be a nilpotent orbit of~$L$,
let $q\cal E\in\Perv(\cal O/L)$ be cuspidal,
let $q\cal E'$ be associated to~$q\cal E$ as in \Cref{thm88},
and let $\kappa_{(\cal O,q\cal E')}$ be an AMS cocycle.
If $L$ is a quasi-Levi subgroup of~$G$
then each AMS isomorphism induces an isomorphism
\[
\End\bigl(\ind_{L\subseteq P}^G(\IC(\cal O,q\cal E))\bigr)
\simeq \overline\bbQ_\ell\bigl[W(G,L)_{(\cal O,q\cal E')},\kappa_{(\cal O,q\cal E')}\bigr].
\]
\end{proposition}

\begin{proof}
Combine \Cref{thm87,thm88}.
\end{proof}

For future applications,
we finish this section by rewriting \Cref{thm89}
in the notation of \cite[\S5]{AMS18}.
Let $M$ be a quasi-Levi subgroup of~$G$,
let $\cal O$ be a nilpotent orbit for~$M$,
and let $q\cal E\in\Loc(\cal O/M)$ be cuspidal.
Let $qt \defeq (M,\cal O,q\cal E)$
and let $q\frak t \defeq [M,\cal O,q\cal E]_G$.
Our goal is to parameterize the set
\[
\qcsupp_G^{-1}(q\mathfrak{t})
\]
of equivalence classes $[\cal O',\cal F]_G$
with quasi-cuspidal support $q\frak t$.

Let $N_G(qt)$ be the stabilizer of $qt$ in~$G$
and let
\[
W_{qt} \defeq N_G(qt)/M = W(G,M)_{(\cal O,q\cal E)}
\]
where $qt = (M,\cal O,q\cal E)$ as above.
Every $g\in G$ induces by conjugation
an isomorphism $W_{qt}\to W_{qt'}$
where $qt' = {}^gqt$.
Let
\[
W_{q\frak t} = \varprojlim_{[qt]_G = q\frak t} W_{qt}
\]
denote the limit over the corresponding inverse system.
Since the automorphisms of each $W_{qt}$
induced by this system of isomorphisms are all inner,
there is a canonical identification
\[
\Irr(W_{q\frak t})
= \varprojlim_{[qt]_G = q\frak t} \Irr(W_{qt}).
\]
Next, choose an irreducible summand $\cal E$
of $\Res^G_{G^\circ}(q\cal E)$,
let $t^\circ \defeq (M,\cal O,\cal E)$,
and let $\frak t^\circ \defeq [M,\cal O,q\cal E]_{G^\circ}$.
Define
\[
W_{t^\circ} \defeq N_{G^\circ}(M^\circ)/M^\circ.
\]
In spite of appearances, the group $W_{t^\circ}$ is a special case
of the group $W_{qt}$ by \cite[3.1(a)]{AMS18}
since $q\cal E$ is cuspidal.
One can again form the inverse limit $W_{\frak t^\circ}$
using still the action of~$G$ (not of~$G^\circ$).
Working component by component,
we see that $W_{\frak t^\circ}$
naturally identifies with a normal subgroup of $W_{q\frak t}$.
Our analogue of \cite[5.3]{AMS18} is the following.

\begin{lemma}\label{AMSLemma5.3}
Let $\kappa_{q\frak t'}\in Z^2(W_{q\frak t},\overline\bbQ_\ell^\times)$
be an AMS cocycle.
Then each AMS isomorphism induces a bijection
\[
q\Sigma_{q\frak t} \colon \qcsupp_G^{-1}(q\mathfrak{t})
\simeq \Irr(\bbC[W_{q\frak t}, \kappa_{q\frak t}]).
\]
\end{lemma}

\begin{proof}
Given \Cref{thm89},
this is a very mild rephrasing of \cite[5.3]{AMS18},
the only difference being our definitions
of $W_{q\frak t}$ and $W_{\frak t^\circ}$ as inverse limits.
What must be checked is that the cocycle~$\kappa_{q\frak t}$
and the bijection $q\Sigma_{q\frak t}$
are well-defined in these limits.
The verification is straightforward.
For instance, for the bijection
$q\Sigma_{q\frak t}$,
we can choose bijections
\[
q\Sigma_{qt} \colon \qcsupp_G^{-1}(q\mathfrak{t})
\simeq \Irr(\bbC[W_{q\frak t}, \kappa_{q\frak t}]).
\]
for each representative $qt$ of~$q\frak t$
such that for any other choice of
$qt' = (M', \cal O', q\mathcal{E}')$
representing $q\mathfrak{t}$ and
any $g \in G$ with $^gqt = qt'$,
the following diagram commutes:
\[
\begin{tikzcd}
\qcsupp_G^{-1}(q\mathfrak{t}) \dar[equals]  \arrow["q\Sigma_{qt}"]{r} &
\Irr(\bbC[W_{qt}, \kappa_{qt}]) \arrow["\text{ad}(g)"]{d} \\
\qcsupp_G^{-1}(q\mathfrak{t}) \arrow["q\Sigma_{qt'}"]{r} &
\Irr(\bbC[W_{qt'}, \kappa_{qt'}]) 
\end{tikzcd}
\]
Note that the right-hand vertical identification is independent of the choice of $g$
since we are already considering these representations up to isomorphism.
\end{proof}

\begin{remark}[Comparison with AMS parameterization] \label{thm96}
It seems quite likely that our parameterization
of the quasi-cuspidal induction series is identical
to that of \cite{AMS18}.
To prove the compatibility, one would at least need to verify \Cref{thm91}.
Nonetheless, the weaker statement that there is some
parameterization by irreducible representations
of a twisted Weyl-group algebra, as in \Cref{AMSLemma5.3},
is enough for our applications.
\end{remark}

\subsection{Comparison of connected cuspidal supports}
\label{sec:param:comparison}
In this section we explain why our definition
of the connected cuspidal support agrees with that of
\cite{lusztig84b} and \cite{AMS18},
largely a matter of unpacking notation.
Although Lusztig assumed $G=G^\circ$,
so that we can only compare with his results in that case,
the construction works just as well for disconnected $G$.
Throughout we work with the unipotent variety $\cal U_G$
rather than the nilpotent cone~$\cal N_G$,
a modification that is permitted by
the discussion in \Cref{sec:param:springer}.

Let $P$ be a parabolic subgroup of~$G$ with Levi quotient~$L$,
let $\cal O_1\subseteq\cal U_L$ be a unipotent conjugacy class of~$L$,
let $\cal E_1$ be an $L$-equivariant local system on~$\cal O_1$,
let $\cal O\subseteq\cal U_G$ be a unipotent conjugacy class of~$G$,
and let $\cal E$ be a $G$-equivariant local system on~$\cal O$.
Let $S_1 = \cal O_1\times Z^\circ(L)\subseteq L$.
Let $\overline Y\subset G$ be the union of the $G$-conjugacy
classes that contain an element of the closure~$\overline S_1$.

In stack language, Lusztig considers the following correspondence,
where $\pi\colon P\to L$ is the projection,
$i\colon P\to G$ is the inclusion,
and we use the same notation~$\pi$
(resp.~$i$) for the map induced by~$\pi$ (resp.~$i$)
on quotient stacks.
\[
\begin{tikzcd}[row sep=small]
& \pi^{-1}(S_1)/P \dlar[swap]{\pi}\drar{i} & \\
S_1/L & & G/G
\end{tikzcd}
\]
Following Lusztig's notation \cite[4.4, 6.4]{lusztig84b}, let
\[
K(\cal E_1) \defeq \IC\bigl(\pi^{-1}(S_1),
\pi^*(\cal E_1\boxtimes\tn{triv})\bigr)[-\dim(S_1)-\dim(G/L)]
\]
and let
$d = (v_G - \tfrac12\dim(\cal O))
- (v_L - \tfrac12\dim(\cal O_1))$
where $v_G$ is the number of positive roots of~$G$.
The shift in~$K(\cal E_1)$ arises from Lusztig's convention
that IC sheaves are not shifted
to make them perverse \cite[0.1]{lusztig84b}:
this shift is the dimension of the space
$G\times^P\pi^{-1}(S_1)$ on which Lusztig's
sheaf $\overline{\cal E_1}$ lives.

\begin{definition} \label{thm44}
The pair $(\cal O,\cal E)$ has connected cuspidal support
$(\cal O_1,\cal E_1)$ in Lusztig's sense if 
$\cal O\subseteq\overline Y$ and $\cal E$ is a direct summand of
\[
H^{2d}\bigl(i_!(K(\cal E_1))|_{\cal O}\bigr).
\]
\end{definition}

Here $H^{2d}$ denotes the cohomology in degree~$2d$ of the complex, a sheaf.
When $G=G^\circ$ this definition is not Lusztig's original definition,
but he proves that the two are equivalent in \cite[6.5(a,b)]{lusztig84b}.

\begin{lemma} \label{thm42}
Suppose $\cal O\subseteq\overline Y$.
Then $i_!(K(\cal E_1))|_{\cal O}[2d] \simeq
\ind_{L\subseteq P}^G\bigl(\IC(\cal O_1,\cal E_1)\bigr)|_{\cal O}[-\dim(\cal O)]$.
\end{lemma}

\begin{proof}
Let $\cal O'$ be the unipotent conjugacy class in~$G$ induced%
\footnote{See \cite[\S7.1]{collingwood_mcgovern93}
for a summary of properties of induced orbits.
As usual, even though the theory developed
in \cite{collingwood_mcgovern93} book takes place on
the Lie algebra one can with moderate care 
pass to the group using a Springer isomorphism.}
from~$\cal O_1$.
Consider the following commutative diagram of correspondences,
in which the top two rows of vertical morphisms
consist of closed embeddings and the bottom row
consists of open embeddings.
\[
\begin{tikzcd}
\cal U_L/L &
\cal U_P/P \lar[swap]{\pi}\rar{i} &
\cal U_G/G \\
\overline{\cal O}_1/L \dar\uar &
\pi^{-1}(\overline{\cal O}_1)/P \lar[swap]{\pi}\rar{i} \dar\uar &
\overline{\cal O'}/G \dar\uar \\
\overline S_1/L &
\pi^{-1}(\overline S_1)/P \lar[swap]{\pi}\rar{i} &
G/G \\
S_1/L \uar &
\pi^{-1}(S_1)/P \lar[swap]{\pi}\rar{i} \uar &
G/G \uar[equals]
\end{tikzcd}
\]
The three squares on the right are Cartesian.
For the bottom square this is obvious and
this follows for the two squares above it from the equation
\[
\overline{\cal O'}\cap\cal U_P = \pi^{-1}(\overline{\cal O_1}),
\]
which holds because $\cal O'\cap\cal U_P$
is a dense open subset of $\pi^{-1}(\cal O_1)$.

We work from Lusztig's construction to ours
by climbing this diagram from bottom to top.
The bottom row of squares and exactness of pullback shows that
\[
i_!(K(\cal E_1))[\dim(S_1) + \dim(G/L)]
\simeq i_!\pi^*\IC(S_1,\cal E_1\boxtimes\tn{triv}).
\]
The middle row of squares shows that
\[
\bigl(i_!\pi^*\IC(S_1,\cal E_1\boxtimes\tn{triv})\bigr)|_{\overline{\cal O'}}
\simeq i_!\pi^*\bigl(\IC(S_1,\cal E_1\boxtimes\tn{triv})|_{\overline{\cal O_1}}\bigr)
\simeq i_!\pi^*\IC(\cal O_1,\cal E_1)[\dim Z^\circ(L)].
\]
The top row of squares shows that
$\ind_{L\subseteq P}^G\bigl(\IC(\cal O_1,\cal E_1)\bigr)$
is the extension by zero of $i_!\pi^*\IC(\cal O_1,\cal E_1)$.
Together these implications show that
\[
i_!(K(\cal E_1))|_{\overline{\cal O'}}[\dim(\cal O_1) + \dim(G/L)] \simeq
\ind_{L\subseteq P}^G\bigl(\IC(\cal O_1,\cal E_1)\bigr)|_{\overline{\cal O'}}.
\]
The shifts work out because $2v_G-2v_L = \dim(G/L)$.
To finish, use that $\cal O\subseteq\overline Y$
if (and only if) $\cal O\subseteq\overline{\cal O'}$.
\end{proof}

\begin{corollary} \label{thm47}
The local system $\cal E$ is a summand of
$H^{2d}\bigl(i_!(K(\cal E_1))|_{\cal O}\bigr)$
if and only if the perverse sheaf
$\IC(\cal O,\cal E)$ is a direct summand of
$\ind_{L\subseteq P}^G\bigl(\IC(\cal O_1,\cal E_1)\bigr)$.
\end{corollary}

\begin{proof}
For brevity, write
$\cal F = \ind_{L\subseteq P}^G\bigl(\IC(\cal O_1,\cal E_1)\bigr)$.
By \Cref{thm42}, the local system $\cal E$ is a summand
of $H^{2d}\bigl(i_!(K(\cal E_1))|_{\cal O}\bigr)$
if and only if $\cal E$ is a summand of
$H^{-\dim(\cal O)}(\cal F|_{\cal O})$.
Since the sheaf $\cal F$
(or if one likes, its corresponding shift
by $[-\dim(G)]$ on the stack) is perverse,
it follows from the intrinsic description
of the intermediate extension \cite[Lemma~3.3.3]{achar21}
that this happens if and only if
$\IC(\cal O,\cal E)$ is a direct summand of~$\cal F$.
\end{proof}

\begin{corollary}\label{compcor1}
Our notion of connected cuspidal support agrees with that of \Cref{thm44}:
if $(\cal O,\cal E)$ has cuspidal support $(\cal O_1,\cal E_1)$
in Lusztig's sense (\Cref{thm44})
then $\IC(\cal O,\cal E)$ has connected cuspidal support
$\IC(\cal O_1,\cal E_1)$ in our sense (\Cref{thm32}).
\end{corollary}

In particular, when $G=G^\circ$,
our connected cuspidal support map agrees with Lusztig's
and $(\cal O,\cal E)$ is cuspidal in Lusztig's sense
if and only if $\IC(\cal O,\cal E)$ is cuspidal in our sense.

In \cite[\S 4]{AMS18}, the connected cuspidal support
is defined in a slightly different (but equivalent) way, as follows.
Clifford theory (see e.g.\ \cite[1.2]{AMS18})
asserts that for $\rho \in \Irr(A_{G}(u))$
there is a $\rho^{\circ} \in \Irr(A_{G^{\circ}}(u))$,
unique up to $A_{G}(u)$-conjugacy, such that
\[
\rho = \text{Ind}^{A_{G}(u)}_{A_{G}(u)_{\rho^{\circ}}}(\tau \otimes \rho^{\circ})
\]
for some $\tau$ as above.
Alternatively and equivalently,
we can define $\rho^\circ$
as any chosen irreducible subrepresentation of
$\Res_{A_{G^\circ}(u)}^{A_G(u)}(\rho)$.
Then \cite[(39)]{AMS18} defines the connected cuspidal support of $(u, \rho)$
to be the $G$-conjugacy class of the cuspidal support of $(u, \rho^{\circ})$.

\begin{lemma}
Let $\cal F\in\Irr\bigl(\Perv(\cal U_G/G)\bigr)$ and
let $\bigl(u\in\cal U_G,\rho\in\Irr(A_G(u))\bigr)$
correspond to~$\cal F$.
Then the connected cuspidal support of $\cal F$ in our sense
corresponds to the (connected) cuspidal support of $(u,\rho)$
in the sense of \cite[(39)]{AMS18}.
\end{lemma}

\begin{proof}
This follows immediately from \Cref{thm37}.
\end{proof}

So we can say definitively that all notions
of connected cuspidal support are the same.

\section{Cuspidal support}
\label{sec:bernstein}
Let $F$ be a nonarchimedean local field and $G$ a reductive $F$-group.
A convenient starting point for organizing
the smooth dual $\Pi(G)$
is the supercuspidal representations
of $G$ and its Levi subgroups.
Every $\pi\in\Pi(G)$ 
arises as a subquotient of a (normalized) parabolic induction
$\ind_{L\subseteq P}^G(\sigma)$ of a supercuspidal
representation~$\sigma$ of a Levi subgroup $L$ of~$G$,
and by the Bernstein--Zelevinsky Geometrical Lemma,
the $G(F)$-conjugacy class of the pair $(L,\sigma)$ is unique.
We call such a pair $(L,\sigma)$
a \mathdef{cuspidal pair} for~$G$
and its conjugacy class $[L,\sigma]_G$
the \mathdef{supercuspidal support} $\Sc(\pi)$ of~$\pi$.
Writing $\Omega(G)$ for the set
of $G(F)$-conjugacy classes of cuspidal pairs,
we thus have a map
\begin{equation} \label{thm49}
\Sc \colon \Pi(G) \longrightarrow \Omega(G).
\end{equation}

In this section we define a Galois analogue 
\[
\Cusp \colon \Phi_\tn{e}^Z(G) \longrightarrow
\Omega_Z({}^LG),
\]
of the supercuspidal support map $\Sc$ above,
called the \mathdef{cuspidal support} map,
then use it to study~$\Phi_\tn{e}^Z(G)$.
The definition of~$\Cusp$ culminates
in \Cref{sec:bernstein:support}
after rewriting in the language of $L$-parameters
all the notions of the previous paragraph,
such as smooth irreducible representations,
supercuspidality of such, Levi subgroups,
supercuspidal support, and the Bernstein variety.
With the definition of cuspidal support complete,
we turn in \Cref{sec:bernstein:fibers}
to the problem of parameterizing
the fibers of the cuspidal support map.
After some refinement,
this parameterization yields
in \Cref{sec:bernstein:quotient}
a description of $\Phi_\tn{e}^Z(G)$
as a disjoint union of certain twisted extended quotients.
Although our constructions are inescapably
motivated by the Langlands correspondence,
we have chosen to wait until \Cref{sec:corr}
to explain the precise connection
to the representation theory of~$G(F)$.

All constructions in this section take place
on the level of dual groups,
so we may safely assume that $G$ is quasi-split.
We stress that the fundamental objects
in this section---chiefly, the dual Bernstein variety,
the shape of the cuspidal support map,
and twisted extended quotients---appear already
in earlier work of Haines \cite[\S5]{haines14} and
Aubert--Moussaoui--Solleveld \cite[\S7--9]{AMS18}.

\subsection{\texorpdfstring{$L$}{L}-parameters}
\label{sec:lparam}
The notion of $L$-parameter,
strictly speaking, is not well-defined:
there are various models (more precisely, definitions),
of $L$-parameters.
These models differ only in their handling
of the unipotent part of the $L$-parameter.
Although the sets~$\widetilde\Phi_\bullet(G)$
of $L$-parameters in two different models%
\footnote{Here $\bullet$ is a placeholder
for a subscript labeling the model.}
are not in bijection with each other,
their equivalence (that is, $\widehat G$-conjugacy) classes
$\Phi_\bullet(G)$ do lie in canonical bijection.

The two classical models of most relevance to us are
\begin{enumerate}
\item
the set $\widetilde\Phi_\tn{a}(G)$
of $L$-parameters $\bbG_\tn{a}\rtimes W_F\to{}^LG$ and

\item
the set $\widetilde\Phi_{\SL}(G)$
of $L$-parameters $W_F\times\SL_2\to{}^LG$.
\end{enumerate}
Ultimately, however, the most convenient model for us
is the following interpolation of these two.

\begin{definition}
Let $\widetilde\Phi(G)$  be the set of pairs $(\phi,u)$
in which $\phi\colon W_F\to{}^LG$ is 
an admissible homomorphism%
\footnote{That is, $\phi$ is an $L$-parameter in the sense of
\cite[8.2]{borel_corvallis} that is trivial on~$\bbG_\tn{a}$.}
and $u\in Z_{\widehat G}(\phi)$ is unipotent.
Let $\Phi(G) = \widetilde\Phi(G)/\widehat G$.
\end{definition}

We will sometimes write (for example)
$\Phi(G)$ instead of $\Phi({}^LG)$
since the former depends only on~$G$ through its $L$-group.

The three models of $L$-parameter are related
by the following diagram:
\begin{equation}\label{thm30}
\widetilde\Phi_\tn{a}(G) \longleftarrow
\widetilde\Phi_{\SL}(G) \longrightarrow
\widetilde\Phi(G).
\end{equation}
The first map is restriction along the homomorphism
$\bbG_\tn{a}\rtimes W_F \to W_F\times\SL_2$ defined by
$(t,w) \mapsto \iota(w)\cdot\smat1t{}1$,
where $\iota \colon W_F\to W_F\times\SL_2$ is the homomorphism
\[
\iota \colon w \mapsto w\cdot\begin{bmatrix}|w|^{1/2}&\\&|w|^{-1/2}\end{bmatrix}.
\]
The second map is the assignment
$\varphi\mapsto\bigl(\varphi|_{W_F},
\bigl(\smat11{}1\bigr)\bigr)$.
For convenience we say $\varphi$ \mathdef{enriches}
$(\phi,u)$ if $\varphi\mapsto(\phi,u)$ in this way.
By the Jacobson--Morozov theorem for~$\SL_2$
(or more precisely, a variant of it
\cite[Lemma~2.1]{gross_reeder} in the first case),
the maps in \eqref{thm30} are surjective and
induce bijections on equivalence classes.

The model $\widetilde\Phi_\tn{a}(G)$ is relevant for us
because it is closely related to the following 
important notion of Vogan \cite{vogan93}
and Haines \cite{haines14}. 

\begin{definition}
The \mathdef{infinitesimal character} of
$\varphi\in\widetilde\Phi_\tn{a}(G)$
is the $\widehat G$-conjugacy class
of the admissible homomorphism
$\varphi|_{W_F}\colon W_F\to{}^LG$.
\end{definition}

In light of the comparison between $\widetilde\Phi_\tn{a}(G)$
and $\widetilde\Phi_{\SL}(G)$ above,
we define the infinitesimal character
of $\varphi\in\widetilde\Phi_{\SL}(G)$
to be the $\widehat G$-conjugacy class of $\varphi\circ\iota$.
As for the model $\widetilde\Phi(G)$,
here one can only define an infinitesimal
character on the level of equivalence classes,
not for the $L$-parameters themselves,
since the restriction map $\iota$ uses
the entire $\SL_2$-homomorphism,
not just its unipotent part.

\begin{remark}
When one performs constructions with $L$-parameters,
it often happens that a construction is easier to state
using one model rather than some alternative second model,
in the sense that for the first model the construction
is defined on the level of $L$-parameters,
not just equivalence classes thereof.
This event does not present a technical obstructions to switching models
because any such construction should have the feature that
equivalent inputs produce conjugate outputs,
and so one can simply transfer the construction
to equivalence classes of $L$-parameters in the second model
using the bijection with those of the first model.

However, it often happens that in the second model,
the transferred construction does not lift to
a well-defined construction on $L$-parameters in the second model,
as we saw for the infinitesimal character.
This phenomenon can make constructions that mix models
--- for instance, our definition of the cuspidal support
map for $L$-parameters -- more subtle to define
because it forces one to make choices,
then check that the construction is independent
of all choices up to conjugacy.
This problem appears to be an inescapable
feature of the theory.
\end{remark}

\begin{warning}
The models $\widetilde\Phi_\tn{a}(G)$ and $\widetilde\Phi(G)$
of $L$-parameters are superficially similar,
as they both involve an admissible homomorphism
$W_F\to{}^LG$ and a choice of unipotent element of~$\widehat G$.
Although it is always possible to arrange for the representatives
of the equivalence classes of $L$-parameters on the left and right 
of \eqref{thm30} which are the image of some fixed class in
$\widetilde{\Phi}_{\SL_{2}}$ to have the same unipotent element,
typically the admissible homomorphisms of the two representatives
are genuinely different, though they differ only
on the image of Frobenius.
In general the difference between the two
is not an unramified twist,
though it will be in a particular case
that arises in the construction of
the cuspidal support map for $L$-parameters.
\end{warning}

\subsection{Enhancements}
\label{sec:corr:enh}
As we mentioned in the previous section,
all models of $L$-parameter ultimately give rise
to the same notion because equivalence classes
in different models are in bijection.
Such a simple state of affairs no longer exists
for enhancements of $L$-parameters, however:
there are many essentially different ways one can enhance
an $L$-parameter to capture information
about the internal structure of $L$-packets.
For the most part, the choice of enhancement
has to do with the flavor of inner form
one works with on the automorphic side,
for instance, pure, extended pure, or rigid.
We point the reader to \cite{kaletha_taibi}
for a summary of the various enhancements in vogue.

In this article we will use the enhancement due to Kaletha
\cite{Kal16}, with a brief comparison in \Cref{thm50}
to the enhancement used in \cite{AMS18}.
For brevity we work only on the Galois side,
deferring the explanation of the relationship
with rigid inner forms to \Cref{thm50}.

Our enhancement depends on a finite subgroup $Z\subseteq Z_G$,
which we may take as fixed for the rest of the article.
Let $\widehat G^+ \defeq\widehat{G/Z}$.
More generally, given a subgroup $H\subseteq\widehat G$,
let $H^+$ denote the preimage of~$H$ in $\widehat G^+$
under the isogeny $\widehat G^+\to\widehat G$.

\begin{definition} \label{enhancedparamdef}
A ($Z$-)\mathdef{enhanced $L$-parameter} for $G$ (or for ${}^LG$)
is a triple $(\phi,u,\rho)$ consisting of an $L$-parameter $(\phi,u)$ of~$G$ and
an irreducible complex representation $\rho$ of~$\pi_0(Z_{\widehat G}(\phi,u)^+)$.
\end{definition}

The representation~$\rho$ is called the \mathdef{enhancement}.
Write $\Phi_\tn{e}^Z(G)$ for the set of
equivalence (that is $\widehat G$-conjugacy)
classes of $Z$-enhanced $L$-parameters.

In the more widely-used $\SL_2$-model for $L$-parameters,
an enhancement of $\varphi\in\widetilde\Phi_{\SL}(G)$
is by definition a representation of
$\pi_0\bigl(Z_{\widehat G}(\varphi)^+\bigr)$.
The notions end up being the same because
the two groups are canonically isomorphic.
Even better, before passage to~$\pi_0$
there is already a pleasant relationship
between $Z_{\widehat G}(\varphi)^+$
and $Z_{\widehat G}(\phi,u)^+$ when $\varphi$ enriches $(\phi,u)$.

\begin{fact} \label{thm31}
Let $(\phi,u)$ be an $L$-parameter of~$G$.
\begin{enumerate}
\item
The group $Z_{\widehat G}(\phi,u)^+$ admits a Levi decomposition
\[
Z_{\widehat G}(\phi,u)^+ \simeq \Ru\bigl(Z_{\widehat G}(\phi,u)\bigr)
\rtimes Z_{\widehat G}(\phi,u)^+_\tn{red}
\]
and any two Levi factors (that is, maximal reductive subgroups)
are conjugate under $\Ru\bigl(Z_{\widehat G}(\phi,u)\bigr)$.

\item
Every Levi factor of~$Z_{\widehat G}(\phi,u)^+$
is of the form $Z_{\widehat G}(\varphi)^+$ for some
enrichment $\varphi\in\widetilde\Phi_{\SL}(G)$ of~$(\phi,u)$.
\end{enumerate}
\end{fact}

Here $\Ru(-)$ denotes the unipotent radical
and $(-)_\tn{red}$ the maximal reductive quotient.
\Cref{thm31} is originally due to Barbasch--Vogan and Kostant;
see \cite[3.7.3]{collingwood_mcgovern93}
and \cite[2.4]{kazhdan_lusztig87} for discussion.
The first part holds not just for $Z_{\widehat G}(\phi,u)^+$,
but for any linear algebraic group (as $\bbC$ has characteristic zero).

\begin{corollary}\label{cuspidalparameterlemma}
Let $(\phi,u)$ be an $L$-parameter of~$G$.
If $\varphi\in\widetilde\Phi_{\SL}(G)$ enriches~$(\phi,u)$ then
\[
\pi_0\bigl(Z_{\widehat G}(\phi,u)^+\bigr)
\simeq \pi_0\bigl(Z_{\widehat G}(\varphi)^+\bigr).
\]
\end{corollary}

\subsection{Levi \texorpdfstring{$L$}{L}-subgroups}
\label{sec:bernstein:levi}

For the convenience of the reader we review
this classical notion here.

\begin{definition}
A subgroup $\cal M$ of ${}^LG$
is a \mathdef{Levi $L$-subgroup}
if one of the equivalent properties below holds.
\begin{enumerate}
\item
$\cal M$ is the centralizer in~${}^LG$ of a torus of~$\widehat G$
and the projection $\cal M\to\Gamma$ is surjective.

\item
$\cal M\subseteq{}^LG$ is conjugate
to a subgroup of the form $\widehat M\rtimes W_F$
for $\widehat M$ a $W_F$-stable
Levi subgroup of~$\widehat G$.
\end{enumerate}
\end{definition}

These definitions go back to Borel's Corvallis article
\cite[3.3, 3.4]{borel_corvallis}.
For a proof of their equivalence,
see \cite[6.2]{AMS18}.
The essential value of the definition 
is that for $G$ quasi-split,
there is a canonical bijection between Levi subgroups of~$G$
and Levi $L$-subgroups of~${}^LG$,
which is compatible with formation of $L$-groups.

Let $\cal M$ be a Levi $L$-subgroup.
The central subgroup $Z\subseteq G$
lies in every Levi subgroup of~$G$,
in particular a Levi~$M$ corresponding to~$\cal M$.
We may therefore speak of $Z$-enhanced $L$-parameters
for any Levi $L$-subgroup~$\cal M$.

\begin{remark}
In this remark, take as ${}^LG$ the minimal form
${}^LG = \widehat G\rtimes\Gal(E/F)$
where $E$ is the splitting field of~$G$.
With this convention ${}^LG$ is a disconnected linear algebraic group,
and thus fits into the framework of \Cref{sec:springer}.
It is natural to wonder whether there is any relationship
between the two types of disconnected Levi subgroups
that arise in this article, Levi $L$-subgroups
and the quasi-Levi subgroups of \Cref{thm51}.
We cannot find any relationship;
in particular, neither property implies the other in general
(cf.~\cite[6.3]{AMS18}).

On the one hand, there is a bijection between quasi-Levi subgroups
and Levi subgroups of the identity component,
and the components group of a quasi-Levi subgroup
can vary for different quasi-Levis.
On the other hand, not every Levi subgroup of~$\widehat G$
is the identity component of a Levi $L$-subgroup
and the components group of a Levi $L$-subgroup
of ${}^LG$ is always the same, $\Gal(E/F)$.
\end{remark}

\subsection{Cuspidal enhanced $L$-parameters}
In this section we define a class of enhanced $L$-parameters
$(\phi,u,\rho)$ which ought to correspond to supercuspidal representations. 
The definition passes through the theory of cuspidality for equivariant
perverse sheaves on the nilpotent cone as in \Cref{sec:springer},
so the main task is to interpret the enhancement as such a sheaf,
or equivalently, by \Cref{thm32}, as an irreducible representation
of $\pi_0\bigl(Z_{H^+}(u)\bigr)$ for some
reductive group~$H^+$ and unipotent element $u\in H^+$.
This interpretation arises from computing
the centralizers in two steps instead of all at once:
writing $H^+ \defeq Z_{\widehat G}(\phi)^+ = Z_{\widehat G^+}(\phi)$,
we see that
\[
Z_{\widehat G}(\phi,u)^+ = Z_{H^+}(u).
\]

\begin{definition}
An enhanced $L$-parameter $(\phi,u,\rho)$ is \mathdef{cuspidal}
if $(\phi,u)$ is discrete and $(u,\rho)$
is cuspidal in the sense of \Cref{thm32}.
\end{definition}

We denote the set of equivalence classes
of cuspidal enhanced $L$-parameters
by $\Phi_\tn{e,cusp}^Z({}^LG)$.
In our model of $L$-parameter,
we declare $(\phi,u)$ to be discrete
if $Z_{\widehat G}(\phi,u)/Z(\widehat G)^\Gamma$ is finite.
By \Cref{thm31}, this definition is compatible with the usual definition,
that $Z_{\widehat G}(\varphi)/Z(\widehat G)^\Gamma$ be finite
(say, for any $\varphi$ enriching $(\phi,u)$).

\subsection{The dual Bernstein variety}
\label{sec:bernstein:variety}
The goal of this section is to define 
an analogue for $L$-parameters of the Bernstein variety.
To orient ourselves, we recall
the construction of the classical Bernstein variety~$\Omega(G)$.

First, we organize $\Omega(G)$ into inertial equivalence classes
$\llbracket M,\sigma \rrbracket_M$,
which will eventually form the connected components
of the variety enriching the set~$\Omega(G)$.
Second, the group $X^*_\tn{nr}(G)$
of unramified characters of~$G(F)$
acts by twisting on the set $\Irr_\tn{sc}(G)$
of supercuspidal representations of~$G(F)$.
We use this action to make $\Irr_\tn{sc}(G)$ a variety,
giving each orbit for this action the quotient structure inherited
from the algebraic torus $X^*_\tn{nr}(G)$.
To ensure that the quotients are algebraic
we use that the isotropy groups of this action are finite:
more precisely, restricting a twisted representation to~$Z(G)$
shows that the isotropy groups are contained in the kernel of the map
\[
X^*_\tn{nr}(G) \to X^*_\tn{nr}\bigl(Z(G)\bigr),
\]
which is finite.
Third, we topologize an inertial equivalence
class $\llbracket M,\sigma \rrbracket_G$
using the covering map
\[
\llbracket M,\sigma \rrbracket_M
\to \llbracket M,\sigma \rrbracket_G
\]
by observing that this map
realizes the target as the quotient of the source
by the action of the finite group $W(G,M)_\sigma$.
So all in all, the components of the resulting
variety $\Irr_\tn{sc}(G)$
are quotients of $X^*_\tn{nr}(G)$
by the action of a finite group:
\begin{equation} \label{thm28}
\llbracket M,\sigma \rrbracket_G
\simeq \frac{X^*_\tn{nr}(M)}{X^*_\tn{nr}(M)_\sigma\rtimes W(G,M)_\sigma}.
\end{equation}

\begin{warning}
Strictly speaking, $\Omega(G)$ is not a variety in the usual sense,
but rather, an infinite disjoint union of varieties.
The same will be true of~$\Omega({}^LG)$.
This is because when $G$ is nontrivial
$\Omega(G)$ has infinitely many
connected components, reflecting the fact that $G(F)$
has representations of arbitrarily large depth.
Following the conventions in the literature,
we ignore this distinction.
\end{warning}

We will now carry out these three steps for $L$-parameters,
after first defining the set underlying the dual Bernstein variety.
A \mathdef{cuspidal datum} for~${}^LG$ is
a quadruple $(\cal M,\psi,v,\rho)$
where $\cal M$ is a Levi $L$-subgroup of~$G$
and $(\psi,v,\rho)$ is a cuspidal enhanced $L$-parameter for~$\cal M$.
The group~$\widehat G$ acts by conjugation on the set of cuspidal data.
We call the quotient $\Omega_Z({}^LG)$ by this action
the \mathdef{dual Bernstein variety} of~$G$.

For the first step, the $L$-parameter analogue of the torus of
unramified characters of~$G(F)$ is the torus
\[
X^*_\tn{nr}({}^LG) \defeq Z(\widehat G)^{I_F,\circ}_\tn{Frob}.
\]
Indeed \cite[3.3.1]{haines14},
the Kottwitz homomorphism induces an isomorphism
\begin{equation} \label{thm54}
X^*_\tn{nr}({}^LG) \simeq X^*_\tn{nr}(G).
\end{equation}
In other words, write $\val_F\colon W_F\to\bbZ$
for the composition of the Artin reciprocity map
with the usual valuation on~$F^\times$.
There is a map $X^*_\tn{nr}({}^LG)\to\Phi({}^LG)$
which we denote by $z\mapsto z^{\val_F}$
and is defined to send $z$
to the $L$-parameter~$\tilde z^{\val _F}$,
where $\tilde z$ is a lift of~$z$;
the resulting equivalence class is independent
of the choice of~$\tilde z$.
If $z$ corresponds to the unramified character~$\chi$
of~$G(F)$ then $z^{\val_F}$ is the $L$-parameter of~$\chi$.
Under the local Langlands correspondence,
the action of $X^*_\tn{nr}(G)$ on~$\Pi(G)$
becomes the action of $X^*_\tn{nr}({}^LG)$
on admissible homomorphisms $\phi\colon W_F\to{}^LG$ by
\[
(z,\phi)\mapsto z^{\val_F}\cdot\phi.
\]
A short check shows that this formula is again
independent of the choice of lift of~$z$,
up to equivalence.
We lift the twisting action to~$\Phi_\tn{e}^Z({}^LG)$
by imposing the trivial action on
the unipotent element and the enhancement.

\begin{definition}
Two cuspidal data $(\cal M,\psi,v,\rho)$
and $(\cal M',\psi',v',\rho')$ are
(${}^LG$-)\mathdef{inertially equivalent} if
there are $g\in\widehat G$ and $z\in X^*_\tn{nr}(\cal M)$ such that
\begin{equation} \label{thm26}
{}^g(\cal M',\psi',v',\rho') = (\cal M,z^{\val_F}\cdot\psi,v,\rho).
\end{equation}
Let $\llbracket \cal M,\psi,v,\rho\rrbracket_{\widehat G}$
denote the inertial equivalence class of $(\cal M,\psi,v,\rho)$.
\end{definition}

There is a slight subtlety here:
an inertial equivalence class is not a subset of~$\Omega_Z({}^LG)$,
simply because the elements of the equivalence
class are $L$-parameters, not equivalence classes of such,
and although the set $\llbracket\cal M,\psi,v,\rho\rrbracket_{\widehat G}$
likely admits a reasonable variety structure,
the variety of interest should be some quotient of it.
We write
\[
\llbracket \cal M, [\psi,v,\rho]_{\cal M^\circ}\rrbracket_{\widehat G}
\]
for the image of $\llbracket\cal M,\psi,v,\rho\rrbracket_{\widehat G}$
in~$\Omega_Z({}^LG)$.
The notation is motivated by observing that \eqref{thm26}
defines an equivalence relation on quadruples
$(\cal M,[\psi,v,\rho]_{\cal M^\circ})$
in which the second entry is not an $L$-parameter,
but an equivalence class thereof.
This completes the first step,
the coarse identification of
the connected components of $\Omega_Z({}^LG)$.

\begin{remark}
The equivalence relation of inertial equivalence
does not quite arise from a group action
since the group of unramified characters
depends on~$\cal M$.
Rather, it should be possible to describe
the equivalence classes using the action
of a certain groupoid of unramified characters
of Levi $L$-subgroups.
\end{remark}

For the second step, we use the twisting action
to enrich $\Phi_\tn{e,cusp}^Z({}^LG)$ to a variety
by giving each orbit under the action
the quotient variety structure inherited
from $X^*_\tn{nr}({}^LG)$.
To ensure that the quotients are algebraic
we must again check that the isotropy groups are finite.

\begin{lemma}
The action of $X^*_\tn{nr}({}^LG)$ on~$\Phi({}^LG)$ has finite isotropy groups.
\end{lemma}

\begin{proof}
We imitate the proof sketched above,
performing the analogue for $L$-parameters
of restricting a supercuspidal representation to the center.
In this setting it is more convenient to replace the center
with the connected center $Z^\circ(G)$, however,
since the latter, unlike the former, has a dual group:
\[
\widehat{Z^\circ(G)} = \widehat G/\widehat G_\tn{der}.
\]
Functoriality of unramified characters
yields a morphism of varieties with group action
\[
\bigl(X^*_\tn{nr}({}^LG)\actson\Phi({}^LG)\bigr)
\longrightarrow
\bigl(X^*_\tn{nr}({}^LZ^\circ_G)\actson\Phi(Z^\circ_G)\bigr)
\]
Since the action on the right is faithful and the homomorphism
$X^*_\tn{nr}({}^LG)\to X^*_\tn{nr}({}^LZ^\circ_G)$ has finite kernel,
the isotropy groups for the action on the left are finite.
\end{proof}

We have now enriched $\Phi_\tn{e,cusp}^Z({}^LG)$ to a variety.
The third step is to handle the rest of~$\Omega_Z({}^LG)$.
Take an individual inertial equivalence class
$\llbracket\cal M,[\psi,v,\rho]_{\cal M^\circ}\rrbracket_{\widehat G}$.
There is a tautological map
\begin{equation}\label{thm27}
\llbracket\cal M,[\psi,v,\rho]_{\cal M^\circ}\rrbracket_{\cal M^\circ} \to
\llbracket\cal M,[\psi,v,\rho]_{\cal M^\circ}\rrbracket_{\widehat G}.
\end{equation}
The finite group
\[
W({}^LG,\cal M)
\defeq N_{\widehat G}(\cal M)/\cal M
\simeq N_{{}^LG}(\cal M)/\cal M
\]
acts on $\Phi_\tn{e,cusp}(\cal M)$, 
the stabilizer $W({}^LG,\cal M)_{[\psi,v,\rho]}$
acts on $\llbracket\cal M,[\psi,v,\rho]_{\cal M^\circ}\rrbracket_{\cal M^\circ}$,
and the resulting action is transitive on the fibers of \eqref{thm27}.
We again give 
$\llbracket\cal M,[\psi,v,\rho]_{\cal M^\circ}\rrbracket_{\widehat G}$
the unique variety structure making \eqref{thm27}
into a quotient map for this group action.
All in all, we again find that
\begin{equation} \label{thm29}
\llbracket\cal M,[\psi,v,\rho]_{\cal M^\circ}\rrbracket_{\widehat G}
\simeq \frac{X^*_\tn{nr}(\cal M)}%
{X^*_\tn{nr}(\cal M)_{[\psi,v,\rho]_{\cal M^\circ}}
\rtimes W({}^LG,\cal M)_{[\psi,v,\rho]_{\cal M^\circ}}}.
\end{equation}
This completes the enrichment of $\Omega_Z({}^LG)$ to a variety.

We will often represent the tuple $(\cal M,\psi,v,\rho)$
encoding a cuspidal datum by the letter~$\hat s$.
In this abbreviated form, the following notations are convenient.

\begin{definition} \label{thm68}
Suppose $\hat s \defeq (\cal M,\psi,v,\rho)$
is a cuspidal datum for~${}^LG$.
\begin{enumerate}
\item
Let $\hat{\frak s}_{\widehat G}
\defeq \llbracket \hat{s} \rrbracket_{\widehat G}
= \llbracket \cal M, \psi, v, \varrho \rrbracket_{\widehat G}$
denote the inertial equivalence class of~$\hat s$.

\item
Let $\frak B_Z({}^LG)$ denote the set
of all such inertial equivalence classes.

\item
Let $\Omega_Z^{\hat{\frak s}_{\widehat G}}({}^LG)$
denote the connected component of $\Omega_Z({}^LG)$
corresponding to $\hat{\frak s}_{\widehat G}$.
\end{enumerate}
\end{definition}

The definition uses that
$\pi_0\bigl(\Omega_Z({}^LG)\bigr) \simeq \frak B_Z({}^LG)$,
by construction.
In this notation, \eqref{thm27} becomes
\[
\Omega_Z^{\hat{\frak s}_{\cal M^\circ}}(\cal M)
\to \Omega_Z^{\hat{\frak s}_{\widehat G}}({}^LG).
\]
\Cref{thm68} has an obvious analogue on the automorphic side
which we will not spell out but which
is implicit in our summary of the classical Bernstein decomposition above.

\begin{remark}
A cuspidal datum~$\hat s=(\cal M,\psi,v,\rho)$
does not know ``the group for which it is a cuspidal datum'':
for instance, we can interpret $\hat s$ as a cuspidal datum
for $G$ or for~$\cal M$.
Because both perspectives are useful,
we have encoded this dependence in the notation~$\frak s_{\widehat G}$.
Relatedly, although $\hat{\frak s}_{\widehat G}$
can be recovered from~$\hat{\frak s}_{\cal M^\circ}$,
it is generally impossible to recover
$\hat{\frak s}_{\cal M^\circ}$ from~$\hat{\frak s}_{\widehat G}$:
the projection $\Omega(\cal M)\to\Omega({}^LG)$
is typically not injective on connected components.
\end{remark}

\subsection{Cuspidal support for \texorpdfstring{$L$}{L}-parameters}
\label{sec:bernstein:support}
In this section we define the cuspidal support map
\[
\Cusp \colon \Phi_\tn{e}^Z({}^LG) \to \Omega_Z({}^LG),
\]
an analogue for $L$-parameters of the supercuspidal support map
$\Sc \colon \Pi(G) \to \Omega(G)$.

\begin{remark}
With the notions of cuspidal support developed in \Cref{sec:springer} in hand,
the construction of cuspidal support for enhanced $L$-parameters as well as parameterization
of its fibers (carried out in \Cref{sec:bernstein:support} through \Cref{sec:bernstein:quotient})
are generalizations of the analogous constructions in \cite{AMS18} with minor adjustments to account 
for the finite central subgroup $Z$.
\end{remark}

\begin{construction} \label{thm33}
Given an enhanced $L$-parameter $(\phi,u,\rho)$ for~$G$,
choose a triple $(K^+,v,\varrho)$ lying in
the quasi-cuspidal support of~$(u,\rho)$
in the sense of \Cref{thm32} and let
\[
\cal M \defeq Z_{{}^LG}\bigl(Z^\circ(K^+)\bigr).
\]
Form the quadruple $(\cal M,\phi,v,\varrho)$.
\end{construction}

Although $\cal M$ is constructed from~$K^+$,
since $K^+$ is a quasi-Levi subgroup of~$H^+$
we can recover it from~$\cal M$ and~$H$
via the recipe
\begin{equation} \label{thm52}
K^+ = H^+\cap\cal M^{\circ,+}.
\end{equation}
In what follows, we will implicitly
define $K^+$ by \eqref{thm52}
given a Levi $L$-subgroup~$\cal M$.

\begin{lemma} \label{thm35}
\begin{enumerate}
\item
In \Cref{thm33}, conjugating $(\phi,u,\rho)$ by~$\widehat G$
or choosing a different representative $(K^+,v,\varrho)$
of the quasi-cuspidal support yield $\widehat G$-conjugate outputs.

\item
The quadruple $(\cal M,\phi,v,\varrho)$
of \Cref{thm33} is a cuspidal datum.
\end{enumerate}
\end{lemma}

\begin{proof}
Let $H^+\defeq Z_{\widehat G^+}(\phi)$,
the group on which the quasi-cuspidal support
is computed in \Cref{thm33}.
The first part is straightforward
and uses the fact that all choices of $(K^+,v,\varrho)$
for fixed $(\phi,u,\rho)$ are $H^+$-conjugate.
The second part asserts that
\begin{enumerate}
\item \label{thm35a}
$\cal M$ is a Levi $L$-subgroup of~${}^LG$,

\item \label{thm35b}
$\phi$ factors through $\cal M$,

\item \label{thm35c}
$(\phi,u)$ is a discrete $L$-parameter for $\cal M$, and

\item \label{thm35d}
the pair $(v,\varrho)$ is cuspidal.
\end{enumerate}
\Cref{thm35d} is a tautology.
The proof of \Cref{thm35a} is the same
as that in the proof of Proposition~7.3 of \cite{AMS18}.
For \Cref{thm35c}, the identical argument
in that Proposition shows that $\cal M$
is a minimal Levi $L$-subgroup of ${}^LG$ containing $\phi(W_F) \cup \{v\}$,
so that $(\phi, v)$ is a discrete $L$-parameter for $\cal M$.
\Cref{thm35b} holds because $\phi$ already factors
through the subgroup~$Z_{{}^LG}(H^+)$ of~$\cal M$.
\end{proof}

Hence \Cref{thm33} yields a well-defined map
from enhanced $L$-parameters to 
conjugacy classes of cuspidal data.
However, this \namecref{thm33} is not yet
the cuspidal support map because
it fails to preserve infinitesimal characters.

\begin{example}
Consider an enhanced $L$-parameter $(\tn{triv},u,\rho)$
with trivial restriction to~$W_F$
and suppose that the restriction of~$\rho$
to $Z(H^+)$ is trivial.
As \Cref{sec:tables} shows
--- for instance, in types $E_8$, $F_4$,
and $G_2$ when the $L$-parameter is
not already cuspidal ---
it often happens in these circumstances that
the cuspidal support has $v=1$
and (therefore) $\varrho=\tn{triv}$.
If $u\neq 1$ then the infinitesimal character
of the original $L$-parameter is nontrivial,
but it becomes trivial after formation
of the cuspidal support.
\end{example}

\begin{lemma} \label{thm34}
In the setting of \Cref{thm33},
there exists $z\in X^*(\cal M)_\tn{nr}$ such that
the $L$-parameters $(\phi,u)$ for ${}^LG$
and $(z^{\val_F}\cdot\phi,v)$ for $\cal M$
have the same infinitesimal characters.
\end{lemma}

Here we compare infinitesimal characters in the usual way,
by composition with the inclusion $\cal M\to{}^LG$.
Although the element~$z$ may not be unique,
simply because the isotropy subgroup for~$\phi$
in the group of unramified twists can be nontrivial,
the requirement that $(\phi,u)$ and
$(z^{\val_F}\cdot\phi,v)$ have
the same infinitesimal characters
uniquely determines the equivalence class
of the latter $L$-parameter
because formation of the infinitesimal character
separates distinct unramified twists.

\begin{proof}
A proof of this \namecref{thm34}
is sketched in and around \cite[7.6]{AMS18},
using results on $\SL_2$-triples from
\cite[2.4]{kazhdan_lusztig87}.
The same argument as in \cite{AMS18}
(though not the results themselves,
since our version~$K^+$ of their~$M$
is slightly different) constructs a cocharacter
$\chi_{(\phi,u),\rho}\colon\bbG_\tn{m}\to Z(K^+)^\circ$,
and we then take as $z$ the image of
$\chi_{(\phi,u),\rho}(q^{1/2})$ in $\cal M$.
Since $\phi$ factors through $\cal M$,
which by definition centralizes 
the target $Z(K^+)^\circ$ of~$\chi_{(\phi,u),\rho}$,
we see that $z\in Z(\cal M^\circ)^\Gamma$.
\end{proof}

\begin{definition}
Let $(\phi,u,\rho)$ be an enhanced $L$-parameter for~$G$,
let $(\cal M,\phi,v,\varrho)$ be as in \Cref{thm33},
and let~$z$ be as in \Cref{thm34}.
\begin{enumerate}
\item
The \mathdef{unnormalized cuspidal support} of~$(\phi,u,\rho)$ is
the $\widehat G$-conjugacy class
\[
\uCusp(\phi,u,\rho)\defeq[\cal M,\phi,v,\varrho]_{\widehat G}.
\]

\item
The \mathdef{cuspidal support} of~$(\phi,u,\rho)$ is
the $\widehat G$-conjugacy class
\[
\Cusp(\phi,u,\rho) \defeq
[\cal M,z^{\val_F}\cdot\phi,v,\varrho]_{\widehat G}.
\]
\end{enumerate}
\end{definition}

The distinction between $\uCusp$ and $\Cusp$,
then, is that $\uCusp$ is simpler to define
but $\Cusp$ preserves infinitesimal characters.
See pp.~175--176 of \cite{AMS18} for further discussion of this point.

In our eventual application to the local Langlands conjectures,
the central character of the enhancement
will carry important information about
the (rigid) inner form of~$G$
that carries the corresponding smooth representation.
Fortunately, the cuspidal support map preserves central characters.

\begin{lemma} \label{thm53}
Let $(\phi,u,\rho)$ be an enhanced $L$-parameter for~$G$,
let $[\cal M,\phi,v,\varrho]_{\widehat G}$
be the unnormalized cuspidal support of~$(\phi,u,\rho)$,
and let $H^+=Z_{\widehat G^+}(\phi)$.
The restrictions of $\rho$ and~$\varrho$ to $Z(H^+)$
are isotypic for the same character.
\end{lemma}

Since $Z(\widehat G)^{\Gamma,+} = Z(\widehat G^+)\cap H^+$,
the restrictions to $Z(\widehat G^+)\cap H^+$
are also isotypic for the same character.

\begin{proof}
This follows from \Cref{thm94}.
\end{proof}

\subsection{Fibers of the cuspidal support map}
\label{sec:bernstein:fibers}
We have now defined an analogue for $L$-parameters of the supercuspidal support map.
In this section we use the results of \Cref{sec:param:fibers} to 
describe the fiber of this map over the $\widehat G$-conjugacy class
of a given cuspidal datum
\[
\hat{s} \defeq (\cal M, \phi, v, \varrho).
\]
In different language,
this parameterization will describe $\Phi_\tn{e}^Z(G)$
as disjoint union of twisted extended quotients,
which we explain in \Cref{sec:bernstein:quotient}.

As in \Cref{sec:bernstein:support},
for a fixed $L$-parameter $(\phi,u)$ we let
\[
H^+ \defeq Z_{\widehat G}(\phi)^+
= Z_{\widehat G^+}(\phi),
\]
and we write $K^+$ for some Levi subgroup of~$H^+$
appearing in the quasi-cuspidal support
of $(u,\rho)$, as in \Cref{thm33}.
Given some other enhanced $L$-parameter
$(\phi',u',\rho')$ for~$G$,
write $H'$ and $K'$ for the corresponding groups.
From $\hat s$ we can construct the triples
\[
t^\circ  \defeq (K^{+\circ}, [v]_{K^{+\circ}}, \cal E)
\qquad\tn{and}\qquad
qt \defeq (K^+, [v]_{K^+}, q\cal E)
\]
as well as their conjugacy classes
$\frak t^\circ$ and $q\frak t$
as in \Cref{sec:param:fibers}.

\begin{proposition}\label{AMS9.1}(Analogue of \cite[9.1]{AMS18}) \begin{enumerate}
\item
The bijection $q\Sigma_{q\frak t}$ of \Cref{AMSLemma5.3} induces a bijection 
\begin{align*}
{}^L\Sigma_{q\frak t,\psi,v} \colon
\uCusp^{-1}([\cal M, \psi, v, \varrho]_{\widehat G})
&\to \Irr(\bbC[W_{q\frak t}, \kappa_{q\frak t}]) \\
[\phi', u', \rho']_{\widehat G}
&\mapsto q\Sigma_{q\frak t}\bigl([u']_{H^{\prime+}}, \rho'\bigr)
\end{align*}
which is canonical up to a choice of character of
$W_{q\frak t}/W_{\frak t^{\circ}}$.

\item
Denoting by $\Sigma_{\frak t^{\circ}}$ the canonical bijection between
$(\connsupp)^{-1}_{H^{\circ}}(\frak t^\circ)$
and $\bbC[W_{\frak t^\circ}]$ obtained as a special case of \Cref{AMSLemma5.3},
we have the following compatibility: 
\[
{}^L\Sigma_{q\frak t,\psi,v}\bigl([\phi, u, \rho]_{\widehat G}\bigr)
|_{W_{\frak t^{\circ}}}
= \bigoplus_i \Sigma_{\frak t^{\circ}}(u, \rho_i),
\]
where $\rho |_{\pi_0(Z_{H^{\circ}}(u))} = \bigoplus_i \rho_i$
is the decomposition into irreducible representations.

\item
The $\widehat{G}$-conjugacy class of a triple
$(\phi, u, \rho_{i})$ as in~(2) is determined by any irreducible
$\bbC[W_{\mathfrak{t}^{\circ}}]$-subrepresentation of
${}^L\Sigma_{q\frak t,\psi,v}\bigl([\phi, u, \rho]_{\widehat G}\bigr)$.

\end{enumerate}
\end{proposition}

We wrote ``$\psi$" in subscript
in ${}^L\Sigma_{q\frak t,\psi,v}$
to indicate that this map depends on~$\psi$.
It is possible for two cuspidal data
$\hat{s}$ and $\hat{s}'$ to have the same associated $q\frak t$'s
but different (even non-$\widehat G$-conjugate)
admissible homomorphisms $W_F\to{}^LG$
(that is, the second element in the quadruple).
The map is defined identically in both cases,
but the domains are different.

The formulation of (1) above is unnatural because
we have artificially chosen a representative $\hat s$
of the unnormalized cuspidal support to define
$q\frak t$ and thus the target of
the bijection ${}^L\Sigma_{q\frak t,\psi,v}$.
This deficiency will be rectified in the following section,
but before that, we explain how to interpret
our notation in~(1).
Indeed, on the face of it, the map in (1) is not well-defined:
$\bigl([u']_{H^{\prime+}},\rho'\bigr)$
need not lie in the domain of $q\Sigma_{q\frak t}$.
Briefly, the problem is resolve by conjugating.

In more detail, choose a representative $(\phi', u', \rho')$
for an equivalence class of enhanced $L$-parameters
with unnormalized cuspidal support 
$[\cal M, \phi, v, \varrho]_{\widehat G}$
and let
\[
(\cal M',\phi',v',\varrho')
\]
be the corresponding representative
of $\uCusp\bigl([\phi,u,\rho]_{\widehat G}\bigr)$
as in \Cref{thm33}, so that in particular
\[
\qcsupp_{H^{\prime+}}\bigl([u']_{H^{\prime+}}, \rho'\bigr)
= [K^{\prime+}, v', \varrho']_{H^{\prime+}}.
\]
By assumption, $(\cal M, \phi, v, \varrho)$
is $\widehat{G}$-conjugate to $(\cal M', \phi', v', \varrho')$,
and hence $H$ and $H'$ are $\widehat{G}$-conjugate in a way that
identifies the representations $\varrho$ and $\varrho'$
and the Levi subgroups $K^+$ and $K^{\prime+}$. 
Choose any $g \in \widehat G$ which conjugates
$(\cal M', \phi', v', \varrho')$ to $\hat s$.
Any two such choices differ by the stabilizer
of $\hat{s}$ in $\widehat{G}$,
which is a subgroup of $H'$.
We then define 
\[
q\Sigma_{q\frak t}(\phi, u, \rho)
\defeq q\Sigma_{q\frak t}\bigl({}^g(\phi, u, \rho)\bigr).
\]
This expression is independent of the choice of $g$
because any other choice preserves the $H$-conjugacy class of the pair,
and $q\Sigma_{qt}$ is invariant under $H$-conjugacy. 

\begin{proof}[{Proof of \Cref{AMS9.1}}]
The same proof as in \cite[9.1]{AMS18} works here, replacing their $G$ with our $H$.
\end{proof}

To finish this section, we show that the group $W_{q\frak t}$
of \Cref{AMS9.1} is isomorphic to a different group
which is better adapted to the Bernstein variety
and formation of twisted extended quotients.

Recall that $W({}^LG, \cal M) \defeq N_{\widehat G}(\cal M)/\cal M^\circ$.
It is explained in \cite[\S 8]{AMS18} that  $N_{\widehat{G}}(\cal M)$
stabilizes $X^*_\tn{nr}(\cal M)$ and $\cal M^\circ$ fixes $X^*_\tn{nr}(\cal M)$ pointwise.
It follows that $W({}^LG,\cal M)$ acts on the set of
$X^*_\tn{nr}(\cal M)$-orbits of cuspidal enhanced $L$-parameters for $\cal M$.
Given two cuspidal data $s = (\cal M,\psi,v,\varrho)$
and $s' = (\cal M,\psi',v'\varrho')$ for~${}^LG$
with common Levi $L$-subgroup~$\cal M$,
the corresponding inertial classes
$\hat{\frak s}_{\widehat G}$
and $\hat{\frak s}'_{\widehat G}$
are equal if and only if
their $X^*_\tn{nr}(\cal M)$-orbits
$\hat{\frak s}_{\cal M^\circ}$
and $\hat{\frak s}'_{\cal M^\circ}$
are in the same $W({}^LG,\cal M)$-orbit.
Define the finite group
\[
W_{\hat{\frak s}_{\cal M^\circ}}
\defeq Z_{W({}^LG, \cal M)}(\hat{\frak s}_{\cal M^\circ}).
\]
Recall from \Cref{sec:bernstein:support} that the datum $\hat{s}$
gives rise to the triple $q\frak t = [K^+, v,\varrho]_{K^+}$.
The following result is an analogue of \cite[8.2]{AMS18},
which compares the groups $W_{qt}$ and $W_{\hat{\frak s}_{\cal M^\circ}}$.

\begin{lemma}\label{AMS8.2}
The group $W_{q\frak t}$ is canonically isomorphic to the stabilizer of
$[\cal M, \psi, v, \varrho]_{\cal M^\circ}$ in
either of the groups $W({}^LG,\cal M)$
or $W_{\hat{\frak s}_{\cal M^\circ}}$.
\end{lemma}

\begin{proof}
It is enough to prove the result for $W({}^LG,\cal M)$.
We first construct a canonical map
\[
Z_{W({}^LG,\cal M)}\bigl([\cal M, \psi, v, \varrho]_{\cal M^\circ}\bigr) \to W_{qt},
\]
which will turn out to be injective.
Take $n \in N_{\widehat{G}}(\cal M)$ which fixes $[\psi, v, \varrho]_{\cal M}$.
Then $n \in Z_{\widehat{G}}(\psi) \cdot\cal M^\circ$,
and, since $n$ normalizes $\cal M$,
it actually lies in $Z_{N_{\widehat{G}}(\cal M)}(\psi) \cdot \cal M^\circ$.
Hence we may assume that $n\in Z_{N_{\widehat G}(\cal M)}(\psi)$
after possibly translating by an element of~$\cal M$
(which has no effect on the image of~$n$ in $W({}^LG,\cal M)$).
With this restriction,
the element~$n$ is unique modulo $K^+$.
We claim that $n$ stabilizes $[K^+, [v]_{K^+}, \varrho]_{K^+}$.
Clearly $n$ stabilizes $\cal M^{\circ+}$,
hence stabilizes $K^+ = \cal M^{\circ+}\cap H^+$,
because it stabilizes $\cal M^\circ$.
Therefore $n$ stabilizes the orbit $[v]_{K^+}$
because it stabilizes $[\psi, v]_{K^+}$.
And by construction $n$ preserves $\varrho$
(via the identifications from Lemma \ref{cuspidalparameterlemma}).
Putting this all together produces the desired injective homomorphism.

To finish proving the Lemma, it remains to show that the above map is surjective.
Suppose that $x \in H^+$ stabilizes $[K^+, [v]_{K^+}, \varrho]_{K^+}$. 
Then $x$ normalizes $Z^\circ(K^+)$,
and it therefore also normalizes $Z_{{}^LG}\bigl(Z^\circ(K^+)\bigr)$,
which is just $\cal M$.
By definition, $x$ centralizes~$\psi$ and $[v,\varrho]_{K^+}$.
So $h$ preserves $[\cal M, \psi, v, \varrho]_{\cal M^\circ}$, giving the result.

We leave it to the reader to check compatibilities
needed to pass this isomorphism from $W_{qt}$ to $W_{q\frak t}$.
\end{proof}

\subsection{Twisted extended quotients}
\label{sec:bernstein:quotient}
In this section we recall the definition of twisted extended quotients
and show that $\Phi_\tn{e}^Z({}^LG)$ is a disjoint union of such objects.

Let $\Gamma$ be a finite group acting on a set~$X$.
A twisted extended quotient of $X$ by $\Gamma$
is an enhancement of the ordinary quotient $X/\Gamma$
in which each orbit is replaced by the set
of irreducible representations of the isotropy group,
twisted by a cocycle.

\begin{definition}
A \mathdef{twisted extended quotient} of~$X$ by~$\Gamma$
is a set~$Y$ together with a map $p\colon Y\to X/\Gamma$ and
for each $\bar x\in X/\Gamma$,
\begin{enumerate}
\item a cocycle class
$\bar\kappa_{\bar x} \in H^2(\Gamma_x,C^\times)$ and

\item a bijection $\beta_{\bar x}\colon
p^{-1}(\bar x) \simeq \Irr\bigl(\bbC[\Gamma_x,\bar\kappa_x]\bigr)$.
\end{enumerate}
\end{definition}

Here $\Gamma_x$ is the stabilizer in~$\Gamma$ of~$x$.
The objects $\bar\kappa$ and $\beta$
have the following interpretations.
For $\bar\kappa$, one chooses a lift $x$ of~$\bar x$
; for any other lift~$x'$ one can set $\bar\kappa_{x'} \defeq \gamma_*\bar\kappa_x$
where $\gamma\in\Gamma$ takes $x$ to~$x'$,
and, since $\Gamma_{x}$ acts trivially on $H^2(\Gamma_x,C^\times)$, the class of $\bar{\kappa}_{x'}$ is
independent of the choice of~$\gamma$.
For $\beta$, observe first that for any two cocycles
$\kappa_x,\kappa_x'\in Z^2(\Gamma_x,C^\times)$
lifting~$\kappa_x$, an element of $Z^1(\Gamma_x,C^\times)$
whose coboundary takes $\kappa_x$ to~$\kappa_x'$
yields an algebra isomorphism
$\bbC[\Gamma_x,\kappa_x]\simeq \bbC[\Gamma_x,\kappa_x']$
and thus an identification of $\Irr(-)$ of both objects.
So for a fixed choice of lift $x$ of~$\bar x$
we may unambiguously write $\Irr\bigl(\bbC[\Gamma_x,\bar\kappa_x]\bigr)$,
and one then checks as for~$\kappa$ that varying
the lift~$x$ is permissible.
This could all be phrased more concisely
with an inverse limit construction.

Our definition of twisted extended quotient
is phrased differently from the original one
\cite[\S 2.1]{aubert_baum_plymen_solleveld17a}
but in the end the two definitions are equivalent,
as we now explain.
We chose our variant to clarify the proof of \Cref{thm58} below.

\begin{definition}
\mathdef{Twisted quotient data} for the action
$\Gamma\actson X$ is a pair $(\kappa,\theta)$
consisting of
\begin{itemize}
\item
for each $x\in X$, a 2-cocycle
$\kappa_x\in Z^2(\Gamma_x,C^\times)$, and

\item
for each $x\in X$ and $\gamma\in\Gamma$,
a $\bbC$-algebra isomorphism $\theta_{\gamma,x}
\colon \bbC[\Gamma_x,\kappa_x] \xrightarrow{\sim}
\bbC[\Gamma_{\gamma x},\kappa_{\gamma x}]$,
\end{itemize}
such that
\begin{enumerate}
\item
for all $x\in X$ and $\gamma\in\Gamma$,
the cocycles $\kappa_{\gamma x}$
and $\gamma_*(\kappa_x)$ are cohomologous,

\item
for all $x\in X$ and $\gamma\in\Gamma_x$,
the automorphism $\theta_{\gamma,x}$ is inner, and

\item
for all $\gamma, \gamma' \in \Gamma$ and $x \in X$ we have
$\theta_{\gamma', \gamma x} \circ \theta_{\gamma, x}
= \theta_{\gamma'\gamma, x}$.
\end{enumerate}
\end{definition}

Given twisted quotient data $(\theta,\kappa)$ for $\Gamma\actson X$,
let 
$ \widetilde X_\kappa = \bigl\{(x, \rho) \colon x \in X,
\rho \in \Irr(\bbC[\Gamma_{x}, \kappa_{x}])\bigr\}$.
Use $\theta$ to define an action of $\Gamma$ on $\widetilde X_{\kappa}$ by 
$\gamma \cdot (x, \rho) = (\gamma x, \rho \circ \theta_{\gamma, x}^{-1})$.
The \mathdef{strict twisted extended quotient}
$(X\sslash\Gamma)_\kappa$
of $X$ by $(\Gamma,\kappa,\theta)$
is defined as the quotient of $\widetilde X_\kappa$
by this $\theta$-twisted action of~$\Gamma$.
Although this object depends on~$\theta$ as well as~$\kappa$,
we will follow the convention in the literature
by suppressing~$\theta$ from the notation.

It is not so hard to see why a strict twisted
extended quotient is a twisted extended quotient in our sense:
given $(\kappa,\theta)$, one takes $\bar\kappa_x$ 
to be the cohomology class of~$\kappa_x$
and the bijection $\beta_{\bar x}$ to be that induced by
the gluing maps~$(\theta_{\gamma,x})_{\gamma\in\Gamma}$.
The converse holds as well, and the notions are equivalent.

\begin{lemma}
Every twisted extended quotient
is isomorphic to a strict twisted extended quotient.
\end{lemma}

\begin{proof}
Without loss of generality the action of $\Gamma\actson X$ is transitive.
Given a twisted extended quotient $(Y,p,\bar\kappa,\beta)$,
choose a basepoint $x_0\in X$.
For each $x\in X$, choose $\sigma_x\in \Gamma$
such that $\sigma_x x_0 = x$.
Choose a lift $\kappa_0$ of $\bar\kappa_{x_0}$.
Set $\kappa_x \defeq \sigma_{x*} \kappa_0$
(maps induced on cohomology)
and set $\theta_{\gamma,x} \defeq \sigma_{\gamma x*}\gamma_{x*}^{-1}$
(maps induced on algebras).
One easily checks that $(\kappa,\theta)$
are twisted quotient data for which
$(X\sslash\Gamma)_\kappa\simeq Y$.
\end{proof}

\begin{remark}
The twisted extended quotient
is an enhancement of the ordinary notion of group quotient.
In \Cref{sec:springer} we saw a different enhancement of this notion,
the quotient stack $[X/\Gamma]$.
These two enhancements are related in the following way.

Let $(\kappa,\theta)$ be the trivial family of extended quotient data,
meaning each $\kappa_x$ is the trivial cocycle
and $\theta_{\gamma,x}$ is conjugation by~$\gamma$.
Then there is a tautological identification
\[
(X\sslash\Gamma)_\kappa
\simeq \Irr\bigl(\Perv([X/\Gamma],C)\bigr).
\]
It would be interesting to generalize this comparison
to the case where $(\kappa,\theta)$ is nontrivial.
When $\Gamma$ is not finite, but rather an algebraic group,
one can prove a similar comparison after modifying
the definition of the twisted extended quotient
to use $\pi_0(\Gamma_x)$.
\end{remark}

Let $X = \bigsqcup_{i\in I} X_i$ be a partition of~$X$.
The ordinary quotient $X/\Gamma$ can be formed in two steps,
by first taking the disjoint union of the partial quotients
\[
\bigsqcup_{i\in I} X_i/N_\Gamma(X_i)
\]
of the blocks in the partition,
where $N_\Gamma(X_i)$ is the stabilizer in~$\Gamma$ of the set~$X_i$,
then taking the further quotient by~$\Gamma$.
A similar property holds for twisted extended quotients
under a mild compatibility assumption
between isotropy groups and the partition.
Later, in the proof of \Cref{AMS9.3bis},
this observation will help us to rewrite our
particular twisted extended quotient in a simpler form.

\begin{lemma} \label{thm58}
Let $X = \bigsqcup_{i\in I} X_i$
and for each $i\in I$, let $Y_i$ be
a twisted extended quotient of $X_i$ by $N_\Gamma(X_i)$.
Let $Y\defeq\sqcup_{i\in I} Y_i$.
Suppose that $N_\Gamma(X_i)_x = \Gamma_x$
for every $i\in I$ and $x\in X_i$.
Then the ordinary quotient $Y/\Gamma$
is canonically a twisted extended quotient of~$X$ by~$\Gamma$.
\end{lemma}

\begin{proof}
Taken together, the projection maps $p_i\colon Y_i\to X/\Gamma$
induce a projection map $p\colon Y\to X/\Gamma$.
The assumption on stabilizers implies that
we can reuse the cocycle classes $\bar\kappa_x$
and bijections $\beta_x$ to give
a twisted extended quotient structure on~$Y$.
\end{proof}

We now return to $L$-parameters.
Fix a cuspidal datum
$\hat{s} = (\cal M, \psi, v, \varrho)$ for ${}^LG$.
Since the first component of our cuspidal data
will always be $\cal M$ in what follows,
we omit it from the notation.
From $\hat s$ we would like to make a (strict)
twisted extended quotient of $X =\hat{\frak s}_{\cal M^\circ}$
by the action of $\Gamma = W_{\hat{\frak s}_{\cal M^\circ}}$.
As the group action is clear,
to complete the definition of this quotient
we must define the families of cocycles and gluing maps.

For the cocycles,
set $qt = (K^+, [v]_{K^+}, q\mathcal{E})$
and $t^{\circ} =  (K^{+\circ}, [v]_{K^+\circ}, \mathcal{E})$
with $q\frak t$ and $\frak t^{\circ}$ the corresponding conjugacy classes,
as in \Cref{sec:param:fibers}.
Recall that $\hat{\frak s}_{\widehat G} \in \frak B_Z({}^LG)$
denotes the inertial equivalence class
$\llbracket \cal M, \psi, v, \varrho \rrbracket_{\widehat{G}}$
and $\hat{\frak s}_{\cal M^\circ}$ is the orbit of
$[\hat{s}]_{\cal M^\circ}$
via the action of $X^{*}_{\text{nr}}(\cal M)$. 
For $(\psi', v', \varrho') \in \hat{\frak s}_{\cal M^\circ}$,
define $W_{\hat{\frak s}_{\cal M^\circ}, \psi', v', \varrho'}$
to be its stabilizer inside $W_{\hat{\frak s}_{\cal M^\circ}}$.
Recall that \Cref{AMSLemma5.3} gives, for our fixed $qt$,
a $2$-cocycle $\kappa_{qt}$ of $W_{qt}$,
and that, by \Cref{AMS8.2},
there is a canonical identification $W_{qt} \xrightarrow{\sim}
W_{\hat{\frak s}_{\cal M^\circ}, \psi, v, \varrho}
\subset W_{\hat{\frak s}_{\cal M^\circ}}$.
If one replaces $\hat{s}$ with an unramified twist
$\hat{s}' = (z^{\val_{F}} \cdot \psi, v, \varrho)$
then $qt$ is unchanged,
so by the same \nameCref{AMSLemma5.3} we also get an identification
$W_{\hat{\frak s}_{\cal M^\circ}, z ^{\val_{F}} \cdot \psi, v, \varrho}
\xrightarrow{\sim} W_{qt}$.
For any $z$, let $\kappa_{z^{\val_{F}} \cdot \psi, v, \varrho}$
denote the $2$-cocycle of
$W_{\hat{\frak s}_{\cal M^\circ}, z^{\val_{F}} \cdot \psi, v, \varrho}$ 
obtained from $\kappa_{qt}$ by transfer along this isomorphism.
These are the cocycles we will take
for the extended quotient.

As for the gluing maps, they are given by the following result.

\begin{lemma}
For $w \in W_{\hat{\frak s}_{\cal M^\circ}}$
and $z \in X^{*}_{\tn{nr}}(\cal M)$ with
$w \cdot \hat{s} = (z^{\val_{F}} \cdot \psi, v, \varrho)$
there is a family of algebra isomorphisms (as we vary $w,z$)
\[
\theta_{w, \psi, v, \varrho} \colon
\bbC[W_{\hat{\frak s}_{\cal M^\circ}, \psi, v, \varrho}, \kappa_{\psi, v, \varrho}]
\xrightarrow{\sim}  \bbC[W_{\hat{\frak s}_{\cal M^\circ}, z^{\val_{F}} \cdot \psi, v, \varrho},
\kappa_{z^{\val_{F}} \cdot \psi, v, \varrho}]
\]
satisfying the following properties.
\begin{enumerate}

\item
If $w \in W_{\hat{\frak s}_{\cal M^\circ}, \psi, v, \varrho}$
then $\theta_{w, \psi, v, \varrho}$ is conjugation by $w$.

\item
$\theta_{w', z^{\val_{F}} \cdot \psi, v, \varrho}
\circ \theta_{w, \psi, v, \varrho} = \theta_{w'w, \psi, v, \varrho}$
for all $w, w' \in W_{\hat{\frak s}_{\cal M^\circ}}$.

\item
There is a natural identification (of $\widehat{G}$-conjugacy classes of
enhanced $L$-parameters of $\widehat{G}$)
\[
^{L}\Sigma_{q\frak t, \psi, v}^{-1}(\rho) \cong ^{L}\Sigma^{-1}_{q\frak t}
(\rho \circ \psi^{-1}_{w, \psi, v, \varrho})
\]
for all $\rho \in \Irr(\bbC[W_{\hat{\frak s}_{\cal M^\circ}, \psi, v, \varrho},
\kappa_{\psi, v, \varrho}])$.

\end{enumerate}
\end{lemma}

\begin{proof}
The same proof as in \cite[9.2]{AMS18} works here,
replacing their $L_{\tn{sc}}$ with $L^+$.
\end{proof}

We define 
\[
\Phi_\tn{e}^Z(G)^{\hat{\frak s}_{\widehat G}} \defeq
\bigl\{(\phi, u, \rho) \in \Phi_\tn{e}^Z(G)  \mid
\llbracket \Cusp(\phi, u, \rho) \rrbracket_{\widehat{G}}
= \hat{\frak s}_{\widehat G}\bigr\}.
\]
Note that we could also replace ``$\Cusp$" with ``$\uCusp$" in the above definition.
Our main theorem parameterizes the set $\Phi_\tn{e}^Z(G)^{\hat{\frak s}_{\widehat G}}$.

\begin{theorem}\label{AMS9.3}
Let $\hat{\frak s}_{\cal M^\circ}$ denote the inertial orbit of the cuspidal datum
$(\cal M, \psi, v, \varrho)$ with $qt$ as defined above.
Then the maps $^{L}\Sigma_{q\frak t, \psi, v}$ from \Cref{AMS9.1} give a bijection
\begin{align*}
\Phi_\tn{e}^Z({}^LG)^{\hat{\frak s}_{\widehat G}} &\longleftrightarrow
(\Phi_\tn{e}^Z(\cal M)^{\hat{\frak s}_{\cal M^\circ}}
\sslash W_{\hat{\frak s}_{\cal M^\circ}})_{\kappa} \\
(\phi, u, \rho) &\longmapsto \bigl(\uCusp(\phi, u, \rho), q\Sigma_{q\frak t}(u, \rho)\bigr) \\
(\psi, q\Sigma_{q\frak t}^{-1}(\tau)) &\longmapsfrom \bigl((\cal M, \psi, v, \varrho), \tau\bigr).
\end{align*}
The bijection %
preserves boundedness of $L$-parameters
and the restriction of $\tau$ to $W_{t^{\circ}}$ canonically determines
the (non-enhanced) $L$-parameter in $^{L}\Sigma_{q\frak t}^{-1}(\tau)$.
\end{theorem}

\begin{proof}
We will explain why this map is well-defined.
When we write $\uCusp(\phi, u, \rho)$ in the output pair,
we mean that we choose a representative of $\uCusp(\phi, u, \rho)$
that lies in $\Phi_{\tn e}^Z(\cal M)^{\hat{\frak s}_{\cal M^\circ}}$
and consider it up to $\cal M^\circ$-conjugacy.
Two different choices of $\widehat{G}$-conjugate representatives
are in the same $W_{\hat{\frak s}_{\cal M^\circ}}$-orbit,
via some element that also identifies the corresponding twisted group algebra representations
(by construction, cf.~\Cref{AMS9.1}),
giving equality in the twisted extended quotient.
The rest of the proof follows from an identical argument as in \cite[9.3]{AMS18}.
\end{proof}

Assembling the bijections of \Cref{AMS9.3} for all~$\hat{\frak s}$,
one obtains a description of $\Phi_\tn{e}^Z(G)$
as a disjoint union of twisted extended quotients,
which we leave to the reader to write out explicitly.
Intermediate between these two extremes,
one can consider the set
\[
\Phi_\tn{e}^Z({}^LG,\cal M)
\]
of $L$-parameters with cuspidal support
in a given Levi $L$-subgroup~$\cal M$.
A priori $\Phi_\tn{e}^Z({}^LG,\cal M)$ is a union
of twisted extended quotients for different groups,
namely, the stabilizers in $W({}^LG,\cal M)$
of various components of $\Omega_Z({}^LG)$.
Using \Cref{thm58}, one can show that
the components can be grouped
so that $\Phi_\tn{e}^Z({}^LG,\cal M)$
is a twisted extended quotient
for the group $W({}^LG,\cal M)$.

\subsection{Functoriality in \texorpdfstring{$Z$}{Z}}
\label{sec:changeofZ}
We conclude by briefly explaining how taking cuspidal support is compatible with changing the finite central subgroup $Z$. Given two such subgroups $Z \to Z' \to G$, there is a dual central isogeny $\widehat{G}^{+'} \xrightarrow{f} \widehat{G}^{+}$, where $\widehat{G}^{+'} := \widehat{G/Z'}$, and we use the $+'$-superscript to denote preimages in the latter group. We also have, for any $Z$-enhanced $L$-parameter $(\phi, u, \rho)$ for $G$, a central isogeny $H^{+'}:= Z_{\widehat{G}^{+'}}(\phi) \to Z_{\widehat{G}^{+}}(\phi)$ which by abuse of notation we also denote by $f$.

For a fixed $\mathcal{F} \in \Irr(\Perv(\cal{U}_{H^{+}}/H^{+}))$, the fact that $f$ is a central isogeny implies that the pullback $f^{*}\mathcal{F}$ is in $\Irr(\Perv(\cal{U}_{H^{+'}}/H^{+'}))$, and if $(K^{+}, v, \varrho)$ lies in the quasi-cuspidal support of $\mathcal{F}$, then we claim that $(K^{+'}, v, f^{*}\varrho)$ lies in the quasi-cuspidal support of $f^{*}\mathcal{F}$. To see this, we first note that $f^{*}$ induces a bijection between the quasi-parabolic subgroups of $H^{+}$ and $H^{+'}$, and so the $f$-pullback of the quasi-cuspidal support of $\mathcal{F}$ remains quasi-cuspidal. Moreover, for a fixed quasi-parabolic subgroup $P$ of $H^{+}$ with Levi subgroup $L$ and $L$-equivariant local system $\mathcal{E}$ supported on $\cal{O}$, any nonzero morphism $$\ind_{L\subseteq P}^{H^+}(\IC(\cal{O}, \cal{E})) \to \mathcal{F}$$ yields a nonzero morphism $$f^{*}[\ind_{L\subseteq P}^{H^+}(\IC(\cal{O}, \cal{E}))] \to f^{*}\mathcal{F},$$ and we deduce the claim via the identification
$$f^{*}[\ind_{L\subseteq P}^{H^+}(\IC(\cal{O}, \cal{E}))] = \ind_{f^{*}L\subseteq f^{*}P}^{H^{+'}}(\IC(f^{*}\cal{O}, f^{*}\cal{E})).$$

After this observation that the $H^{+}$ quasi-cuspidal support of $(u, \rho)$ agrees with the $H^{+'}$ quasi-cuspidal support of $(u, f^{*}\rho)$ it is a routine verification to deduce that the (normalized or unnormalized) cuspidal support of $(\phi, u, \rho)$ agrees with that of $(\phi, u, f^{*}\rho)$.

\section{The local Langlands conjectures}
\label{sec:corr}
In this section we finally come to the local Langlands conjectures.
The heart is \Cref{sec:corr:conj}, where we summarize
the expected relationship between the local Langlands correspondence
and the cuspidal support map.
This is preceded by two sections reviewing the setup
of this correspondence for rigid inner forms
and followed by two sections checking certain compatibilities
with change of rigidification or Whittaker datum.
We finish with a short application to non-singular representations.

Let $G$ be a quasi-split reductive $F$-group
and let $Z$ be a finite central subgroup of~$G$.
For brevity, we sometimes use here the more common model of an $L$-parameter
as a homomorphism $\varphi\colon W_F\times\SL_2\to{}^LG$ (see~\Cref{sec:lparam}).

\subsection{Gerbe cohomology and rigid inner twists}
\label{sec:corr:gerbe}
Rigid inner twists are certain enhancements of inner twists
that conveniently organize $L$-packets of different groups at once;
we refer the reader to \cite{Kal16}, especially the introduction,
for history and further discussion.
In this section we recall the definition of these objects.
It is subtle in full generality, when $F$ has positive characteristic.

\begin{definition}
Recall from \cite{Dillery} that the profinite group scheme 
\[
u = \varprojlim_{E/F,n} \frac{\text{Res}_{E/F}(\mu_{n})}{\mu_{n}}
\]
satisfies $H^{2}_{\text{fppf}}(F, u) = \widehat{\mathbb{Z}}$. We choose an fpqc $u$-gerbe $\mathcal{E} \to \mathrm{Sch}/F$ representing the class of $-1$ and we call this gerbe the \mathdef{Kaletha gerbe}.
\end{definition}

\begin{remark} To work with the above object in a more concrete manner, one chooses a \v{C}ech $2$-cocycle $a \in u(\overline{F}\otimes_F\overline{F}\otimes_F\overline{F})$ (so in characteristic zero, this is a $2$-cocycle of the absolute Galois group $\Gamma$ in $u(\overline{F})$) and uses the gerbe $\mathcal{E}_{a} \to \mathrm{Sch}/F$ of $a$-twisted $u$-torsors (cf.~\cite[\S 2.4]{Dillery}).
\end{remark}

\begin{definition} Denote by $H^{1}(\mathcal{E}, G)$ all isomorphism classes of fpqc $G$-torsors on $\mathcal{E}$, and denote by $$H^{1}(\mathcal{E}, Z \to G) \subset H^{1}(\mathcal{E}, G)$$ the subset of all (isomorphism classes of) torsors $\mathcal{T}$ such that $\mathcal{T} \times^{G} G/Z$ descends to a $G/Z$-torsor over~$F$. We call $\mathcal{T}$ as above a \mathdef{$Z$-twisted $G_{\mathcal{E}}$-torsor}, and denote the set of all such torsors (not their equivalence classes) by $Z^{1}(\mathcal{E}, Z \to G)$.
\end{definition}

Note that, via pullback, there is a canonical inclusion $H^{1}(F, G) \to H^{1}(\mathcal{E}, Z \to G)$.

\begin{remark} A concrete way of interpreting the above definition is as follows: Using $\mathcal{E} = \mathcal{E}_{a}$, we can identify elements of $H^{1}(\mathcal{E}, Z \to G)$ with equivalence classes of pairs $(c, f)$, where $f \colon u \to Z$ is a morphism of algebraic groups and $c \in G(\overline{F}\otimes_F\overline{F})$ (i.e., a ``\v{C}ech $1$-cochain"---when $F$ has characteristic zero, this is a genuine $1$-cochain) satisfying $dc = f(a)$. We call such a pair an \mathdef{$a$-twisted ($1$-)cocycle}. We also remark that this set only depends on the choice of representative $\mathcal{E}$ up to canonical isomorphism, so there is no problem in making such a choice.
\end{remark}

A key result that connects the above constructions to the Galois side of the Langlands correspondence is the following:

\begin{theorem}\label{RigidTN}(\cite[Theorem 4.11]{Kal16}, \cite[Theorem 5.10]{Dillery}) There is a canonical isomorphism
\[
H^{1}(\mathcal{E}, Z \to G) \xrightarrow{\sim} \pi_{0}\bigl(Z(\widehat G)^{\Gamma,+}\bigr)^{*}.
\]
The restriction of this isomorphism to $H^{1}(F, G)$ recovers Kottwitz's duality result (from \cite{kot86}).
\end{theorem}

Denote by $$p_{1},p_{2} \colon \mathrm{Spec}(\overline{F}) \times \mathrm{Spec}(\overline{F}) \to \mathrm{Spec}(\overline{F})$$ the two canonical projections. Let $(G', \xi)$ be an inner twist of $G$; recall that this means that $\xi$ is an isomorphism $G_{\overline{F}} \xrightarrow{\sim} G'_{\overline{F}}$ such that the automorphism $p_{1}^{*}\xi^{-1} \circ p_{2}^{*}\xi$ of $G_{\overline{F} \otimes_{F} \overline{F}}$ is inner. One could also replace $\overline{F}$ by the separable closure $F^{\text{sep}}$ and rephrase the above condition more concretely as requiring that $\xi^{-1} \circ ^\gamma\xi$ be an inner automorphism of $G_{F^{\text{sep}}}$ for every $\gamma \in \Gamma$. We say that two inner twists $(G_{1}, \xi_{1})$, $(G_{2}, \xi_{2})$ are equivalent if there is an isomorphism $f \colon G_{1} \xrightarrow{\sim} G_{2}$ such that $\xi_{2}^{-1} \circ f \circ \xi_{1}$ is inner. It is a basic fact that inner twists (up to equivalence) are in canonical bijection with the set $H^{1}(F, G_{\text{ad}})$ via the obvious map. 

A more compact way of describing $(G',\xi)$ as an inner twist of $G$ up to equivalence is as an element of $[\underline{\mathrm{Isom}}(G, G')/G_{\text{ad}}](F)$, where $G_{\text{ad}}$ acts via precomposing by inner automorphisms. The fiber over this point in $\underline{\mathrm{Isom}}(G, G')$ is a $G_{\text{ad}}$-torsor, denoted by $T_{\xi}$ (and its image in $H^{1}(F, G_{\text{ad}})$ is as expected).

\begin{proposition}(\cite[Corollary 3.8]{Kal16}, \cite[Proposition 5.12]{Dillery}) For every $z \in H^{1}(F, G_{\text{ad}})$ there is some finite central $Z'$ such that $z$ lies in the image of the canonical map
$$H^{1}(F, Z' \to G) \to H^{1}(F, G/Z') \to H^{1}(F, G_{\text{ad}}),$$
where the left-hand map sends $\mathcal{T}$ to (the descent of) $\mathcal{T} \times^{G} G/Z'$.
\end{proposition}

We can finally give the main definition:

\begin{definition} A \mathdef{$Z$-rigid inner twist} of $G$ is a tuple $(G', \xi, \mathcal{T}, \bar{h})$,  where $(G', \xi)$ is an inner twist of $G$, $\mathcal{T}$ is a $Z$-twisted $G_{\mathcal{E}}$-torsor, and $\bar{h}$ is an isomorphism of $G_{\text{ad}}$-torsors $$\mathcal{T} \times^{G} G_{\text{ad}} \xrightarrow{\sim} T_{\xi}.$$
We say that two rigid inner twists $(G_{1}, \xi_{1}, \mathcal{T}_{1}, \bar{h}_{1})$, $(G_{2}, \xi_{2}, \mathcal{T}_{2}, \bar{h}_{2})$ are \mathdef{equivalent} if there is an isomorphism of $G_{\mathcal{E}}$-torsors $\mathcal{T}_{1} \xrightarrow{\sim} \mathcal{T}_{2}$.
\end{definition}

For any $Z$-rigid inner twist $(G', \xi, \mathcal{T}, \bar{h})$ of $G$,
the underlying inner form $G'$ depends only on~$\cal T$ up to isomorphism,
and we denote it by $G^{\cal T}\defeq G'$.

\begin{remark}
For convenience, we have fixed a choice of~$Z$ in this paper.
However, it is possible to organize together all
rigid inner forms as $Z$ varies \cite[\S 5.1]{Kal16}.
This point motivates in part our short \Cref{sec:changeofZ}.
\end{remark}

One checks easily that equivalent rigid inner twists have equivalent underlying inner twists.
Moreover, changing only the isomorphism $\bar{h}$ always yields an equivalent rigid inner twist.

We conclude this subsection by briefly discussing Levi subgroups of rigid inner twists.

\begin{definition}
For $(G',\xi, \mathcal{T},\bar h)$ a ($Z$-)rigid inner twist of $G$,
a \mathdef{Levi subgroup} of $(G',\xi, \mathcal{T},\bar h)$
is a rigid inner twist $(M',\xi_{M}, \mathcal{T}_M,\bar h_M)$
of a Levi subgroup $M$ of~$G$
such that $\mathcal{T} = \mathcal{T}_{M} \times^{M} G$, $T_{\xi} = T_{\xi_{M}} \times^{M} G_{\text{ad}}$ and $\bar{h} = \bar{h}_{M} \times^{M} \text{id}_{G_{\text{ad}}}$.
\end{definition}

In other words, a Levi subgroup of a rigid inner twist
is simply a rigid inner twist of a Levi subgroup
of the underlying inner form that induces the given
rigid inner twisting of the ambient group.
Since we require strictly equality of torsors,
this definition is not invariant under equivalence
of rigid inner twists of the Levi.

Let $(M', P')$ be a fixed Levi pair of an inner twist $(G', \xi)$. Recall that there is a Levi pair $(M, P)$ of $G$ such that $(\xi(M),\xi(P))$ is $G'(F^\tn{sep})$-conjugate to $(M',P')$, and this pair $(M,P)$ is unique up to $G(F)$-conjugacy. We say $(M,P)$ \mathdef{corresponds} to $(M',P')$. Similarly, if we fix only $M'$, and not~$P'$, then we say that a Levi subgroup $M$ of $G$ \mathdef{corresponds} to $M'$ if there exist $P$ and $P'$ as above.
We then have the following analogue of Lemma 2.4 in \cite{ABMUnitary},
whose proof is essentially the same.
 
 \begin{lemma}\label{LeviRIT} Let $(G',\xi, \mathcal{T},\bar h)$ be an enrichment of $(G, \xi)$ to a $Z$-rigid inner twist of $G$ and let $M$ be a Levi subgroup of $G$ corresponding to $M' \subset G'$. Then, up to equivalence of rigid inner twists, there exists a unique enrichment of $M'$ to a Levi subgroup $(M',\xi_{M}, \mathcal{T}_M,\bar h_M)$ of $(G',\xi, \mathcal{T},\bar h)$.
 \end{lemma}

As above, we say this enrichment of the Levi subgroup of $(G',\xi, \mathcal{T},\bar h)$ \mathdef{corresponds} to~$M'$.
 
 \begin{proof} Pick $P$ and $P'$ as above. We may assume that $\mathcal{E} = \mathcal{E}_{a}$ for a \v{C}ech $2$-cocycle $a$ in order to view $(\mathcal{T}, \bar{h})$ as an $a$-twisted cocycle $(f, c)$. Since we are working up to equivalence, we may assume that $\xi(P) = P'$, which, using that $M, M', P$, and $P'$ are defined over $F$, implies that the image $\bar{c} \in G_{\tn{ad}}(\ov{F} \otimes_{F} \ov{F})$ of~$c$ normalizes $p_{1}^{*}P$. So $\bar{c} \in (P/Z)(\ov{F} \otimes_{F} \ov{F})$ and hence also $\bar{c} \in (M/Z)(\ov{F} \otimes_{F} \ov{F})$, since it normalizes $p_{1}^{*}M$ as well.
 
 It is clear that $\bar{c} \in (M/Z)(\ov{F} \otimes_{F} \ov{F})$ implies that $c$ lies in the image of $Z^{1}(\mathcal{E}, Z \to M) \to  Z^{1}(\mathcal{E}, Z \to G)$, as desired. By twisting, uniqueness reduces to the injectivity of $H^{1}(F, M) \to H^{1}(F, G)$, which is known.
 \end{proof}

\subsection{Relevance}
Fix a rigid inner twist $(G',\xi, \mathcal{T},\bar h)$ of $G$,
which we abbreviate in this section by setting
\[
\underline{x} \defeq (G',\xi,\mathcal{T},\bar h).
\]
After recalling the notion of relevance for enhanced $L$-parameters,
we will explain how it interacts with the constructions
of \Cref{sec:bernstein}.

\begin{definition}
A $Z$-enhanced $L$-parameter $(\phi,u,\rho)$ of ${}^LG$ is \mathdef{relevant} for $\underline{x}$ if the image of $[\mathcal{T}] \in H^{1}(\mathcal{E}, Z \to G)$ via the duality isomorphism of \Cref{RigidTN} equals the character of $\pi_0\bigl(Z(\widehat{G})^{\Gamma,+}\bigr)$ induced by $\rho$. We may also refer to $(\phi, u, \rho)$ as an \mathdef{enhanced $L$-parameter} of $\underline{x}$. Denote the set of all such enhanced $L$-parameters by $\Phi_{\tn e}^Z(G', \mathcal{T})$.
\end{definition}

Here the character induced by $\rho$ is defined as the pullback of~$\rho$ along the canonical map
\[
\pi_{0}(Z(\widehat{G})^{\Gamma,+}) \to \pi_{0}(Z_{\widehat{G}^{+}}(\phi,u));
\]
since this map has central image, the pullback is isotypic for a single character.
The property of relevance is obviously preserved by $\widehat{G}$-conjugacy, unramified twisting, inertial equivalence, and equivalence of rigid inner twists.
Our abbreviated notation $\Phi_\tn{e}^Z(G',\mathcal T)$ reflects that the equivalence class of $\underline{x}$ can be recovered from $(G',\cal T)$.

\begin{remark}
There is a related notion of relevance
of $L$-parameters in the classical,
nonenhanced Langlands correspondence
\cite[3.3, 8.2]{borel_corvallis}.
The classical notion is compatible
with the notion for enhanced $L$-parameters
in that if $(\phi,u,\rho)$ is relevant for $\underline{x}$
then $(\phi,u)$ is relevant for~$G'$ \cite[\S5.5]{Kal16}.

In the classical, nonenhanced Langlands correspondence,
an $L$-parameter can be relevant for many different inner forms:
for example, a discrete $L$-parameter is relevant for every inner form.
One advantage of the enhanced setting is that
an enhanced $L$-parameters remembers
the (rigid) inner form whose representation it parameterizes.
In other words, an enhanced $L$-parameter
is relevant for a unique rigid inner form.
\end{remark}

\begin{remark} \label{thm50}
This remark explains why the enhancement of an $L$-parameter
as defined in the introduction to \cite{AMS18}
is the same as an enhancement in our sense, and how the two associated notions of relevance are related.
Take $Z= Z(\widehat G_\tn{der})$ and fix an explicit $\xi \in Z^{1}(F, G_{\tn{ad}})$
realizing $G'$ as an inner twist of $G$. 

In \cite{AMS18}, an enhanced $L$-parameter consists of an $L$-parameter~$\varphi$ and an irreducible representation of the group $\pi_{0}(S^{\tn{sc}}_{\varphi})$, where $S_\varphi \defeq Z_{\widehat G}(\varphi)$ and $S^{\tn{sc}}_{\varphi}$ is the preimage in $\widehat{G}_{\tn{sc}}$ of the image of $S_{\varphi}$ in $\widehat{G}_{\tn{ad}}$. For a fixed character $\zeta$ of $Z(\widehat G_\tn{sc})^\Gamma$ corresponding to $(G', \xi)$ and a choice of extension $\zeta_{\xi}$ of $\zeta$ to $Z(\widehat G_\tn{sc})$, the set of \mathdef{enhanced $L$-parameters associated to $\zeta_{\xi}$} is the set $\Phi_{\tn e,\zeta_{\xi}}$ of enhanced $L$-parameters $(\varphi, \rho)$ of ${}^LG$ such that the restriction of the central character of $\rho$ to $Z(\widehat G_\tn{sc})$ is $\zeta_{\xi}$.

Now fix $\underline{x}$ a rigid inner twist enriching $(G', \xi)$;
denote by $\zeta_{\mathcal{T}}$ the corresponding character of $Z(\widehat{G})^{\Gamma,+}$.
We can pick an element $\mathcal{T}^\tn{sc} \in Z^{1}(\cal{E}, Z(\widehat G_\tn{sc}) \to G_{\tn{sc}})$
lifting $\xi$, and, since $\widehat{G_\tn{sc}} = \widehat{G}_\tn{ad}$,
the duality isomorphism becomes
\[
H^{1}(\cal{E}, Z(\widehat G_{\tn{sc}}) \to G_{\tn{sc}}) \xrightarrow{\sim} Z(\widehat G_\tn{sc})^*,
\]
so that $\mathcal{T}^\tn{sc}$ gives rise to a character
$\zeta_{\mathcal{T}^\tn{sc}}$ of $Z(\widehat G_\tn{sc})$.
Then, as explained in \cite{kaletha18a}, there is a (non-canonical) isomorphism 
\[
S_{\varphi}^{\tn{sc}} \xrightarrow{\sim} S_{\varphi}^{+}
\times^{Z(\widehat{G})^{\Gamma,+}} Z(\widehat{G})_{\tn{sc}},
\]
which induces a canonical bijection 
\[
\tn{Irr}(\pi_{0}(S_{\varphi}^{+}), \zeta_{\mathcal{T}}) \xrightarrow{\sim} \tn{Irr}(\pi_{0}(S_{\varphi}^{\tn{sc}}), \zeta_{\mathcal{T}^\tn{sc}})
\]
via $\rho \mapsto \rho \otimes \zeta_{\mathcal{T}}$.
It is thus natural to define $\Phi_{\tn e,\zeta_{\mathcal{T}}}$ to be those
enhanced $L$-parameters of ${}^LG$ (using our \Cref{enhancedparamdef})
such that $\rho$ induces the character $\zeta_{\mathcal{T}}$.
By construction, this set is in canonical bijection with $\Phi_{\tn e,\zeta_{\mathcal{T}^\tn{sc}}}$
as defined in \cite{AMS18} and recalled in the previous paragraph.

This comparison is the basis of the approach in \cite[\S 7]{Solleveld}
to extending the cuspidal support map to $Z(G_\tn{der})$-enhanced $L$-parameters.
\end{remark}

Now for the comparison with \Cref{sec:bernstein}.
First is the cuspidal support map.

\begin{proposition}\label{AMS7.3bis}
Fix an $\underline{x}$-relevant enhanced $L$-parameter $(\phi, u, \rho)$
and let 
\[
\uCusp\bigl([\phi, u, \rho]_{\widehat{G}}\bigr)
= [{}^LM, \psi, v, \varrho]_{\widehat{G}}
\]
where $M$ is a Levi subgroup of~$G$.
After possibly replacing $\underline{x}$ by an equivalent rigid inner twist, there is a Levi subgroup $(M', \xi_{M}, \mathcal{T}_{M}, \bar{h}_{M})$ of $\underline{x}$ such that $M$ corresponds to $M'$ and the cuspidal enhanced $L$-parameter $(\psi, v, \varrho)$ for ${}^LM$ is relevant for $(M', \xi_{M}, \mathcal{T}_{M}, \bar{h}_{M})$. Moreover, this Levi subgroup is unique up to equivalence of rigid inner twists of $M$
\end{proposition}

\begin{proof} It is straightforward to check (after possibly replacing $\xi$ by an $M(F^{\text{sep}})$-conjugate, which preserves equivalence) that the cuspidal enhanced $L$-parameter $(\psi, v, \varrho)$ for ${}^LM$ is relevant for the rigid inner twist $\underline{x}_{M} \defeq (\xi(M), \xi |_{M_{F^{\text{sep}}}}, \mathcal{T}_{M}, \overline{h}')$, where we temporarily abuse notation and let $\xi(M)$ denote the descent of $\xi(M_{F^{\text{sep}}})$ to an $F$-subgroup of $G'$. Set $\xi_{M} \defeq \xi |_{M_{F^{\text{sep}}}}$ and set $M'$ to be this descent. Recall that $\bar{h}'$ is a choice of isomorphism $\mathcal{T}_{M} \times^{M} M_{\text{ad}} \xrightarrow{\sim} T_{\xi_{M}}$.

Now \Cref{thm53} exactly says that the character of $\pi_{0}(Z(\widehat{M})^{\Gamma,+})$ induced by $\varrho$ (which we can identify with the isomorphism class of $\mathcal{T}_{M}$) has the same restriction to $\pi_{0}(Z(\widehat{G})^{\Gamma,+})$ as $\rho$, and so it follows that, after possibly replacing $\underline{x}$ by an equivalent rigid inner twist, that $\underline{x}_{M}$ is the desired Levi subgroup.
\end{proof}

Next are twisted extended quotients.
For a Levi subgroup $M'$ of $G'$, define
\[
\Phi_{\tn e}^Z(G', \mathcal{T}, M')
\defeq \{[\phi, u, \rho]_{\widehat{G}} \in \Phi_{\tn e}^Z(G',\mathcal{T}) 
\mid \Cusp(\phi, u, \rho) \in \Phi_\tn{e,cusp}^Z(M', \mathcal{T}_{M})\},
\]
where $(M', \xi_{M}, \mathcal{T}_{M}, \bar{h}_{M})$ is a Levi subgroup of a rigid inner twist equivalent to $\underline{x}$ as in Lemma \ref{LeviRIT}. Taking $M \subset G$ corresponding to $M'$, as the notation suggests, $\Phi_\tn{e,cusp}^Z(M', \mathcal{T}_{M})$ denotes the set of cuspidal enhanced $L$-parameters in $\Phi_\tn{e}^Z(M', \mathcal{T}_{M})$.  When we write $\Cusp(\phi, u, \rho) \in \Phi_{\tn{cusp}}(M', \mathcal{T}_{M})$, we mean that there is some element in $\Cusp(\varphi, \rho)$ (an equivalence class of $\widehat{G}$-conjugates) contained in $\Phi_\tn{e,cusp}^Z(M', \mathcal{T}_{M})$.

We can refine our Theorem \ref{AMS9.1} to a statement about these relevant $L$-parameters:

\begin{theorem}\label{AMS9.3bis}
The bijections of Theorem \ref{AMS9.3} induce
a twisted extended quotient structure
\[
\Phi_\tn{e}^Z(G', \mathcal{T}, M')  \simeq
(\Phi_\tn{e,cusp}^Z(M', \mathcal{T}_{M}) \sslash W({}^LG, {}^LM))_\kappa.
\]
\end{theorem}

\begin{proof} 
This follows immediately from \Cref{thm58}.
The compatibility condition there is satisfied here
because of the variety structure on~$\Omega_Z(\cal M)$:
any automorphism of a variety that stabilizes
a given point in a connected component
must stabilize the entire connected component.
\end{proof}

An immediate consequence is:

\begin{corollary}\label{AMS9.3bisbis}
Let $\Lev(G')$ be a set of representatives of $G'(F)$-conjugacy classes of Levi subgroups of~$G'$. 
\begin{enumerate}

\item
The maps from Theorem \ref{AMS9.3bis} combine to give a bijection
\[
\Phi_\tn{e}^Z(G', \mathcal{T}) \leftrightarrow \bigsqcup_{M' \in \Lev(G')}  (\Phi_\tn{e,cusp}^Z(M', \mathcal{T}_{M}) \sslash W({}^LG, {}^LM))_{\kappa}.
\]

\item
With this notation, the maps from part (e) combine to give a bijection
\[
\Phi_\tn{e}^Z({}^LG) \leftrightarrow
\bigsqcup_{\substack{[\mathcal{T}] \in H^{1}(\mathcal{E}, Z \to G)\\
M'\in\Lev(G^\mathcal{T})}} (\Phi_\tn{e,cusp}^Z(M', \mathcal{T}_{M}) \sslash W({}^LG, {}^LM))_{\kappa},
\]
where, as before, $M$ is a Levi subgroup of $G$ corresponding to $M'$.
\end{enumerate}
\end{corollary}

We remark that for $Z \to Z' \to G$ two finite central subgroups, there is a canonical inclusion $$H^{1}(\mathcal{E}, Z \to G) \to H^{1}(\mathcal{E}, Z' \to G),$$ as well as a map $Z_{\widehat{G}^{+'}}(\varphi) \to Z_{\widehat{G}^{+}}(\varphi)  $. It is easy to check that all of our constructions from this section are compatible with these maps (cf.~\cref{sec:changeofZ}).

\subsection{Conjectures}
\label{sec:corr:conj}
In this section we explain how the cuspidal support map
and twisted extended quotients of \Cref{sec:bernstein}
are expected to fit into the local Langlands conjectures.

We first briefly recall Kaletha's rigid formulation of the local Langlands conjectures.
Recall that a \mathdef{representation} of a rigid inner twist
$(G',\xi, \mathcal{T},\bar h)$ of $G$ is a tuple
$\dot\pi = (G',\xi, \mathcal{T},\bar h, \pi)$,
where $\pi$ is a representation of $G'(F)$.
We say that two such representations are \mathdef{isomorphic}
if there is an equivalence between the two rigid inner twists which identifies the two representations,
and write $\Pi^{\rig,Z}(G)$ for the set of isomorphism classes
of irreducible representations of $Z$-rigid inner twists of $G$.
It is a feature of rigid inner twists---one not shared, for instance, by inner twists---
that two such isomorphic representations are isomorphic
as representations of $p$-adic groups in the usual sense.

For the most part, the representation theory of rigid inner twists
is the same as the ordinary representation theory of reductive $F$-groups.
For $\cal P$ a property of ordinary representations,
such as smooth, irreducible, supercuspidal, discrete series, or tempered,
we define $\dot\pi$ to satisfy $\cal P$
if its underlying ordinary representation $\pi$ does.
So the rigid enhancement is not so relevant
to the study of an individual group.
Rather, the enhancement collects representations
of different groups, by varying the rigid inner twist,
that the local Langlands conjectures
predict should be collected together.
We have already seen this pattern in the definition of~$\Pi^{\rig,Z}(G)$.
Similarly, one can write
\[
\Omega_{\rig,Z}(G) \defeq
\bigsqcup_{[\cal T] \in H^{1}(\mathcal{E}, Z \to G)}
\Omega(G^{\cal T})
\]
for the union of the Bernstein varieties
of the rigid inner twists of~$G$ and 
\[
\frak B_{\rig,Z}(G) \defeq \pi_0\bigl(\Omega_{\rig,Z}(G)\bigr)
= \bigsqcup_{[\cal T] \in H^{1}(\mathcal{E}, Z \to G)}
\frak B(G^{\cal T})
\]
for the inertial supports of the rigid inner twists.

\begin{definition}
Let $\frak w$ be a Whittaker datum for~$G$.
A (weak\footnote{
The word ``weak'' here refers to our omission of the conditions
that ought to be satisfied by the the ``true'' correspondence.},
compound) \mathdef{local Langlands
correspondence} for~$G$ is a bijection
$\rec_{\frak w}\colon\Pi^\rig(G) \to\Phi_\tn{e}^Z(G)$
making the following diagram commute:
\[
\begin{tikzcd}
\Pi^{\rig,Z}(G) \rar{\rec_{\frak w}}[swap]{\sim} \dar &
\Phi_\tn{e}^Z({}^LG) \dar \\
H^1(\cal E,Z\to G) \rar{\kappa_G}[swap]{\sim} &
\pi_0\bigl(Z(\widehat G)^{\Gamma,+}\bigr)^*.
\end{tikzcd}
\]
\end{definition}

Here the left vertical arrow passes to
the cohomology class parameterizing the underlying rigid inner twist,
the right vertical arrow is restriction of~$\rho$,
and the left horizontal arrow is Kaletha's generalization
of the Tate--Nakayama duality isomorphism.
We will comment on the role of Whittaker data
in \Cref{sec:corr:whittaker},
explaining in particular how to compatibly
choose a Whittaker datum $\frak w_M$
for a Levi subgroup~$M$ of~$G$
in the local Langlands correspondence.

\begin{remark} \label{thm69}
Besides the difference in enhancement,
which we have already discussed,
Kaletha's formulation of the local Langlands conjecture
differs from that of \cite{AMS18} in that the former restricts to tempered $L$-parameters
while the latter is stated for any $L$-parameter, tempered or not.
Despite appearances there is no space between these formulations:
one can use the results of \cite{silberger_zink18},
in particular, isomorphism~(6) on page~339,
to show that any enhanced local Langlands correspondence
for tempered $L$-parameters gives rise
to such a correspondence for general $L$-parameters.
\end{remark}

Given a compound local Langlands correspondence for~$G$
and a rigid inner twist $(G',\xi, \mathcal{T},\bar h)$ of~$G$
one can construct an ordinary local Langlands correspondence
\[
\rec_{\frak w,(G',\xi, \mathcal{T},\bar h)}\colon
\Pi(G') \simeq \Phi_\tn{e}^Z(G', \mathcal{T})
\]
by taking the fiber over~$\zeta_\mathcal{T}$, the character of $\pi_0(Z(\widehat G)^{\Gamma,+})$ corresponding to $[\mathcal{T}]$.
Further restriction to a single
$(G',\mathcal{T})$-relevant $L$-parameter~$\varphi$
yields a parameterization
\[
\rec_{\frak w,(G',\xi, \mathcal{T},\bar h),\varphi}\colon
\Pi_\varphi(G') \simeq
\Irr\bigl(\pi_0\bigl(Z_{\widehat G}(\varphi)^+\bigr),\zeta_\mathcal{T} \bigr)
\]
of the corresponding $L$-packet.
Conversely, giving such a compound correspondence
amounts to giving the collection
of sets $\Pi^{\rig,Z}_\varphi\bigl(G',\xi, \mathcal{T},\bar h\bigr)$ and maps
$\rec_{\frak w,(G',\xi, \mathcal{T},\bar h),\varphi}$.
In light of this reduction, we mostly focus
on a fixed rigid inner twist
$(G',\xi,\cal T,\bar h)$ of~$G$ for the rest of the section.

Now we turn to Bernstein blocks and the cuspidal support map.
Roughly speaking, the following conjectures
assert that the constructions of \Cref{sec:bernstein}
are compatible with their analogues
in the representation theory of $p$-adic groups.
As part of these, we give a conjecture predicting which Bernstein block
contains a given representation $\dot\pi$.

The foundational conjecture
is that cuspidal enhanced $L$-parameters correspond to supercuspidal representations.

\begin{conjecture} \label{compatiblesc}
An irreducible representation $\dot\pi$ is supercuspidal
if and only if $\rec_{\frak w}(\dot\pi)$ is cuspidal.
\end{conjecture}

Comparing Bernstein blocks requires two additional compatibilities,
twisting by unramified representations
and equivariance with respect to outer automorphisms.

\begin{conjecture} \label{compatibleunram}
Let $\chi\in X^*_\tn{nr}(G)$ correspond to $z\in X^*_\tn{nr}({}^LG)$
under \eqref{thm54}.
Then
\[
\rec_{\frak w}(\chi\otimes\dot\pi)
= z^{\val_F}\cdot\rec_{\frak w}(\dot\pi).
\]
\end{conjecture}

Let $\Out_F(G)$ denote the group
of ($F$-rational) outer automorphisms of~$G$.

\begin{conjecture} \label{compatibleouter}
Let $\theta\in\Out_F(G)$.  Then
$\rec_{\theta(\frak w)}\bigl(\theta(\dot\pi)\bigr)
= \theta\bigl(\rec_{\frak w}(\dot\pi)\bigr)$.
\end{conjecture}

We refer the reader to the \cite[Conjecture 2.12]{kaletha22}
and the discussion surrounding it for a fuller
explanation of this conjecture.
There are some subtleties here:
outer automorphisms of~$G$ can mix
the equivalence classes of its rigid inner twists,
meaning that $\theta(\dot\pi)$ and $\dot\pi$
might be representations of different groups.

All together, the following properties ensure compatibility
with the constructions of \Cref{sec:bernstein}.

\begin{conjecture} \label{thm55}
Every Levi subgroup of~$G$ satisfies \Cref{compatiblesc}
and satisfies \Cref{compatibleunram,compatibleouter}
for $\dot\pi$ supercuspidal.
\end{conjecture}

\begin{lemma} \label{thm56}
Suppose $G$ satisfies \Cref{thm55}.
Then the maps $(\rec_{\frak w_M})_{M\in\Lev(G)}$
induce a commutative diagram
\[
\begin{tikzcd}
\Pi^{\rig,Z}(G) \rar{\Sc}\dar{\rec_{\frak w}} &
\Omega_{\rig,Z}(G) \rar\dar &
\frak B^{\rig,Z}(G) \rar\dar &
H^1(\cal E,Z\to G) \dar \\
\Phi_\tn{e}^Z(G) \rar{\Cusp} &
\Omega_Z({}^LG) \rar &
\rar \frak B({}^LG) \rar &
\pi_0\bigl(Z(\widehat G)^{\Gamma,+}\bigr)^*
\end{tikzcd}
\]
in which the vertical maps are isomorphisms
and the second of these is a map of varieties.
\end{lemma}

\begin{proof}[Proof sketch]
Let $(M',\xi_M,\cal T_M,\bar h_M)$ be a Levi subgroup
of $(G',\xi,\cal T,\bar h)$ corresponding to
a Levi subgroup~$M$ of~$G$
and let 
\[
\dot s \defeq (M',\xi_M,\cal T_M,\bar h_M,\dot\sigma)
\]
be a cuspidal datum for~$(G',\xi,\cal T,\bar h)$.
Writing $(\psi,\varrho) \defeq \rec_{\frak w_M}(\dot\sigma)$,
we have a dual cuspidal datum
\[
\hat s \defeq ({}^LM,\psi,\varrho).
\]
It follows from \Cref{compatibleouter}
and the well-known isomorphism
\[
W(G',M') \simeq W({}^LG,{}^LM)
\]
that the $\widehat G$-conjugacy class
of $\hat s$ depends only on
the $\widehat G$-conjugacy class of~$\dot s$.
The assignment $[s]_G\mapsto [\dot s]_{\widehat G}$
defines a bijection
\begin{equation} \label{thm59}
\Omega^{\dot{\frak s}}_{\rig,Z}(G) \simeq
\Omega^{\hat{\frak s}}({}^LG)
\end{equation}
and these assemble to the desired map
$\Omega_{\rig,Z}(G)\simeq\Omega_Z({}^LG)$.
By \Cref{compatibleunram}, the map \eqref{thm59}
with $G$ replaced by~$M$ is an isomorphism of varieties,
as the variety structure in this case
is defined using unramified twisting.
Using \Cref{compatibleouter} 
and the quotient description \eqref{thm29}
of $\Omega^{\dot{\frak s}}_{\rig,Z}(G)$,
we find that \eqref{thm59},
and thus the composite map on Bernstein varieties,
is an isomorphism of varieties.
\end{proof}

\begin{remark} \label{thm65}
\Cref{thm56} gives a recipe for computing
the supercuspidal support map arithmetically
using the local Langlands correspondence.
Namely, given $\pi\in\Pi(G')$,
choose a rigid enhancement $(G',\xi,\cal T,\bar h)$ of~$G'$
and a Whittaker datum~$\frak w$ for~$G$.
Letting $[\cal M,\phi,v,\varrho]$ be
the cuspidal support of $\rec_{\frak w}(\dot\pi)$,
choose a Levi subgroup $M$ of~$G$
with ${}^LM$ conjugate to~$\cal M$
and let $\dot\sigma\defeq\rec_{\frak w_M}^{-1}(\phi,v,\varrho)$.
Then $[M,\sigma]_G$ ought to be
the supercuspidal support of~$\pi$.

On the one hand, this procedure appears to depend
on the choices of rigidification and Whittaker datum.
On the other hand, the supercuspidal support map
is independent of such choices.
Comparing these two points of view predicts
compatibility conditions on the cuspidal support map
for $L$-parameters with respect to change of rigidification
and change of Whittaker datum.
We will formulate and verify these conditions
in \Cref{sec:corr:rigid,sec:corr:whittaker}.
\end{remark}

\Cref{thm56} suggests a strategy for reducing
local Langlands correspondences to the supercuspidal case,
an idea that goes back to \cite{aubert_baum_plymen_solleveld17a}.
Supposing \Cref{thm55} to be satisfied---the reduction---
one would hope to use the isomorphism
$\Omega_{\rig,Z}(G)\simeq\Omega_Z({}^LG)$
to construct the reciprocity map $\rec_{\frak w}$.
For this, it would suffice to match the fibers
of the supercuspidal and cuspidal support maps,
and we are free to work one Bernstein block at a time.

On the Galois side, everything is understood:
\Cref{AMS9.3} describes each Bernstein block
as a twisted extended quotient.
On the automorphic side, one would expect that
each Bernstein block again admits a natural
structure of a twisted extended quotient.
This expectation is known as the ABPS conjecture
\cite[Conjecture~2]{aubert_baum_plymen_solleveld17a},
and was recently proved in full generality
by Solleveld \cite[Theorem~E]{solleveld22}:
he constructs, for each Levi~$M'$ of~$G'$
and inertial equivalence class $\frak s_{M'}$ of~$M'$,
a family of cocycles~$\natural$, gluing maps,
and a bijection
\[
\Pi^{\frak s_{G'}}(G') \simeq
\bigl(\Pi^{\frak s_{M'}}(M')\sslash W(G',M')_{\frak s_{M'}}\bigr)_\natural.
\]
One expects the local Langlands correspondence
to identify these two twisted extended quotients.

\begin{conjecture} \label{thm60}
Let $(M',\xi_M,\cal T_M,\bar h_M)$
be a Levi subgroup of $(G',\xi,\cal T,\bar h)$,
let $s = (M',\sigma)$ be a cuspidal datum for~$G'$,
and let $\hat s = ({}^LM,\psi,v,\varrho)$
be a cuspidal datum for~${}^LG$ such that
$(\psi,v,\varrho)$ is relevant for $(M',\xi_M,\cal T_M,\bar h_M)$.
Suppose $G$ satisfies \Cref{thm55}
and $\dot{\frak s}_{G'}$ corresponds to
$\frak s_{\widehat G}$ as in \Cref{thm56}.
Then there is an isomorphism of twisted extended quotients
fitting into a commutative diagram
\[
\begin{tikzcd}
\bigl(\Omega_{\rig,Z}^{\dot{\frak s}_M}(M)\sslash
W(G',M')_{\frak s_{M'}}\bigr)_\natural \rar{\sim}\dar &
\bigl(\Omega^{\hat{\frak s}_M}({}^LM)\sslash 
W({}^LG,{}^LM)_{\hat{\frak s}_M}\bigr)_\kappa \dar \\
\Omega_{\rig,Z}^{\dot{\frak s}_G}(G) \rar &
\Omega_Z^{\hat{\frak s}_G}({}^LG)
\end{tikzcd}
\]
in which the vertical maps are projections
to the naive quotients.
\end{conjecture}

Here $\dot{\frak s}_{G'}$ is the enhancement
of $\frak s_{G'}$ in which $M'$ is replaced
by the rigid inner twist $(M',\xi_M,\cal T_M,\bar h_M)$.

A direct proof of \Cref{thm60} would require an analysis
of the families of cocycles $\natural$ and~$\kappa$,
which we expect to be quite subtle in general.
However, when the cocycle families are trivial,
which is true for many Bernstein blocks
(cf.~\cite[Remark 3.1.44]{AX22}),
\Cref{thm60} follows already from \Cref{thm55}.

\begin{remark} \label{thm61}
Putting everything together,
we see that a proof of the general local Langlands correspondence
follows from \Cref{thm55} and the triviality
of the cocycle families $\natural$ and~$\kappa$.
This strategy can also be used to prove partial
local Langlands correspondences:
if one has a local Langlands correspondence $\rec_{\frak w_M}$ 
for a Levi~$M'$ defined only on a single Bernstein block $\frak s_{M'}$
then \Cref{compatibleunram,compatibleouter}
together with the triviality of $\natural$ and~$\kappa$
yields a local Langlands correspondence
for the corresponding Bernstein block $\frak s_{G'}$ of~$G'$.
\end{remark}

\subsection{Change of rigidification}
\label{sec:corr:rigid}
Although our formulation of the local Langlands correspondence uses compound $L$-packets,
and hence requires rigid inner forms on the automorphic side,
the supercuspidal support map has no such requirement.
This observation points to a compatibility between
the cuspidal support map for $L$-parameters and
change of rigidification as we now explain,
elaborating on \Cref{thm65}.

We need an analogue of Tate-Nakayama duality
for finite multiplicative group schemes on the gerbe $\mathcal{E}$.

\begin{proposition}\label{Zduality} There is a canonical, functorial isomorphism 
$$H^{1}(\mathcal{E},Z) \xrightarrow{\sim} Z^{1}(\Gamma, \widehat{Z})^{*},$$
where $\widehat{Z} \defeq \mathrm{Hom}(X^{*}(Z), \mathbb{Q}/\mathbb{Z}) \cong \mathrm{Hom}(X^{*}(Z), \mathbb{C}^\times).$
\end{proposition}

\begin{proof} In characteristic zero this is proved in \cite[\S 6.1]{kaletha18a}) and \cite[Proposition 5.1]{Dillery23} proves it for arbitrary nonarchimedean local fields.
\end{proof}

Let $z_{1}, z_{2}$ denote two rigidifications of
the same inner twist $(G', \xi)$ of $G$,
and set 
\[
\Irr\bigl(\pi_{0}\bigl(Z_{\widehat G}(\varphi)^+\bigr), z_{i}\bigr)
\defeq \tn{rec}_{\varphi, \mathfrak{w}, z_{i}}(\Pi_{\varphi}(G')).
\]

Since $z_{1}, z_{2}$ rigidify the same inner twist, there is some
$y \in Z^{1}(\mathcal{E}, Z)$ with $z_{2} = y \cdot z_{1}$.
We summarize the arguments in \cite[\S 6.3]{kaletha18a}
which explain how to extract from $y$ a character of
$\pi_{0}\bigl(Z_{\widehat G}(\varphi)^+\bigr)$, using \Cref{Zduality}. 
The differential (with respect to $W_F \times \SL_{2}$-cohomology) induces a map 
\[
d \colon \pi_{0}\bigl(Z_{\widehat G}(\varphi)^+\bigr) \to Z^{1}(\Gamma, \widehat Z),
\]
and so we obtain a character of $\pi_{0}\bigl(Z_{\widehat G}(\varphi)^+\bigr)$ by taking the composition $[y]^{*} \circ -d$, where we are denoting the dual of $[y] \in H^{1}(\mathcal{E}, Z)$ as in Proposition \ref{Zduality} by $[y]^{*}$. We then define a bijection 
\[
\Irr(\pi_{0}\bigl(Z_{\widehat G}(\varphi)^+\bigr), z_{1}) \to \Irr\bigl(\pi_{0}\bigl(Z_{\widehat G}(\varphi)^+\bigr), z_{2}\bigr)
\]
by twisting by this character. One checks that this is a well-defined map between the above two sets (and then it is evidently a bijection). It is shown loc.\ cit.\ that if restricting $\tn{rec}_{\varphi, \mathfrak{w}, z_{1}}$ gives the local Langlands correspondence for $\Pi_{\varphi}(G')$, then the composition of $\tn{rec}_{\varphi, \mathfrak{w}, z_{1}}$ with our above bijection gives the local Langlands correspondence for $\Pi_{\varphi}(G')$ as well, and is equal to $\tn{rec}_{\varphi, \mathfrak{w}, z_{2}}$.

Using this notation, we can formulate the following compatibility
of cuspidal support and change of rigidification.

\begin{proposition} \label{thm66}
The following diagram commutes:
\begin{center}
\begin{tikzcd}[column sep=huge]
\Phi_{\tn e}^Z(G) \rar{\chi_{y,-}} \dar{\Cusp} &
\Phi_{\tn e}^Z(G) \dar{\Cusp} \\
\bigsqcup\limits_{M\in\Lev(G)} \Phi_\tn{e,cusp}^Z(M)
\rar{\coprod_{M}\chi_{y,-}} &
\bigsqcup\limits_{M\in\Lev(G)} \Phi_\tn{e,cusp}^Z(M).
\end{tikzcd}
\end{center}
\end{proposition}

Here $\chi_{y,-}$ acts on enhancements of $\varphi$
(an $L$-parameter for ${}^LG$ or ${}^LM$) via the character
$\chi_{y,\varphi}$ of $\pi_{0}\bigl(Z_{\widehat G}(\varphi)^+\big)$
attached to $y \in H^{1}(\mathcal{E}, Z)$ as explained above.

Consequently, the inertial class of the cuspidal pair $\frak s$
associated to a representation $\pi \in \Pi_{\varphi}(G')$
via \Cref{thm56} (as in \Cref{thm65}) does not depend on the choice of rigidification.

\begin{proof}  Let $\dot{\pi}_{i} = (z_{i}, \pi)$ be two rigidifications. By the above discussion, $\tn{rec}_{\mathfrak{w}}(\dot{\pi}_{2}) = \chi_{y}\cdot\tn{rec}_{\mathfrak{w}}(\dot{\pi}_{1})$. We need to compare $\Cusp(\varphi, \chi_{y} \tn{rec}_{\mathfrak{w}}(\dot{\pi}_{1}))$ to $\Cusp(\varphi, \tn{rec}_{\mathfrak{w}}(\dot{\pi}_{1}))$.

First, the group $H := Z_{\widehat{G}}(\varphi |_{W_F})$ is the same in both cases, as is the unipotent element in the $L$-parameter. Although $\chi_{y}$ is a character of $\pi_{0}\bigl(Z_{\widehat G}(\varphi)^+\bigr)$ (or $Z_{\widehat G}(\varphi)^+$) rather than $H^{+}$, we can define a similar character $\chi_{y,H^{+}}$ via the composition $[y]^{*} \circ -d$ applied to $H^{+}$ rather than $Z_{\widehat G}(\varphi)^+$. By construction, applying the differential $d$ to $H^{+}$ still lands in $Z^{1}(\Gamma, \widehat{Z})$, and since $\widehat{Z}$ is finite, this composition will factor through $\pi_{0}(H^{+})$, as in the $Z_{\widehat G}(\varphi)^+$-case. Moreover, the restriction of $\chi_{z,H^{+}}$ to $\pi_{0}(Z_{H^{+}}(u_{\varphi}))$ agrees with $\chi_{y,\varphi}$ via the identification of Corollary \ref{cuspidalparameterlemma}.

By \Cref{thm64}, if $[\mathcal{M}, v, \varrho]_{\mathcal{M}^{\circ}}$ is the quasi-cuspidal support associated to $(\varphi, \tn{rec}_{\mathfrak{w}}(\dot{\pi}_{1}))$ then the quasi-cuspidal support associated to $(\varphi, \chi_{y,\varphi} \tn{rec}_{\mathfrak{w}}(\dot{\pi}_{1})) $ is $$[\mathcal{M}, v, (\chi_{y,H^{+}}|_{\mathcal{M}}) \otimes \varrho]_{\mathcal{M}^{\circ}}.$$
Thus, by our construction of the map $\Cusp$, we find that if $\Cusp(\varphi, \tn{rec}_{\mathfrak{w}}(\dot{\pi}_{1})) = [{}^LM, \psi, \varrho]_{\widehat{G}}$, then 
\[
\Cusp(\varphi, \chi_{y,\varphi}\cdot\tn{rec}_{\mathfrak{w}}(\dot{\pi}_{1})) = [{}^LM, \psi, \chi_{y,\psi} \otimes \varrho]_{\widehat{G}}.
\]
Denote by $x_{1}, x_{2}$ the two transfers of $z_{1}, z_{2}$ to rigid inner twists of $M$. We have the identity $$\tn{rec}_{\mathfrak{w}_{M},x_{2}}^{-1} (\chi_{y,\psi} \otimes \varrho) = \tn{rec}_{\mathfrak{w}_{M},x_{1}}^{-1} (\varrho),$$ from which the result follows. \end{proof}

\subsection{Change of Whittaker datum}
\label{sec:corr:whittaker}
Although the parameterization $\tn{rec}_{\mathfrak{w}}$ of
the compound $L$-packets depends on the choice of Whittaker datum $\mathfrak{w}$,
the supercuspidal map does not.
As we mentioned in \Cref{thm65},
this observation leads to a compatibility condition
between cuspidal support and change of Whittaker datum,
which we explain in this section.

Fix a nontrivial character $\psi_0\colon F\to\bbC^\times$.
There is a natural construction of a Whittaker datum
from a quasi-pinning of~$G$ depending only on the choice of~$\psi_0$.
Given a quasi-pinning $p=(B,T,\{X_\alpha\}_{\alpha\in\Delta})$,
let $U$ be the unipotent radical of~$B$
and $U_+$ the smooth subgroup of~$U$
whose Lie algebra is spanned by
the root spaces for the roots of $\Phi^+\setminus\Delta$.
The inclusions $U_\alpha\to U$ of simple root groups
induce an isomorphism $U/U_+\simeq\prod_{\alpha\in\Delta} U_\alpha$.
The elements~$X_\alpha$ in turn yield
an isomorphism $U/U_+\simeq\prod_{\alpha\in\Delta/\Gamma}
\tn{Res}_{F_\alpha/F}\bbG_{\tn a}$.
Define the algebraic character $\psi_p\colon U/U_+\to\bbG_{\tn a}$
as the composition of this isomorphism with the componentwise trace-summation map
$(t_\alpha)_{\alpha\in\Delta/\Gamma}
\mapsto\sum_{\alpha\in\Delta/\Gamma}
\tn{Tr}_{F_\alpha/F}(t_\alpha)$.
On $F$-points, $\psi_p$ induces a map $U(F)/U_+(F)\to F$,
also denoted by~$\psi_p$.
From the quasi-pinning~$p$ we form the Whittaker datum $(B,\psi_0\circ\psi_p)$.

Every Whittaker datum is constructed from a quasi-pinning in this way.
Although we cannot recover the quasi-pinning from the Whittaker datum,
any two quasi-pinnings that yield the same Whittaker datum are $G(F)$-conjugate.
Moreover, the set of $G(F)$-conjugacy classes of quasi-pinnings of~$G$
is a $G_\tn{ad}(F)/G(F)$-torsor \cite[XXIV.3.10]{sga3-3}.
Since the assignment $p\mapsto(B,\psi_0\circ\psi_p)$
is $\Out(G)(F)$-equivariant,
the action of $G_\tn{ad}(F)$ on the set of Whittaker data
makes the set of $G(F)$-conjugacy classes of Whittaker data
a torsor for~$G_\tn{ad}(F)/G(F)$.

The bijection $p\mapsto(B,\psi_0\circ\psi_p)$
on $G(F)$-conjugacy classes depends on~$\psi_0$
in the following way.
Any other nontrivial character of~$F$
is of the form $a\psi_0$ for some $a\in F^\times$.
Then
\[
(a\psi_0)\circ\psi_p = \psi_0\circ\psi_{\rho(a)p}
\]
where $\rho\in X_*(T_\tn{ad})$ is the half-sum of positive coroots
(all taken with respect to~$p$).
But since $G_\tn{ad}(F)/G(F)$ is abelian,
the bijection on $G(F)$-conjugacy classes
is ultimately independent of the choice of~$\psi_0$.

We can restrict Whittaker data to Levi subgroups $M$ of~$G$ as follows.
Using the bijection above, it is enough to explain
how to restrict $G(F)$-conjugacy classes of quasi-pinnings of~$G$
to $M(F)$-conjugacy classes of quasi-pinnings of~$M(F)$.
Fix a minimal Levi subgroup~$T$ (a torus) of~$M$
and Borel subgroups $B_M$ of~$M$ and~$B$ of~$G$ containing~$T$.
Every quasi-pinning of~$G$ is rationally conjugate to one that begins with $(B,T)$,
and such a quasi-pinning $p=(B,T,\{X_\alpha\}_{\alpha\in\Delta})$ restricts
to the quasi-pinning $p|_M=(B_M,T,\{X_\alpha\}_{\alpha\in\Delta_M}$ of~$M$.
The $M(F)$-conjugacy class of the restriction is independent
of the representative of~$p$ because any two such quasi-pinnings of~$G$
(namely, those that begin with $(B,T)$) are $T(F)$-conjugate.
Returning to Whittaker data yields a restriction map
\[
\frak w \mapsto \frak w_M.
\]
The notation is slightly abusive because
the map is defined only on conjugacy classes
of Whittaker data, not on the Whittaker data themselves.

Let $(\phi,u)$ be an enhanced $L$-parameter.
When $F$ has characteristic zero,
\cite[\S4]{kaletha13} constructs a map
\begin{equation} \label{thm10}
\pi_0\bigl(Z_{\widehat G}(\phi)/Z(\widehat G)^\Gamma\bigr)
\to \bigl(G_\tn{ad}(F)/G(F)\bigr)^*
\end{equation}
which is natural with respect to passage to Levi subgroups.
In fact, such a map exists in positive characteristic as well:
one can use the same crossed-module arguments
from \cite[\S 5.5]{kaletha12} together with
\cite[Appendix A]{Dillery2} for the key case of tori.
Given two Whittaker data $\frak w$ and $\frak w'$,
let $(\frak w,\frak w')$ denote the pullback
of the corresponding character of $G_\tn{ad}(F)/G(F)$
along \eqref{thm10} and the natural projection
to the group $\pi_0\bigl(Z_{\widehat G}(\phi)^+\bigr)$,
with the evident dependence on~$\phi$
suppressed from the notation.

\begin{proposition} \label{thm95}
Let $\frak w$ and $\frak w'$ be Whittaker data
and let $(\varphi,\rho)$ be an enhanced $L$-parameter.
\begin{enumerate}
\item
$(\varphi,\rho)$ is cuspidal if and only if
$(\varphi,(\frak w,\frak w')\cdot\rho)$ is cuspidal.

\item
Change of Whittaker datum commutes with cuspidal support
in the sense that the following diagram commutes.

\begin{center}
\begin{tikzcd}[column sep=huge]
\Phi_{\tn e}^Z(G) \rar{(\frak w,\frak w')} \dar{\Cusp} &
\Phi_{\tn e}^Z(G) \dar{\Cusp} \\
\bigsqcup\limits_{M\in\Lev(G)} \Phi_\tn{e,cusp}^Z(M)
\rar{\coprod_M(\frak w_M,\frak w'_M)} &
\bigsqcup\limits_{M\in\Lev(G)} \Phi_\tn{e,cusp}^Z(M)
\end{tikzcd}
\end{center}
\end{enumerate}
\end{proposition}

The diagram also commutes
if we use $\uCusp$ in place of $\Cusp$.

Consequently, the (inertial class of the) cuspidal pair $\frak s$
associated to a representation $\pi \in \Pi_{\varphi}(G')$
via \Cref{thm56} (as in \Cref{thm65})
does not depend on the choice of Whittaker datum.

\begin{proof}
Let $H_\phi^+ = Z_{\widehat G}(\phi)$,
so that $(\frak w,\frak w')\colon\pi_0(H_\phi^+)\to\bbC^\times$.
The first part follows from the fact that
in this disconnected setting cuspidality, by definition,
depends only on the restriction of~$\rho$
to $\pi_0\bigl(Z_{H_\phi^{+\circ}}(u)\bigr)$
and twisting by $(\frak w,\frak w')$ leaves this restriction unchanged.
The second part follows from the naturality of \eqref{thm10} and Proposition \ref{thm64}.
\end{proof}

\subsection{Non-singular enhanced $L$-parameters}
Following the recent work of Aubert and Aubert--Xu
(\cite{AX22, aubert23}) we extend the definition of non-singular $L$-parameters to the rigid setting and connect it to our framework.

First, we briefly recall the notion of a \mathdef{nonsingular representation} of a connected reductive $F$-group $G'$. Such representations arise from the following arithmetic data.

\begin{definition}
A \mathdef{tame $F$-non-singular pair} $(S, \vartheta)$ consists of $S$ an elliptic maximal torus of $G'$ split over a tamely ramified extension $E/F$ along with a character $\vartheta \colon S \to C^\times$ such that:
\begin{enumerate}
\item
$p$ is a very good prime for~$G'$.

\item
Inside the unique reductive $F$-group~$G'^0$ containing $S$ and with root system
\[
R_{0+} = \bigl\{\alpha \in \Phi(G'_{F^{\text{sep}}}, S_{F^{\text{sep}}})
\mid \bigl(\vartheta\circ N_{E/F}\circ \alpha^{\vee}\bigr)(E_{0+}^{\times} = 1\bigr\}
\]
the torus $S$ is maximally unramified.

\item
The character $\vartheta$ is \mathdef{$F$-non-singular} with respect to $G'^{0}$.
\end{enumerate}
\end{definition}

Kaletha showed \cite{kaletha19} how to rewrite a certain case of
Yu's construction of supercuspidal representations \cite{yu01}
in terms of tame $F$-non-singular pairs,
producing from such a pair $(S,\vartheta)$
a supercuspidal representation~$\pi_{S,\vartheta}$.
A representation of $G'(F)$ is called \mathdef{non-singular}
if it arises by this construction.
He isolated a certain class of
\mathdef{torally wild} (supercuspidal) $L$-parameters.
Such $L$-parameters $(\phi,u)$ are always \mathdef{supercuspidal},
meaning (by definition) that $u=1$.%
\footnote{In what follows we suppress the trivial unipotent element
from our notation for supercuspidal $L$-parameters.}
Conversely, when $p$ does not divide the order
of the absolute Weyl group of~$G'$,
every supercuspidal $L$-parameter is torally wild.
Working backwards from the parameterization of non-singular
supercuspidals~$\pi$ by tame $F$-non-singular pairs,
Kaletha attached to each nonsingular
supercuspidal $\dot\pi$ of a rigid inner twist of~$G$
an enhanced $L$-parameter $(\phi_{\dot\pi},\rho_{\dot\pi})$
with $\phi_{\dot\pi}$ torally wild.

So all in all, Kaletha's work constructs
an $L$-packet for every supercuspidal $L$-parameter
when $p$ is not too small,
fleshing out a large portion of the local Langlands correspondence
for supercuspidal representations.

It is natural to extend these developments
to nonsupercuspidal representations via the Bernstein decomposition.
Since non-singularity is preserved by unramified twisting,
we can define a Bernstein block to be \mathdef{non-singular}
if its underlying class of supercuspidal representations is non-singular.
On the level of representations, one has the following
definitions of Aubert \cite{aubert23}.

\begin{definition}
An irreducible smooth representation of $G'$ is \mathdef{non-singular}
if its supercuspidal support is non-singular.
\end{definition}

Similarly, on the Galois side, Aubert isolated
the class of enhanced $L$-parameters that ought
to correspond to non-singular representations.
Her definition makes sense more generally 
in the setting of \Cref{sec:springer}
where $H^+$ is a (possibly disconnected) reductive group over~$\bbC$.

\begin{definition}
A quasi-cuspidal support $[K^+, v, \varrho]_{K^+}$ is \mathdef{semisimple} if $v = 1$. 
\end{definition}

Aubert shows in \cite[Lemma 4.2.4]{aubert23} that the semisimple quasi-cuspidal supports are exactly those where $K^+ = Z_{H^+}(\mathcal{T})$ for a maximal torus $T \subset K^{+\circ}$ and
$\varrho \in \Irr\bigl(\pi_0(K^+)\bigr)$. 
Now specialize this definition to $L$-parameters, taking $H^+ = Z_{\widehat{G}^{+}}(\phi)$.

\begin{definition}
A cuspidal datum $(\cal M, \psi, v, \varrho)$ is \mathdef{semisimple}
if the underlying quasi-cuspidal support is semisimple. 
An enhanced $L$-parameter $(\phi,u,\rho)$ of ${}^LG$ is \mathdef{non-singular}
if its cuspidal support is semisimple.
\end{definition}

In other words, an enhanced $L$-parameter is non-singular
if and only if (the $L$-parameter in) its cuspidal support is supercuspidal.
Unlike in the definition of non-singular representations,
there is no technical obstruction to formulating
the definition of non-singular enhanced $L$-parameter for general~$G$ and~$p$,
and we expect this notion to be quite useful for the general theory.

Aubert and Xu have made partial progress towards
proving a local Langlands correspondence for non-singular representations
via the strategy of \Cref{thm61}.
The first step is to extract from Kaletha's construction
an enhanced $L$-parameter $[\psi_{\sigma}, \rho^\tn{sc}_{\sigma}]_{\widehat{M}}$
in the sense of \cite{AMS18}.
Proposition~3.1.51 of \cite{AX22} shows that the construction
$\sigma \mapsto [\psi_{\sigma}, \rho^{\tn{sc}}_{\sigma}]_{\widehat{M}}$
is compatible with respect to unramified twists (\Cref{compatibleunram}).
When in addition \Cref{compatibleouter} holds
and the families of cocyles $\natural$ and $\kappa$ are trivial
for a particular non-singular Bernstein block
$\Pi^{\frak s_{G'}}(G')$ in question,
we say $\frak s_{G'}$ is \mathdef{good}.
In this case, we have a correspondence.

\begin{theorem}[{\cite[Theorem 3.1.32]{AX22}}]
Let $s$ be a good nonsingular cuspidal datum for~$G$.
The assignment $\sigma \mapsto [\psi_{\sigma}, \rho^\tn{sc}_{\sigma}]_{\widehat{G}}$
induces horizontal bijections fitting into the commutative diagram
\[
\begin{tikzcd}
\Pi^{\frak s_{G'}}(G') \rar[leftrightarrow] \arrow{d} &
\Phi_\tn{AMS}^Z({}^LG)^{\hat{\frak s}_G} \arrow{d} \\
(\Pi^{\frak s_{M'}}(M')\sslash W_{{\frak s}_M})_\natural \rar[leftrightarrow] &
(\Phi_\tn{AMS}^Z({}^LM)^{\hat{\frak s}_M} \sslash W_{\hat{\frak s}_M})_\kappa
\end{tikzcd}
\]
\end{theorem} 

Here we write ``AMS'' in subscript to mean that
we use the notion of enhanced $L$-parameters from \cite{AMS18}
This result is also true for rigid enhancements.

\begin{proposition}
Let $s$ be a good nonsingular cuspidal datum for~$G$.
The assignment $\sigma \mapsto [\psi_{\sigma}, \rho_{\sigma}]_{\widehat{G}}$
induces horizontal bijections fitting into the commutative diagram
\[
\begin{tikzcd}
\Pi^{\frak s_{G'}}(G') \rar[leftrightarrow] \arrow{d} &
\Phi_\tn{e}^Z({}^LG)^{\hat{\frak s}_{G'}} \arrow{d} \\
(\Pi^{\frak s_{M'}}(M')\sslash W_{{\frak s}_M})_\natural \rar[leftrightarrow] &
(\Phi_\tn{e}^Z({}^LM)^{\hat{\frak s}_M} \sslash W_{\hat{\frak s}_M})_\kappa
\end{tikzcd}
\]
\end{proposition}

\begin{proof}
The identical proof as in \cite{AX22} works here.
The key part of their argument is showing that the map
$\sigma \mapsto [\psi_{\sigma}, \rho_{\sigma}]$ induces a bijection on the twisted extended quotients
\[
(
\Omega_{\rig,Z}^{\dot{\frak s}_M}(M) \sslash W_{{\frak s}_M})_{\kappa}
\to (\Phi_{\tn e}^Z({}^LM)^{\dot{\hat{\frak s}}_M} \sslash W_{\hat{\frak s}_M})_{\kappa}.
\]
From here, the desired bijection is defined as the composition
of the other three maps (or their inverses) in the above diagram.
\end{proof}

\appendix
\section{Tables of nilpotent orbits
and cuspidal perverse sheaves}
\label{sec:tables}
In spite of the elaborate trappings of
the generalized Springer correspondence,
nilpotent orbits are classified combinatorially
and cuspidal perverse sheaves are rare enough
that both objects can be explicitly tabulated.
For the edification of the reader we collect the tabulations here,
taking the descriptions of nilpotent orbits from
\cite[5.1, 6.1, 8.4]{collingwood_mcgovern93}
and the classification of cuspidal perverse sheaves
from \cite[\S 10--15]{lusztig84b}
(and, for the Spin group, \cite[4.9]{lusztig_spaltenstein85}).
In each of the following sections,
let $G$ be a simply-connected group
of the given type.
(\cite[10.1]{lusztig84b} explains how to reduce to this case.)
Given a nilpotent orbit~$\cal O$, let
\[
A(\cal O) \defeq \pi_0\bigl(Z_G(u)\bigr)
\]
where $u\in\cal O$.
Although this group depends on~$u$,
its (inner) isomorphism class does not.
We will explicitly describe the cuspidal support map
for the exceptional groups,
but for the four infinite families
we refer the read to the combinatorial
description of Lusztig's original article
\cite[\S 11--14]{lusztig84b}.
Abbreviate ``equivariant perverse sheaf on the nilpotent cone''
to ``perverse sheaf''.

\subsection{Classical groups}
Recall that a \mathdef{partition}
of an integer~$n$ is a nonincreasing finite sequence
\[
\lambda = (\lambda_1 \leq \lambda_2 \leq \cdots \leq \lambda_k)
\]
of integers such that $\sum_{i=1}^k \lambda_i = n$.
The integers $\lambda_i$ are called the \mathdef{parts} of $\lambda$,
and for $d$ a part of~$\lambda$,
the number of $i$ such that $\lambda_i=d$ is
called the \mathdef{multiplicity} of~$d$ in~$\lambda$.
Nilpotent orbits of classical groups
are parameterized in terms of certain partitions.
The bijection is quite explicit
and ultimately passes through the Jordan matrices
\[
\begin{bmatrix}
0 & 1 \\
  & 0 & 1 \\
  &   & \ddots & \ddots  \\
  &   &        & 0  \\
\end{bmatrix}
\]
of elementary linear algebra.
Given a partition~$\lambda$,
we will write $\cal O_\lambda$ for the corresponding
nilpotent orbit.
There are two caveats to this notation.
\begin{enumerate}
\item
In types $B$, $C$, and~$D$,
nilpotent orbits are not parameterized by arbitrary partitions,
but by partitions where every even part (in types $B$ and~$D$)
or odd part (in type~$C$) has even multiplicity.

\item
In type~$D$, some partitions~$\lambda$ correspond to
two distinct nilpotent orbits.
We will denote them by $\cal O_\lambda^\tn{I}$
and $\cal O_\lambda^\tn{II}$.
\end{enumerate}
The group $A(\cal O_\lambda)$
has a simple description in terms of~$\lambda$.
For the description of cuspidality,
we recall that a number is \mathdef{triangular}
if it is of the form $1 + 2 + \cdots + d = d(d+1)/2$.

\subsubsection{Type \texorpdfstring{$A_n$}{Aₙ}}
Nilpotent orbits are in bijection with
partitions~$\lambda$ of~$n+1$ and
\[
A(\cal O_\lambda) \simeq \bbZ/c_\lambda\bbZ
\]
where $c_\lambda$ is the gcd of $(\lambda_1,\dots,\lambda_k)$.
The only nilpotent orbit that supports
a cuspidal perverse sheaf is the regular orbit $\cal O_{n+1}$,
and a character of $A(\cal O_{n+1})$ is cuspidal
if and only if it has order~$n+1$.

\subsubsection{Type \texorpdfstring{$B_n$}{Bₙ}}
Nilpotent orbits are in bijection with partitions of~$2n+1$
for which every even part has even multiplicity.
To describe $A(\cal O_\lambda)$,
let $a_\lambda$ be the number of distinct odd parts of~$\lambda$.
If some odd part of $\lambda$ has multiplicity $\geq2$ then
\[
A(\cal O_\lambda) \simeq (\bbZ/2\bbZ)^{a_\lambda-1}.
\]
Otherwise, if every odd part of $\lambda$
has multiplicity $1$ then
$A(\cal O_\lambda)$ is a certain
central extension of $\bbZ/2\bbZ$ by
$(\bbZ/2\bbZ)^{a_\lambda-1}$
whose precise description can be found
in \cite[\S 14.3]{lusztig84b}.
The group $\Spin_{2n+1}$ has a cuspidal sheaf
if and only if $2n+1$ is either square or triangular.
In the square case, where $2n+1$ has a partition as
$2n + 1 = 1 + 3 + 5 + \cdots$, the sheaf is supported on
\[
\cal O_{1,3,5,\dots}
\]
and has trivial central character.
Similarly, in the triangular case,
the sheaf is supported on either%
\footnote{Say $n = d(d-1)/2$. 
A short calculation shows that the first case happens
when $d$ is even and the second when $d$ is odd.}
\[
\cal O_{1,5,9,\dots}
\qquad\tn{or}\qquad
\cal O_{3,7,11,\dots},
\]
has nontrivial central character,
and corresponds to a representation of $A(\cal O)$
of dimension $2^{\lfloor(a_\lambda-1)/2\rfloor}$.
So if $2n+1$ is both triangular and square,
the smallest such $n$ being $n=612$ and $706860$,
then there are two cuspidal sheaves.

\subsubsection{Type \texorpdfstring{$C_n$}{Cₙ}}
Nilpotent orbits are in bijection with partitions of~$2n$
for which every odd part has even multiplicity.
Let $b_\lambda$ denote the number of
distinct even (nonzero) parts of~$\lambda$.
Then
\[
A(\cal O_\lambda) \simeq (\bbZ/2\bbZ)^{b_\lambda}.
\]
The group $\Sp_{2n}$ admits a cuspidal perverse sheaf
if and only if $n$ is triangular,
in which case $2n$ has a partition as
$2n=2+4+6+\cdots$ and the sheaf is supported on the orbit
\[
\cal O_{2,4,6,\dots}.
\]
The central character is trivial when $n$ is even
and nontrivial when $n$ is odd.

\subsubsection{Type \texorpdfstring{$D_n$}{Dₙ}}
Call a partition~$\lambda$ \mathdef{very even}
if every part of~$\lambda$ is even with even multiplicity.
Nilpotent orbits are in bijection with partitions of~$2n$
for which every even part has even multiplicity,
except that very even partitions give rise
to two partitions, denoted $\cal O_\lambda^\tn{I}$
and $\cal O_\lambda^\tn{II}$.
The description of $A(\cal O_\lambda)$
is exactly the same as in type~$B_n$,
except that when every odd part of $\lambda$
has multiplicity $1$ we must write $\max(0,a_\lambda-1)$
instead of $a_\lambda-1$ to account
for the possibility that $a_\lambda=0$, 
that is, there are no odd parts.
The description of cuspidal sheaves
is nearly the same as in type~$B_n$.
The group $\Spin_{2n}$ has a cuspidal sheaf
if and only if $2n$ is either square or triangular.
In the square case,
where $2n$ has a partition as $2n = 1 + 3 + 5 + \cdots$,
the sheaf is supported on
\[
\cal O_{1,3,5,\dots}
\]
and its central character is trivial when $n/2$ is even
and is nontrivial but descends to $Z(\SO_{2n})$ when $n/2$ is odd.
If $2n$ is triangular, the orbit
\[
\cal O_{1,5,9,\dots}
\qquad\tn{or}\qquad
\cal O_{3,7,11,\dots}
\]
supports two cuspidal sheaves,
one for each of the remaining central characters,
whose corresponding representation of $A(\cal O)$
has dimension $2^{\lfloor(a_\lambda-1)/2\rfloor}$.
So if $2n$ is both triangular and square,
the smallest such $n$ being $n=18$ and~$20808$,
then there are two cuspidal sheaves.

\subsection{Exceptional groups}
The nilpotent orbits of the exceptional groups
are parameterized not by partitions
but by Bala--Carter diagrams,
certain labeled Dynkin diagrams
whose precise definition is not relevant for us.
The exceptional nilpotent orbits,
together with their dimensions and $A(\cal O)$'s,
are tabulated in \cite[8.4]{collingwood_mcgovern93}.
Without reproducing these tables,
we give a flavor for them by recording
the possible values of $A(\cal O)$
and the number of times each occurs.

Our tables also indicate the central characters.
When $G$ is of type $E_8$, $F_4$, and~$G_2$
the center is trivial and there is nothing to say.
In types $E_6$ and~$E_7$, the centers are nontrivial and
isomorphic to $\mu_3$ and~$\mu_2$ respectively.
So $Z(G)$ either maps injectively into~$A(\cal O)$
or has trivial image.
In the latter case we write $\mu_3$
or~$\mu_2$ (respectively) in the table
to indicate the image of the center.

As for cuspidality, all simply-connected exceptional groups
have exactly one cuspidal perverse sheaf
except for~$E_6$, which has two.
In every case the nilpotent orbit~$\cal O_\tn{cusp}$
that supports a cuspidal perverse sheaf
can be described in the following way:
$\cal O_\tn{cusp}$ is distinguished and of minimal dimension
among the distinguished orbits.
In addition, when $G$ does not have type~$E_6$,
the orbit $\cal O_\tn{cusp}$ is the unique nilpotent orbit
for which $A(\cal O_\tn{cusp})$ has maximum cardinality.

The group $A(\cal O_\tn{cusp})$ is 
the product of $Z(G)$ and a symmetric group on~$\leq5$ letters.
The representations of $A(\cal O_\tn{cusp})$
that correspond to cuspidal perverse sheaves
are the product of the sign character of the symmetric group
and a central character, which in types $E_6$ and~$E_7$
must be nontrivial.

The cuspidal support map is no more complicated
than it must be, in light of the fact that
it preserves the central character.
When the center is trivial, in types $E_8$, $F_4$, and~$G_2$,
every noncuspidal perverse sheaf lies in the principal series,
and when the center is nontrivial, in types $E_6$ and~$E_7$,
the cuspidal support of a noncuspidal perverse sheaf
is determined by its central character.
In the table we use boldface to indicate the orbit
that supports a cuspidal sheaf.
The reader may find it illuminating to compare our table
with the tables in the appendix to \cite{achar_henderson_juteau_riche17b}
for $\ell\gg0$.

\begin{table}[h]
\caption{$\#\{\cal O : A(\cal O) = A\}$ for $G$ exceptional}
\begin{tabular}{c@{\hskip 3em}ccccc@{\hskip 2em}cc@{\hskip 2em}ccc} \toprule
& \multicolumn{10}{c}{$A$} \\ \cmidrule{2-11}
$G$ & triv & $S_2$ & $S_3$ & $\bm{S_4}$ & $\bm{S_5}$
    & $\mu_3$ & $\bm{\mu_3\times S_2}$
    & $\mu_2$ & $\mu_2\times S_2$ & $\bm{\mu_2\times S_3}$ \\ \midrule
$G_2$ & 4  &    & \textbf{1} \\
$F_4$ & 9  & 6  &   & \textbf{1} & \\
$E_6$ & 13 & 1  & 1 &   &   & 5 & \textbf{1} \\
$E_7$ & 17 & 8  & 1 &   &   &   &   & 15 & 3 & \textbf{1} \\
$E_8$ & 38 & 25 & 6 &   & \textbf{1} \\ \bottomrule
\end{tabular}
\end{table}

\subsubsection{Type \texorpdfstring{$E_6$}{E₆}}
A perverse sheaf of central character~$\chi$
has cuspidal support $[T,1,\overline\bbQ_\ell]_G$ if $\chi=1$
and is either cuspidal or has cuspidal support
$[L,\cal O_\tn{reg},\cal E_\chi]$ if~$\chi\neq1$.
Here $L$ is a Levi subgroup with $L_\tn{sc} = \SL_3^2$ and $\pi_0(Z_L)\simeq\mu_3$
(see \cite[15.1(b)]{lusztig84b} for a more precise description),
$\cal O_\tn{reg}$ is its regular nilpotent orbit,
and $\cal E_\chi$ is the unique cuspidal perverse sheaf
of~$L$ with central character~$\chi$.

\subsubsection{Type \texorpdfstring{$E_7$}{E₇}}
A perverse sheaf of central character~$\chi$
has cuspidal support $[T,1,\overline\bbQ_\ell]_G$ if $\chi=1$
and is either cuspidal or has cuspidal support
$[L,\cal O_\tn{reg},\cal E_\chi]$ if~$\chi\neq1$.
Here $L$ is a Levi subgroup with $L_\tn{sc} = \SL_2^3$ and $\pi_0(Z_L)\simeq\mu_2$
(see \cite[15.2(b)]{lusztig84b} for a more precise description),
$\cal O_\tn{reg}$ is its regular nilpotent orbit,
and $\cal E_\chi$ is the unique cuspidal perverse sheaf
of~$L$ with central character~$\chi$.

\subsubsection{Types \texorpdfstring{$E_8$}{E₈},
\texorpdfstring{$F_4$}{F₄}, and
\texorpdfstring{$G_2$}{G₂}}
Every noncuspidal perverse sheaf has cuspidal support
$[T,1,\overline\bbQ_\ell]_G$.

\bibliography{rigid-galois-bernstein.bib}
\bibliographystyle{amsalpha}

\end{document}